\pdfoutput=1
\documentclass[11pt]{article}

\usepackage{preamble}

\begin{document}

\title{Structurable equivalence relations and $\@L_{\omega_1\omega}$ interpretations}
\author{Rishi Banerjee and Ruiyuan Chen}
\date{}
\maketitle

\begin{abstract}
We show that the category of countable Borel equivalence relations (CBERs) is dually equivalent to the category of countable $\@L_{\omega_1\omega}$ theories which admit a one-sorted interpretation of a particular theory we call $\@T_\LN \sqcup \@T_\sep$ that witnesses embeddability into $2^\#N$ and the Lusin--Novikov uniformization theorem.
This allows problems about Borel combinatorial structures on CBERs to be translated into syntactic definability problems in $\@L_{\omega_1\omega}$, modulo the extra structure provided by $\@T_\LN \sqcup \@T_\sep$, thereby formalizing a folklore intuition in locally countable Borel combinatorics.
We illustrate this with a catalogue of the precise interpretability relations between several standard classes of structures commonly used in Borel combinatorics, such as Feldman--Moore $\omega$-colorings and the Slaman--Steel marker lemma.
We also generalize this correspondence to locally countable Borel groupoids and theories interpreting $\@T_\LN$, which admit a characterization analogous to that of Hjorth--Kechris for essentially countable isomorphism relations.
\let\thefootnote=\relax
\end{abstract}

\tableofcontents

\section{Introduction}
\label{sec:intro}

This paper is a contribution to the global theory of locally countable Borel combinatorics and equivalence relations and connections with countable model theory.
A \defn{countable Borel equivalence relation (CBER)} $E$ on a standard Borel space $X$ is a Borel equivalence relation $E \subseteq X^2$ such that all $E$-classes are countable.
Over the past thirty years, CBERs have been widely studied in descriptive set theory and adjacent areas, such as group theory, ergodic theory, and combinatorics.
See \cite{Kcber} for a comprehensive survey.

\subsection{Structurability}

An important aspect of the theory of CBERs consists of analyzing the Borel ways of assigning a combinatorial structure to each $E$-class.
For instance, a CBER is called \defn{smooth} iff there is a Borel way to pick a single point from each class.
A CBER is called \defn{hyperfinite} iff it is the orbit equivalence relation of a Borel action $\#Z \actson X$ (see \cite{DJK}); this amounts to a \emph{transitive} $\#Z$-action on each class.
Smooth and hyperfinite CBERs are the ``simplest'' CBERs.
On the other hand, every class of every CBER may be given the structure of a locally finite connected graph in a Borel way \cite{JKL}; and every such Borel locally finite graph admits a Borel $\omega$-coloring \cite{KST}.

In general, given a CBER $(X,E)$, a Borel family of first-order structures (group actions, graphs, etc.)\ on each $E$-class $C \in X/E$ is called a \defn{structuring} of $E$; if one exists, $E$ is called \defn{structurable} (by group actions, graphs, etc.).
See \cref{def:str} for the precise definition, including the meaning of ``Borel family''.
Structurability provides a global framework for comparing and organizing ``all (locally countable) Borel combinatorics problems''; see \cite{Kturing}, \cite{JKL}, \cite{CK}.

It is a basic folklore intuition that doing locally countable Borel combinatorics often amounts to ``countable combinatorics, done canonically/in a uniform Borel way''; see e.g., \cite[\S1.E]{CPTT}.
For instance, not every CBER is hyperfinite, even though obviously every nonempty countable set admits a transitive $\#Z$-action, because there is no ``canonical'' way to choose said action.
On the other hand, after picking a single distinguished point, it becomes possible to define a ``sufficiently canonical'' such action (or indeed any other structure, subject to cardinality constraints).

In this paper, extending the work of \cite{CK}, we develop a correspondence between the theory of structurability of CBERs and countable model theory that allows such intuitions to be made precise.
A weak form of this correspondence concerns existence of structurings:

\begin{theorem}[see \cref{thm:str-impl-interp}]
\label{intro:thm:str-impl-interp}
Let $\@K, \@K'$ be two classes of countable first-order structures, axiomatized by respective countable $\@L_{\omega_1\omega}$ theories $\@T, \@T'$.
The following are equivalent:
\begin{enumerate}[label=(\roman*)]
\item \label{intro:thm:str-impl-interp:str}
Every CBER structurable by structures in $\@K$ is also structurable by structures in $\@K'$.
\item \label{intro:thm:str-impl-interp:interp}
We may uniformly define a countable $\@K'$-structure $\@M'$ from every countable $\@K$-structure $\@M$, equipped with two additional pieces of structure, namely
\begin{enumerate*}[label=(\arabic*)]
\item  a countable family $(U_i)_{i \in \#N}$ of subsets separating distinct points and
\item  a countable family $(f_i)_{i \in \#N}$ of unary functions whose graphs cover $M^2$.
\end{enumerate*}
These definitions are expressible by $\@L_{\omega_1\omega}$ formulas independent of $\@M$.
\end{enumerate}
\end{theorem}

Here $\@L_{\omega_1\omega}$ is the countably infinitary first-order logic, with countable Boolean connectives $\bigwedge, \bigvee, \neg$ as well as the usual (finitary) quantifiers $\forall, \exists$; see \cite{Marker} and \cref{sec:lo1o} below.
For a class of structures to be $\@L_{\omega_1\omega}$-definable means equivalently, in semantic terms, that said class is Borel in the standard Borel space of structures on a fixed countable set, by the Lopez-Escobar \cref{thm:lopez-escobar}.
Given a countable $\@L_{\omega_1\omega}$ theory $\@T$, by a \defn{$\@T$-structuring} we mean a structuring by models of $\@T$.

\begin{example}
\label{ex:finsub-pt}
The class $\@K'$ of sets equipped with a single distinguished point may be axiomatized by the empty (always true) theory $\@T_\pt = \emptyset$ in the first-order language $\{c\}$ with a single constant.
The class $\@K$ of sets equipped with a finite nonempty subset $D$ may be axiomatized by the theory
\begin{align*}
\@T_\finsub :=
    \bigvee_{n \ge 1} \exists z_0, \dotsc, z_{n-1}\, \forall w\, ( D(w) <-> \bigvee_{i < n} (w = z_i)),
\end{align*}
where now $D$ is a unary relation.
Clearly, given a set $M$ equipped with a finite nonempty set $D \subseteq M$, there is in general no way to canonically pick a single distinguished element $c \in M$; thus we cannot uniformly define a model $(M,c) \in \@K'$ of $\@T_\pt$ from a model $(M,D) \in \@K$ of $\@T_\finsub$.
(Formally, this is because $(M,D)$ may have automorphisms with no fixed points.)

However, it is well-known that given a CBER $E$, if there is a Borel way to pick a finite nonempty set in each class, i.e., $E$ is $\@T_\finsub$-structurable, then in fact $E$ is smooth, i.e., $\@T_\pt$-structurable.
This is because, on the one hand, given a finite nonempty set $D$ as well as a linear order $<$, we may canonically pick the least element $c$ of $D$:
\begin{align}
\label{ex:finsub-pt:pt}
c = x  \coloniff  D(x) \wedge \forall y\, ((y < x) -> \neg D(y)).
\end{align}
In other words, we may define a model of $\@T_\pt$ given a model of $\@T_\finsub \sqcup \@T_\LO$, where $\@T_\LO$ is the theory of linear orders.
On the other hand, every CBER $E \subseteq X^2$ is $\@T_\LO$-structurable, because e.g., we may assume the standard Borel space $X$ to be embedded in $2^\#N$, and then take the lexicographical order on each $E$-class.
The embedding $X -> 2^\#N$ is given by the indicator functions of countably many $U_i \subseteq X$ separating points (the subbasic clopens); the lexicographical order is then given by
\begin{align}
\label{ex:finsub-pt:LO}
x < y  \coloniff  \bigvee_{i \in \#N} (\neg U_i(x) \wedge U_i(y) \wedge \bigwedge_{j < i} (U_j(x) <-> U_j(y))).
\end{align}
By substituting this latter formula \cref{ex:finsub-pt:LO} for the symbol $<$ in \cref{ex:finsub-pt:pt}, we obtain a formula $\phi(x)$ uniformly defining a single point from a finite nonempty set $D$ as well as a countable separating family $(U_i)_{i \in \#N}$ as in \cref{intro:thm:str-impl-interp}\cref{intro:thm:str-impl-interp:interp}, thereby witnessing that $\@T_\finsub$-structurability implies $\@T_\pt$-structurability.
\end{example}

Based on anecdotal evidence, \cref{intro:thm:str-impl-interp} seems to be more-or-less folklore among the Borel combinatorics community; at least, it seems widely believed that some statement along these lines ought to be true.
However, to our knowledge, such a formal statement or proof has not appeared before in the literature.

\subsection{Interpretations}
\label{sec:intro-interp}

Our main results in this paper show that there is a ``complete correspondence'' between semantic concepts in the realm of CBERs and structurability, and ``syntactic definitions'' of such concepts in the realm of the infinitary logic $\@L_{\omega_1\omega}$ and countable model theory.
From this correspondence, \cref{intro:thm:str-impl-interp} and similar results will follow as immediate consequences.

\Cref{intro:thm:str-impl-interp}\cref{intro:thm:str-impl-interp:interp} says that ``in every model of $\@T$, we may uniformly define a model of $\@T'$''.
In general, given two $\@L_{\omega_1\omega}$ theories $\@T, \@T'$ (in two different languages, which we assume to be relational for simplicity), by an \defn{interpretation} $\alpha : \@T' -> \@T$, we mean a family of $\@L_{\omega_1\omega}$ formulas $\alpha(R)(x_0,\dotsc,x_{n-1})$ in the language of $\@T$, for each $n$-ary relation symbol $R$ in the language of $\@T'$, such that these formulas define a model of $\@T'$ in every model of $\@T$.
See \cref{sec:interp} for details and alternate formulations.%
\footnote{\label{ft:interp}%
In this paper, we focus on \emph{one-sorted} interpretations (sometimes called \emph{definitions} in model theory), which define a uniform construction of a $\@T'$-model on the same underlying set of a $\@T$-model.
The more general model-theoretic interpretations into imaginaries will play a minor role only; see \cref{sec:intro-esscount}.}
For example, the above formula \cref{ex:finsub-pt:pt} yields an interpretation
\begin{equation}
\label{eq:finsub-pt:interp-pt}
\@T_\pt --> \@T_\finsub \sqcup \@T_\LO.
\end{equation}

Now \cref{intro:thm:str-impl-interp}\cref{intro:thm:str-impl-interp:interp} does not assert the existence of an interpretation $\@T' -> \@T$, but rather the existence of an interpretation $\@T' -> \@T \sqcup \@T_\LN \sqcup \@T_\sep$ into the theory $\@T$ expanded with two additional pieces of structure.
Namely, $\@T_\sep$ is the \defn{theory of countable separating families}, in the language with countably many unary relations $\{U_i\}_{i \in \#N}$ asserting that they separate distinct points (see \cref{def:Tsep}); and $\@T_\LN$ is the \defn{theory of Lusin--Novikov functions}, in the language with countably many unary functions $\{f_i\}_{i \in \#N}$ asserting that they cover all pairs (see \cref{def:TLN}).
Intuitively, the theory $\@T_\LN \sqcup \@T_\sep$ describes combinatorial structure which is available ``for free'' on every CBER, by the Lusin--Novikov theorem and standard Borelness of the underlying space.
For example, the above formula \cref{ex:finsub-pt:LO} yields an interpretation
\begin{equation}
\label{eq:finsub-pt:interp-LO}
\@T_\LO --> \@T_\sep,
\end{equation}
hence also $\@T_\finsub \sqcup \@T_\LO -> \@T_\finsub \sqcup \@T_\sep$, which may be composed with the above interpretation \cref{eq:finsub-pt:interp-pt} to yield an interpretation $\@T_\pt -> \@T_\finsub \sqcup \@T_\sep$ describing how a Borel transversal of a CBER may be constructed from a Borel choice of a finite subset in each class (and a countable separating family, which is available for free on a CBER).

We thus have a category of countable $\@L_{\omega_1\omega}$ theories, with interpretations between them as morphisms.
Within this category, there is a distinguished theory $\@T_\LN \sqcup \@T_\sep$, such that the theories admitting an interpretation from it capture structurability of CBERs.
We show that in fact, such theories are ``equivalent'' to CBERs:

\begin{theorem}[see \cref{thm:scott-LNsep}]
\label{intro:thm:scott-LNsep}
We have a dual equivalence of categories
\begin{align*}
\{\text{CBERs, class-bijective homomorphisms}\} &\simeq \{\text{countable $\@L_{\omega_1\omega}$ theories interpreting $\@T_\LN \sqcup \@T_\sep$}\} \\
E &|-> \@T_E.
\end{align*}
Moreover, $\@T$-structurings of a CBER $E$ are in natural bijection with interpretations $\@T -> \@T_E$.
\end{theorem}

This is illustrated in \cref{fig:cber-interp}.
Here, a \defn{(Borel) class-bijective homomorphism} $f : (X,E) -> (Y,F)$ between two CBERs is a Borel map $f : X -> Y$ whose restriction to each $E$-class is a bijection with some $F$-class.
The theory $\@T_E$ corresponding to a CBER $E$ is the \defn{Scott theory} of $E$, which is defined uniquely up to bi-interpretability by declaring its models on any countable set $Y$ to be bijections between $Y$ and some $E$-class; see \cref{def:scott}.

The Scott theory was introduced in \cite{CK} (there called \emph{Scott sentences}), which also established much of its significance with respect to structurability.
In particular, the construction of $\@T_E$, the last assertion of \cref{intro:thm:scott-LNsep}, as well as the full faithfulness of the functor $E |-> \@T_E$ (i.e., that it is a bijection on morphisms) appeared in that paper in some form.
However, that paper was focused on the CBER-theoretic aspects of structurability, and as such, discussed the model-theoretic significance of the Scott theory in somewhat roundabout terms; in particular, only a terse overview of $\@L_{\omega_1\omega}$ interpretations was given, in the unpublished Appendix~B of the arXiv preprint of \cite{CK}.
(Henceforth, all references to \cite[Appendix~B]{CK} are to the arXiv preprint.)

In this paper, we give a detailed, self-contained introduction to $\@L_{\omega_1\omega}$ interpretations in \cref{sec:interp}, with connections to and equivalent formulations in terms of Borel $S_{\le\infty}$-spaces of countable models, Borel spaces of $\@L_{\omega_1\omega}$ types, and Boolean $\sigma$-algebras of formulas.
After reviewing several equivalent definitions of structurability in \cref{sec:str}, we then give in \cref{sec:scott} an alternate treatment of Scott theories $\@T_E$, based on the conceptual definition up to bi-interpretability given above, rather than an explicit coding as in \cite{CK}; and we show how all of the fundamental properties of Scott theories have simple derivations from this perspective.

In \cref{sec:TLN}, we prove \cref{intro:thm:scott-LNsep} (\cref{thm:scott-LNsep}).
The bulk of the proof is devoted to characterizing the essential image of the functor $E |-> \@T_E$, i.e., showing that every theory $\@T$ interpreting $\@T_\LN \sqcup \@T_\sep$ is bi-interpretable with the Scott theory of a CBER $E$.
This is the main new technical contribution of \cref{thm:scott-LNsep}, and answers a question from \cite[sentence after~B.4]{CK}.
To briefly outline the proof: the CBER $E$ is constructed on the space $\@S_1(\@T)$ of $\@L_{\omega_1\omega}$ 1-types of $\@T$ (equivalently, the space of isomorphism types of pointed models), two such types being equivalent iff they are realized in isomorphic models.
The interpretation of $\@T_\LN \sqcup \@T_\sep$ is used to verify that $E$ is indeed a CBER with Scott theory $\@T_E \cong \@T$, using a quantifier-elimination argument.


In \cref{sec:str-interp}, we show how \cref{intro:thm:str-impl-interp}, as well as syntactic formulations of all other results and constructions in structurability theory, easily follow from \cref{intro:thm:scott-LNsep}.
We stress that this latter result, about individual \emph{structurings} and \emph{interpretations}, contains much more information than the former result about \emph{structurability} and \emph{interpretability}.

For instance, whereas a typical problem concerning CBERs asks whether a CBER may be structured by e.g., graphs of a certain kind, a typical problem in Borel combinatorics asks whether e.g., a given structuring $\@M$ of a CBER by graphs admits a Borel coloring.
This may be abstractly phrased in terms of an \emph{expansion} of $\@M$ to a structuring by graphs equipped with a coloring.
Using \cref{intro:thm:scott-LNsep}, we may likewise show that all such ``Borel expandability problems'' amount to ``countable expandability, done uniformly'', thereby generalizing \cref{intro:thm:str-impl-interp}:

\begin{theorem}[see \cref{thm:str-expan-interp}]
\label{intro:thm:str-expan-interp}
Let $\@T \subseteq \@T'$ be countable $\@L_{\omega_1\omega}$ theories in languages $\@L \subseteq \@L'$.
The following are equivalent:
\begin{enumerate}[label=(\roman*)]
\item
Every $\@T$-structuring of a CBER admits an expansion to a $\@T'$-structuring.
\item
There exists an interpretation $\@T' -> \@T \sqcup \@T_\LN \sqcup \@T_\sep$ whose restriction to the language $\@L$ is ($\@T$-provably equivalent to) the identity.
\end{enumerate}
\end{theorem}

We illustrate this in the aforementioned example of graph colorings in \cref{thm:kst-lfcolor}.

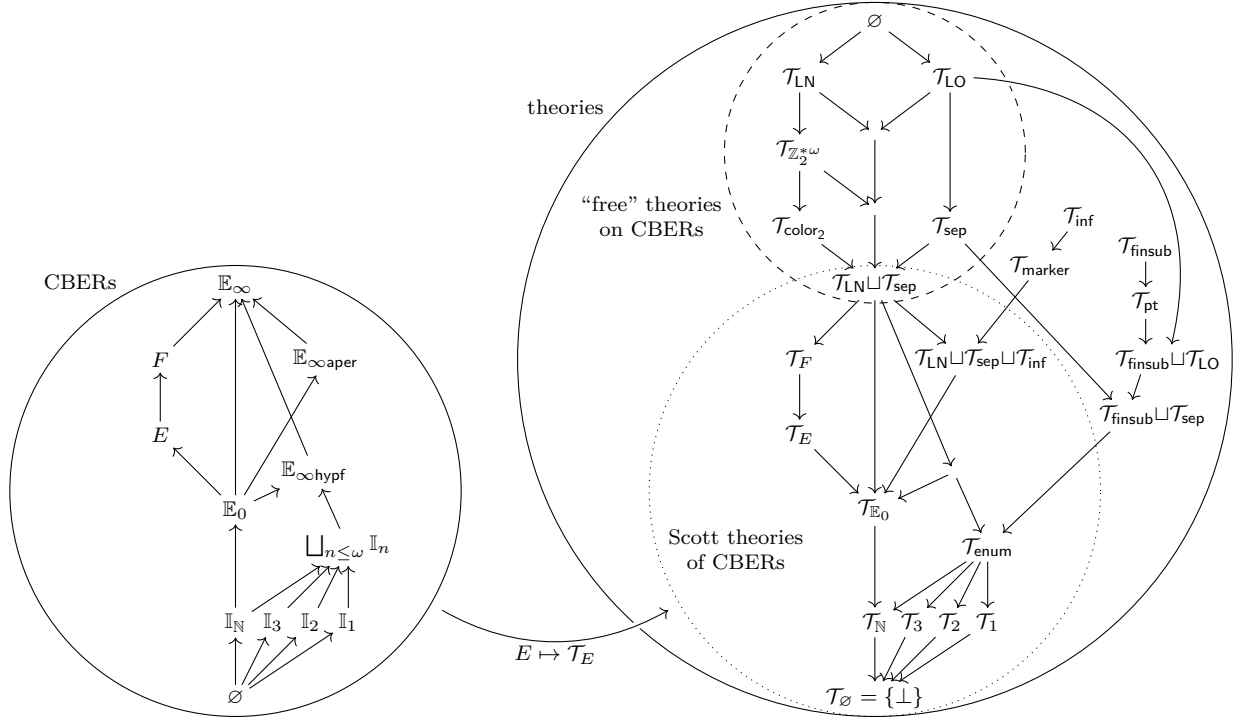
\begin{figure}
\centering
\newcommand*\CBERup{1}  
\begin{tikzpicture}[
    every node/.append style={font=\scriptsize, inner sep=2pt},
]
\begin{scope}[y={(0,\CBERup)}]
\node[CBER] (Eempty)
    {\emptyset};
\node[CBER] (EIN) at ($(Eempty) + (0,1)$)
    {\#I_\#N}
    edge[<-] (Eempty);
\node[CBER] (EI3) at ($(EIN) + (0.5,0)$)
    {\#I_3}
    edge[<-] (Eempty);
\node[CBER] (EI2) at ($(EI3) + (0.5,0)$)
    {\#I_2}
    edge[<-] (Eempty);
\node[CBER] (EI1) at ($(EI2) + (0.5,0)$)
    {\#I_1}
    edge[<-] (Eempty);
\node[CBER] (EI) at ($(EI1) + (0,1)$)
    {\bigsqcup_{n \le \omega} \#I_n}
    edge[<-] (EIN)
    edge[<-] (EI3)
    edge[<-] (EI2)
    edge[<-] (EI1);
\node[CBER] (E0) at ($(EIN) + (0,1.5)$)
    {\#E_0}
    edge[<-] (EIN);
\node[CBER] (Einf) at ($(E0) + (0,3)$)
    {\#E_\infty}
    edge[<-] (E0);
\node[CBER] (Einfh) at ($(EI) + (-0.45,1)$)
    {\#E_{\infty\!{hypf}}}
    edge[<-] (E0)
    edge[<-] (EI)
    edge[->] (Einf);
\node[CBER] (E) at ($(E0) + (-1,1)$)
    {E}
    edge[<-] (E0);
\node[CBER] (F) at ($(E0) + (-1,2)$)
    {F}
    edge[<-] (E)
    edge[->] (Einf);
\node[CBER] (Einfa) at ($(Einf) + (1.2,-1)$)
    {\#E_{\infty\!{aper}}}
    edge[<-] (E0)
    edge[->] (Einf);
\node(CBERs)[circle, draw, inner sep=0pt, minimum width=6cm, label={-\CBERup*240:CBERs}] at ($(Eempty)!0.5!(Einf)$) {};
\end{scope}

\begin{scope}[y={(0,-\CBERup)}]
\node[theory] (Tfalse) at ($(Eempty) + (8.5,0)$)
    {\@T_\emptyset = \{\bot\}};
\node[theory] (TN) at ($(Tfalse) - (Eempty) + (EIN)$)
    {\@T_\#N}
    edge[->] (Tfalse);
\node[theory] (T3) at ($(Tfalse) - (Eempty) + (EI3)$)
    {\@T_3}
    edge[->] (Tfalse);
\node[theory] (T2) at ($(Tfalse) - (Eempty) + (EI2)$)
    {\@T_2}
    edge[->] (Tfalse);
\node[theory] (T1) at ($(Tfalse) - (Eempty) + (EI1)$)
    {\@T_1}
    edge[->] (Tfalse);
\node[theory] (Tenum) at ($(Tfalse) - (Eempty) + (EI)$)
    {\@T_\enum}
    edge[->] (TN)
    edge[->] (T3)
    edge[->] (T2)
    edge[->] (T1);
\node[theory] (TE0) at ($(Tfalse) - (Eempty) + (E0)$)
    {\@T_{\#E_0}}
    edge[->] (TN);
\node[theory] (TLNsep) at ($(Tfalse) - (Eempty) + (Einf)$)
    {\@T_\LN {\sqcup} \@T_\sep}
    edge[->] (TE0);
\node[theory] (TEinfh) at ($(Tfalse) - (Eempty) + (Einfh)$)
    {}
    edge[->] (TE0)
    edge[->] (Tenum)
    edge[<-] (TLNsep);
\node[theory] (TE) at ($(Tfalse) - (Eempty) + (E)$)
    {\@T_E}
    edge[->] (TE0);
\node[theory] (TF) at ($(Tfalse) - (Eempty) + (F)$)
    {\@T_F}
    edge[->] (TE)
    edge[<-] (TLNsep);
\node[theory] (TLNsepinf) at ($(Tfalse) - (Eempty) + (Einfa)$)
    {\@T_\LN {\sqcup} \@T_\sep {\sqcup} \mathrlap{\@T_\inf}}
    edge[->] (TE0)
    edge[<-] (TLNsep);
\node(scotts)[circle, draw, dotted, inner sep=0pt, minimum width=6cm, label={[anchor=west,align=center,xshift={.5ex}]{-\CBERup*165}:Scott theories\\of CBERs}] at ($(Tfalse)!0.5!(TLNsep)$) {};
\node[theory] (TLNinv) at ($(TLNsep) + (-1, -.75)$)
    {\@T_{\coloreq2}}
    edge[->] (TLNsep);
\node[theory] (TFMLO) at ($(TLNsep) + (0, -1)$)
    {}
    edge[->] (TLNsep);
\node[theory] (TFMinv) at ($(TLNinv) - (TLNsep) + (TFMLO)$)
    {\@T_{\#Z_2^{*\omega}}}
    edge[->] (TLNinv)
    edge[->] (TFMLO);
\node[theory] (TLNLO) at ($(TFMLO) - (TLNsep) + (TFMLO)$)
    {}
    edge[->] (TFMLO);
\node[theory] (TLN) at ($(TFMinv) - (TFMLO) + (TLNLO)$)
    {\@T_\LN}
    edge[->] (TFMinv)
    edge[->] (TLNLO);
\node(Tsep) at ($(TLNsep) + (1, -.75)$)
    {$\@T_\sep$}
    edge[->] (TLNsep);
\node(TLO) at (Tsep |- TLN)
    {$\@T_\LO$}
    edge[->] (TLNLO)
    edge[->] (Tsep);
\node(Tempty) at ($(TLN) + (TLO) - (TLNLO)$)
    {$\emptyset$}
    edge[->] (TLN)
    edge[->] (TLO);
\node[circle, draw, dashed, inner sep=0pt, minimum width=4cm, label={[anchor=east,align=center,inner sep=1ex,xshift={-\CBERup*.25ex}]{-\CBERup*155}:``free'' theories\\on CBERs}] at ($(TLNsep)!0.5!(Tempty)$) {};
\node[theory] (Tmarker) at ($(TLNsepinf) + (1, -1.25)$)
    {\@T_\marker}
    edge[->] (TLNsepinf);
\node[theory] (Tinf) at ($(Tmarker) + (0.52, -0.65)$)
    {\@T_\inf}
    edge[->] (Tmarker);
\node[theory] (Tseppt) at ($(Tenum) + (1.85, -1.75)$)
    {\qquad\; \@T_\finsub {\sqcup} \@T_\sep}
    edge[->] (Tenum)
    edge[<-] (Tsep);
\node[theory] (TLOpt) at ($(Tseppt) + (0.25, -0.75)$)
    {\qquad \@T_\finsub {\sqcup} \@T_\LO}
    edge[->] (Tseppt);
\path (TLOpt.{\CBERup*30})
    edge[<-, cross, bend left={-\CBERup*50}, looseness=1.2] (TLO);
\node[theory] (Tpt) at ($(TLOpt) + (0., -0.75)$)
    {\@T_\pt}
    edge[->] (TLOpt);
\node[theory] (Tfinsub) at ($(Tpt) + (0., -0.75)$)
    {\@T_\finsub}
    edge[->] (Tpt);
\node(theories)[circle, draw, inner sep=0pt, minimum width=9.5cm, label={[anchor=east,align=center,xshift=-1ex]{-\CBERup*225}:theories}] at ($(Tfalse)!0.5!(Tempty)$) {};
\end{scope}

\draw[shorten=1ex, ->, cross] (CBERs) to[bend left={-\CBERup*30}, "$E \mapsto \@T_E$"{anchor={\CBERup*90}}] (scotts);
\end{tikzpicture}
\caption{Hierarchies of CBERs under class-bijective homomorphisms, and $\@L_{\omega_1\omega}$ theories under interpretations, with embedding from former to latter via Scott theories.}
\label{fig:cber-interp}
\end{figure}

\subsection{Comparing the strengths of Borel combinatorial structures}

By \cref{intro:thm:str-impl-interp}, every positive result relating two classes of CBERs defined by structurability may in principle be formulated entirely in terms of countable structures, without mentioning Borelness or CBERs.
In fact, in practice many well-known constructions in Borel combinatorics already amount to $\@L_{\omega_1\omega}$ interpretations between the relevant theories.
We illustrate this in \cref{sec:examples} with several examples along the lines of \cref{ex:finsub-pt}, reformulating the standard proofs of e.g.,
the Feldman--Moore theorem~\cite{FM} (\cref{thm:FM}),
its generalization by Kechris--Miller~\cite{KM} to an $\omega$-coloring of the intersection graph on finite subsets (\cref{thm:KM-color}),
and the Slaman--Steel marker lemma~\cite{SlStr} (\cref{thm:marker})
into explicit $\@L_{\omega_1\omega}$ interpretations.

Along the way, we also make a catalogue of \emph{non-interpretability} between the relevant theories.
When the theories in question, such as the three aforementioned, describe structure which is available ``for free'' on every CBER, this means that the construction of said structure on a CBER must necessarily make use of the ``free'' structure on CBERs provided by $\@T_\LN \sqcup \@T_\sep$.
By comparing different theories in this way, we may make precise the idea that certain results in Borel combinatorics are ``stronger'', or ``more specific to CBERs'', than others.

To illustrate this, consider the Feldman--Moore theorem~\cite{FM}, which states that every CBER is the orbit equivalence relation of a countable Borel group action.
However, the original proof showed the \emph{a priori} stronger statement that every CBER is a countable union of Borel involutions, in which form the result is often used in Borel combinatorics (see e.g., \cite[3.4]{Kcber}).
We give the following analysis of these and other forms of the Feldman--Moore theorem, which we state here in somewhat informal terms; see the quoted results for more precise statements:

\begin{theorem}[see \cref{thm:FM,thm:FMbij,ex:Zline-cex,ex:Ztrans-cex,ex:color3-cex}]
\label{intro:thm:FM}
Let $E \subseteq X^2$ be a CBER.
Consider the following 3 forms of the Feldman--Moore theorem:
\begin{enumerate}[label=(\roman*)]
\item \label{intro:thm:FM:FMbij}
$E$ is induced by a Borel action of the free group $\#F_\omega \actson X$.
\item \label{intro:thm:FM:FMinv}
$E$ is induced by a Borel action of a group generated by involutions $\#Z_2^{*\omega} \actson X$.
\item \label{intro:thm:FM:LNinv}
$E = \bigcup_i f_i$ can be covered by the graphs of countably many Borel involutions $f_i : X -> X$.
\end{enumerate}
Then:
\begin{enumerate}[label=(\alph*)]
\item \label{intro:thm:FM:FMbij-LN}
It is possible to prove \cref{intro:thm:FM:FMbij} using only the Lusin--Novikov theorem applied to $E$.
\item \label{intro:thm:FM:FMinv-FMbij}
It is not possible to prove \cref{intro:thm:FM:FMinv} or \cref{intro:thm:FM:LNinv} using only Lusin--Novikov (equivalently, using \cref{intro:thm:FM:FMbij}), without also using a countable separating family of Borel subsets $U_i \subseteq X$.
\item \label{intro:thm:FM:LNinv-FMinv}
It is also not possible to prove \cref{intro:thm:FM:LNinv} using only \cref{intro:thm:FM:FMinv}.
Thus, the 3 versions of Feldman--Moore are ``strictly increasing'' in strength.
\end{enumerate}
\end{theorem}

Here, for example, \cref{intro:thm:FM:LNinv-FMinv} means precisely that the theory $\@T_{\#Z_2^{*\omega}}$ of transitive $\#Z_2^{*\omega}$-actions (structurability by which formalizes the statement of \cref{intro:thm:FM:FMinv}) does not interpret the theory $\@T_\coloreq2$ of edge $\omega$-colorings of the complete graph (which formalizes \cref{intro:thm:FM:LNinv}).
Note that, as \cref{intro:thm:FM:FMbij}--\cref{intro:thm:FM:LNinv} are all true (by the Feldman--Moore theorem), by \cref{intro:thm:str-impl-interp}, the three respective theories do interpret each other when combined with the theory $\@T_\sep$.
Nonetheless, the above result allows us to make precise the idea that e.g., \cref{intro:thm:FM:LNinv} is a ``strictly stronger'' statement of the Feldman--Moore theorem than \cref{intro:thm:FM:FMbij}.

\Cref{fig:cber-interp} shows various other theories describing common Borel combinatorial structures, including several available ``for free'' on a CBER, whose strengths we may distinguish in a similar manner; a more detailed diagram is shown in \cref{fig:cber-interp-big}.

\subsection{Locally countable Borel groupoids}
\label{sec:intro-gpd}

A \defn{groupoid} $(X,G)$ is, in short, a category with invertible morphisms, consisting of a space of objects $X$ and a space of morphisms $G$, together with operations (maps) specifying the domain and codomain of a morphism, as well as composition, identity, and inverse morphisms.
A groupoid can be viewed as a common generalization of a group (when $X = 1$), an equivalence relation (when $G \subseteq X^2$), and a group action (when equipped with a ``fibration'' to the acting group $\Gamma$).

A \defn{locally countable Borel groupoid} is a groupoid in which $X, G$ are standard Borel spaces, all of the groupoid operations are Borel maps, and each object is the domain of only countably many morphisms; see \cref{def:gpd}.
Such groupoids generalize CBERs and countable Borel group actions, and admit a largely analogous structural theory, with natural analogues of concepts such as amenability (see e.g., \cite{TW}), treeability, graphings and cost (see \cite{Ueda}, \cite{Alvarez}, \cite{Carderi}).
Nonetheless, to our knowledge, a general theory of structurability for groupoids, akin to \cite{CK} for CBERs, has not been considered before in the literature.

In \cref{sec:gpd}, we develop the beginnings of such a theory.
Given a locally countable Borel groupoid $(X,G)$ and an $\@L_{\omega_1\omega}$ theory $\@T$, by a \defn{$\@T$-structuring} of $G$, we mean a Borel right-translation-invariant family $(\@M_x)_{x \in X}$ of models of $\@T$ on the set of morphisms with domain $x$, for each $x \in X$; see \cref{def:gpd-str}.
This is the natural generalization of structurability for CBERs, with the expected examples: e.g, a groupoid will be ``treeable'' (structurable by trees) iff it is the free groupoid generated by a Borel multigraph (so in particular, a ``treeable'' group in this sense is a free group; see \cref{ex:gpd-str-tree}).
We prove the following analogue of \cref{intro:thm:scott-LNsep}:

\begin{theorem}[see \cref{thm:gpd-scott-LN,thm:gpd-scott-equiv}]
\label{intro:thm:gpd-scott-LN}
We have a dual equivalence of categories
\begin{align*}
\{\text{locally countable Borel groupoids}\} &\simeq \{\text{countable $\@L_{\omega_1\omega}$ theories interpreting $\@T_\LN$}\} \\
G &|-> \@T_G.
\end{align*}
Moreover, $\@T$-structurings of a locally countable Borel groupoid $G$ are in natural bijection with interpretations $\@T -> \@T_G$.
\end{theorem}

Here the morphisms in the left category are Borel \defn{fibrations} $G -> H$ between groupoids, i.e., functors exhibiting $G$ as the action groupoid of a Borel action of $H$; see \cref{ex:gpd-action}.
The \defn{Scott theory} $\@T_G$ associated to a locally countable Borel groupoid $G$ is defined by declaring its models on a set $Y$ to be fibrations $\#I_Y -> G$ from the indiscrete groupoid $\#I_Y = Y^2$; see \cref{def:gpd-scott}.

Just as \cref{intro:thm:str-impl-interp} shows that $\@T_\LN \sqcup \@T_\sep$ describes all of the combinatorial structure available ``for free'' on a CBER, it follows analogously from \cref{intro:thm:gpd-scott-LN} that $\@T_\LN$ describes precisely the structure available ``for free'' on a locally countable Borel groupoid.
In other words, every Borel combinatorial construction on groupoids may in principle be done ``countably, in a uniform way'', where ``uniform'' here has a more restrictive meaning than in the case of CBERs (namely, we still have access ``for free'' to Lusin--Novikov functions $\{f_i\}_{i \in \#N}$, but no longer have a countable separating family $\{U_i\}_{i \in \#N}$).
See \cref{thm:gpd-str-impl-interp} for the precise statement.

For instance, recall that by \cref{intro:thm:FM}, the Feldman--Moore theorem in the weak form of ``generated by a countable group action'' can be proved using only Lusin--Novikov, but the stronger form ``generated by involutions'' cannot.
It follows that we have a weak ``Feldman--Moore theorem for groupoids'', which says that every locally countable Borel groupoid $(X,G)$ admits an action of a countable group $\Gamma \actson G$ on the space of morphisms, which commutes with right multiplication and whose orbits are precisely all morphisms with a fixed domain.
Equivalently, such data correspond to a countable subgroup of the \defn{Borel full group} $[G]$ which covers $G$; in this form, the ``Feldman--Moore theorem for groupoids'' was first stated (without proof) in \cite{TW}.
See \cref{ex:gpd-str-FM}.

The route we take to \cref{intro:thm:gpd-scott-LN} (\cref{thm:gpd-scott-LN}) is somewhat more involved than in the case of CBERs, but we believe it provides some insight that may be of independent interest.
Given a theory $\@T$ interpreting $\@T_\LN$, we consider the sequence of $\@L_{\omega_1\omega}$ type spaces $\@S_n(\@T)$ for each $n \ge 1$.
We show that these form a \defn{locally countable standard Borel simplicial set}, satisfying the so-called \emph{Grothendieck--Segal condition}, that ensures it is the simplicial nerve of a groupoid.
We will review the relevant concepts from simplicial homotopy theory in \cref{sec:simplicial}.

\subsection{Essential countability and imaginaries}
\label{sec:intro-esscount}

In \cref{sec:esscount}, we briefly discuss the significance of Scott theories of CBERs and locally countable Borel groupoids within the broader context of countable model theory.

In \cite{HK}, Hjorth--Kechris gave a model-theoretic characterization of those countable $\@L_{\omega_1\omega}$ theories $\@T$ whose isomorphism relation $\cong$ is Borel bireducible to a CBER (such $\cong$ are called \defn{essentially countable}).
Namely, every countable model $\@M$ of $\@T$ must be $\aleph_0$-categorical with respect to a countable fragment $\@F$ of $\@L_{\omega_1\omega}$ after fixing finitely many constants $\vec{a} \in M^n$.

In \cite{Cscc,Cgpd}, following earlier work by Harrison-Trainor--Miller--Montalbán \cite{HMM}, Chen showed that \emph{arbitrary} countable $\@L_{\omega_1\omega}$ theories are ``equivalent'' to their Borel groupoids of countable models.
More precisely, we have an equivalence of 2-categories
\begin{align*}
\yesnumber
\label{intro:eq:scc}
\{\text{theories}\} &\simeq \{\text{groupoids}\} \\
\@T &|-> \{\text{countable models of $\@T$}\},
\end{align*}
where the left-hand side consists of countable $\@L_{\omega_1\omega}$ theories and interpretations between them as morphisms (and definable isomorphisms as 2-cells), and the right-hand side consists of standard Borel groupoids obeying a certain topological condition, with Borel functors between them (and Borel natural isomorphisms between those).
Here the \emph{interpretations} between theories $\@T' -> \@T$ are allowed to be more general than those considered above (see \cref{sec:intro-interp}): they describe a uniform construction of a model of $\@T'$ from each model $\@M$ of $\@T$, but living not necessarily on the same underlying set, but rather on a set ``uniformly defined'' from $\@M$.
This ``uniformly defined set'' is the model-theoretic notion of an \defn{imaginary} (see \cite[Ch.~7]{Hodges}), adapted to the $\@L_{\omega_1\omega}$ setting (namely, a definable quotient of a countable disjoint union of definable sets).

In light of \cref{intro:eq:scc}, the aforementioned Hjorth--Kechris result may be regarded as characterizing those theories $\@T$ whose groupoid of models is Borel equivalent to a groupoid each of whose connected components has only countably many objects (namely, isomorphic models).
Note that such a groupoid still need not be a locally countable Borel groupoid; it will be one iff each model moreover has only countably many automorphisms.
From this and \cref{thm:gpd-scott-LN}, we derive:

\begin{theorem}[see \cref{thm:esscaut-LN}]
The above correspondence \cref{intro:eq:scc} restricts to an equivalence between theories obeying the following equivalent conditions, and locally countable Borel groupoids:
\begin{enumerate}[label=(\roman*)]
\item
Every countable model has only countably many automorphisms, and the isomorphism relation $\cong$ is essentially countable.
\item
There is a countable fragment $\@F$ of $\@L_{\omega_1\omega}$, such that every countable model becomes rigid and $\@F$-categorical after fixing finitely many constants $\vec{a}$.
\end{enumerate}
The inverse is given (up to bi-interpretability) by the Scott theory $G |-> \@T_G$ of a groupoid.

Under this correspondence, the CBERs correspond to theories all of whose models are rigid.
\end{theorem}

We expect that $\@L_{\omega_1\omega}$ imaginaries and the more general notion of interpretations into them ought to play a broader role in Borel structurability theory as well.
Indeed, many well-known ``uniform Borel combinatorial constructions'' produce structures on a new underlying set derived from an existing structure.
For example, this is the case for arguments showing closure of a given class of structurable CBERs under Borel reducibility (e.g., the treeable CBERs \cite[3.3]{JKL}).
It would be interesting to have analogues of \cref{intro:thm:str-impl-interp,intro:thm:scott-LNsep} for Borel reducibility, inclusions of subequivalence relations, class-injective homomorphisms (see \cite[\S5]{CK}), etc.; such analogues would presumably involve interpretations into imaginaries.
We leave such questions for future work.

\paragraph*{Acknowledgments}

We would like to thank Alexander Kechris and Anush Tserunyan for encouraging us to pursue this line of work, Robin Tucker-Drob for pointing out the connection to \cite{TW}, and Matthew Harrison-Trainor and Alexander Kechris for helpful comments and suggestions.
R.C.\ was supported by NSF grant DMS-2224709.

\section{Preliminaries}
\label{sec:prelim}

\subsection{Descriptive set theory}
\label{sec:dst}

For background on the descriptive set-theoretic notions we will use, see \cite{Kcdst}, \cite{GaoIDST}.

By a \defn{Borel space}, we will mean what is often called a \defn{measurable space}: a set $X$ equipped with a $\sigma$-algebra $\@B(X) \subseteq \@P(X)$.
A \defn{standard Borel space} is one which is Borel isomorphic to a Borel subspace of the Cantor space $2^\#N$ with the product topology (or equivalently, any Polish space).
A \defn{Borel map} between Borel spaces is one such that the preimage of every Borel set is Borel.

For a group action $G \actson X$ on a standard Borel space $X$, the \defn{orbit equivalence relation} $\#E_G^X$ is
\begin{align*}
x \mathrel{\#E_G^X} y  \coloniff  \exists g \in G\, (g \cdot x = y).
\end{align*}
All of the Borel spaces we consider will be quotients of standard Borel spaces $X$ by orbit equivalence relations $E = \#E_G^X \subseteq X^2$ of Borel actions of Polish groups $G \actson X$.
Recall that the quotient Borel structure $\@B(X/E)$ consists of the subsets whose lift in $X$ is $E$-invariant Borel.

An important fact about maps between such quotient spaces is

\begin{proposition}[{see \cite[5.4.6, 5.2.3, 5.4.12]{GaoIDST}, \cite[2.8]{Cgpd}}]
\label{thm:polgrpact}
Let $X/E, Y/F$ be two quotient spaces of Polish group actions, and $f : X -> Y$ be a Borel map which descends to an injection $X/E `-> Y/F$ (a \defn{Borel reduction}).
\begin{enumerate}[label=(\alph*)]
\item
Then $f$ descends to a Borel embedding $X/E `-> Y/F$.
In other words, each $E$-invariant Borel set in $X$ is the preimage of an $F$-invariant Borel set in $Y$.
\item
If moreover $F \subseteq Y^2$ is Borel, then the image of $X/E$ in $Y/F$ is Borel, i.e., the image saturation $[f(X)]_F \subseteq Y$ is Borel.
Moreover, the inverse map between the quotient spaces $[f(X)]_F/F \cong X/E$ has a Borel lift $[f(X)]_F -> X$.
\end{enumerate}
\end{proposition}

An equivalence relation $E \subseteq X^2$ on a standard Borel space $X$ is \defn{smooth} if $X/E$ has countably many Borel sets which separate points, or equivalently, $X/E$ is Borel isomorphic to an analytic set in a standard Borel space.
A \defn{Borel transversal} of $E$ is a Borel subset $D \subseteq X$ containing exactly one point from each $E$-class.
If $E$ is an analytic equivalence relation, and has a Borel transversal $D$, then the quotient space $X/E$ is standard Borel, being Borel isomorphic to $D$.
If $X/E$ is standard Borel, then clearly $E$ is smooth.
Neither of these two implications reverses in general; however, they do reverse if $E$ is induced by a Polish group action, as a consequence of \cref{thm:polgrpact}.

Another special fact about orbit equivalence relations of Polish group actions $E \subseteq X^2$ is that each orbit is Borel, i.e., each point in $X/E$ is Borel; see \cite[15.14]{Kcdst}, \cite[3.3.2]{GaoIDST}.

\medskip
A \defn{countable Borel equivalence relation (CBER)} is a Borel equivalence relation $E \subseteq X^2$ all of whose classes $[x]_E \subseteq X$ are countable.
Equivalently, they are precisely orbit equivalence relations of Borel actions of countable groups, by the \defn{Feldman--Moore theorem} \cite[Th.~1]{FM} (see also \cite[7.1.4]{GaoIDST}).
See \cite{Kcber} for a comprehensive survey of the theory of CBERs.

The fundamental tool underlying the theory of CBERs is

\begin{theorem}[{Lusin--Novikov; see \cite[18.10]{Kcdst}}]
\label{thm:lusin-novikov}
Let $f : X -> Y$ be a countable-to-1 Borel map between standard Borel spaces.
Then $X = \bigcup_i X_i$ for countably many Borel sets $X_i \subseteq X$ such that $f$ restricts to a bijection $f : X_i \cong f(X)$.
\end{theorem}

\begin{corollary}
\label{thm:lusin-novikov-cber}
Let $E \subseteq X^2$ be a CBER on a standard Borel space $X$.
\begin{enumerate}[label=(\alph*)]
\item \label{thm:lusin-novikov-cber:LN}
There are countably many Borel functions $f_i : X -> X$ such that $E = \bigcup_i f_i$.
\item \label{thm:lusin-novikov-cber:enum}
For each $N \le \omega$, $X_N := \{x \in X \mid \abs{[x]_E} = N\}$ is Borel; and there is a Borel family of enumerations $(g_x : \abs{[x]_E} \cong [x]_E)$, meaning that $(x,i) |-> g_x(i) : \bigsqcup_{N \le \omega} (X_N \times N) -> X$ is Borel.
\end{enumerate}
\end{corollary}

Here \cref{thm:lusin-novikov-cber:LN} follows from Lusin--Novikov applied to the first projection $E -> X$, while \cref{thm:lusin-novikov-cber:enum} is an easy modification; see \cite[18.15]{Kcdst}.
The Lusin--Novikov theorem, especially \cref{thm:lusin-novikov-cber:LN} above, will also play a central role in our analysis of Borel structurability in this paper; see \cref{sec:TLN}.

\subsection{Countable first-order logic $\mathcal{L}_{\omega_1\omega}$}
\label{sec:lo1o}

Let $\@L$ be a first-order language.
The infinitary logic $\@L_{\omega_1\omega}$ is the extension of the usual finitary first-order logic by allowing countable conjunctions $\bigwedge$ and disjunctions $\bigvee$ in formulas (in addition to the usual connectives $->, \neg$, quantifiers $\forall, \exists$, and equality $=$).
See \cite[Ch.~11]{GaoIDST}, \cite{Marker}.

Throughout, when we say ``formula'', ``theory'', etc., we mean in $\@L_{\omega_1\omega}$.
By convention, we require each $\@L_{\omega_1\omega}$ formula $\phi(x_0,\dotsc,x_{n-1})$ to contain only finitely many free variables $x_0,\dotsc,x_{n-1}$, and we identify formulas up to renaming of bound variables.

There is a completeness theorem for $\@L_{\omega_1\omega}$ (for various deductive systems \cite{Karp}, \cite{LE}), which say that an $\@L_{\omega_1\omega}$ sentence has a proof from a \emph{countable} theory iff it is true in every \emph{countable} model.
Thus, we may interchangeably speak of provable or semantic truth.

Unless otherwise specified, by default we assume languages $\@L$, theories $\@T$, and models $\@M$ to be countable.
Note that we might as well then replace theories $\@T$ with single sentences $\bigwedge \@T$; however, we find it more convenient for expositional purposes to allow theories with multiple axioms.

\begin{convention}
\label{cvt:relational}
We will assume that formally, all languages are relational.
In examples, we will often want to consider languages with function symbols; formally, we understand this to mean that each $n$-ary function symbol $f$ is replaced with an $(n+1)$-ary relation symbol $F$, with axioms added to the theory under consideration that say that $F$ is the graph of a function.
Thus for example, a constant (nullary function) $c$ is replaced with a unary relation $C$ subject to the axiom $\exists! x\, C(x)$.
\end{convention}

\begin{notation}
For a language $\@L$, we write $\@L_{\omega_1\omega}$ for the set of all $\@L_{\omega_1\omega}$ formulas.
Occasionally, we will also write $\@L_{\omega_1\omega}^n$ for the set of formulas $\phi(x_0,\dotsc,x_{n-1})$ with $n$ free variables, and similarly $\@L^n$ for the set of $n$-ary relation symbols in $\@L$.
\end{notation}

\subsection{Spaces of countable models}

\begin{definition}
\label{def:mod}
Fix a countable language $\@L$, assumed to be relational as per \cref{cvt:relational}.
For a countable set $Y$, the \defn{space of $\@L$-structures on $Y$} is the product standard Borel space
\begin{equation*}
\Mod_Y(\@L) := \prod_{\substack{n \in \#N \\ R \in \@L^n}} 2^{Y^n} = \set[\big]{\@M = (R^\@M)_{R \in \@L}}{\forall R \in \@L^n\, (R^\@M \subseteq Y^n)}.
\end{equation*}
More generally, for an $\@L_{\omega_1\omega}$ theory $\@T$,
\begin{equation*}
\Mod_Y(\@T) = \Mod_Y(\@L, \@T) := \set[\big]{\@M \in \Mod_Y(\@L)}{\@M |= \@T}.
\end{equation*}
By an easy induction, for each $\@L_{\omega_1\omega}$ formula $\phi(x_0,\dotsc,x_{n-1})$, $\Mod_Y^n(\phi) \subseteq \Mod_Y^n(\@L)$ is Borel.
Thus, for a countable theory $\@T$, $\Mod_Y(\@T) \subseteq \Mod_Y(\@L)$ is Borel.
See \cite[\S16.C]{Kcdst}, \cite[Ch.~11]{GaoIDST}.

Note that the space of models of $\@T$ expanded with $n$ constants is then
\begin{equation*}
\Mod_Y^n(\@T) := \Mod_Y(\@L \sqcup \{x_0, \dotsc, x_{n-1}\}, \@T) \cong \Mod_Y(\@T) \times Y^n = \set{(\@M, \vec{a})}{\@M \in \Mod_Y(\@T),\, \vec{a} \in Y^n}.
\end{equation*}
We call $(\@M, \vec{a}) \in \Mod_Y^n(\@T)$ an \defn{$n$-pointed model}, or a \defn{pointed model} for $n = 1$.
\end{definition}

\begin{definition}
For two sets $Y, Z$, let $\Sym(Y, Z) \subseteq Z^Y$ denote the set of all bijections $g : Y \cong Z$; when $Y = Z$, this is the \defn{symmetric group} $\Sym(Y) := \Sym(Y, Y)$ (often denoted $S_Y$).
For $Y, Z$ countable, $\Sym(Y, Z)$ is a $G_\delta$ subset of $Z^Y$ (with $Z$ discrete); hence $\Sym(Y)$ is a Polish group.

The \defn{logic action} on the space of structures is given by, for any $Y, Z$:%
\footnote{Formally, this is an action of the groupoid of countable sets, with hom-sets $\Sym(Y, Z)$, on the bundle of $\Mod_Y(\@L)$'s.}
\begin{align*}
\Sym(Y, Z) \times \Mod_Y(\@L) &--> \Mod_Z(\@L) \\
(g : Y \cong Z,\, \@M = (R^\@M)_{R \in \@L}) &|--> g \cdot \@M := (g(R^\@M))_{R \in \@L}.
\end{align*}
That is, for a structure $\@M$ on $Y$, $g \cdot \@M$ is the unique structure on $Z$ such that $g : \@M \cong g \cdot \@M$.

Thus, a class $\@K$ of countable structures, on various underlying sets $Y$, is invariant under the logic action (between those $\Mod_Y(\@L)$'s) iff it is closed under isomorphisms.
\end{definition}

Clearly, for any theory $\@T$, $\Mod_Y(\@T) \subseteq \Mod_Y(\@L)$ is closed under isomorphisms.
Conversely:

\begin{theorem}[{Lopez-Escobar \cite{LE}; see also \cite[16.8]{Kcdst}, \cite[11.3.6]{GaoIDST}}]
\label{thm:lopez-escobar}
For any countable set $Y$, a subset $\@K \subseteq \Mod_Y(\@L)$ is axiomatizable by an $\@L_{\omega_1\omega}$ sentence (equivalently, a countable theory) iff it is Borel and closed under isomorphisms.
\end{theorem}

A special case is \defn{Scott's isomorphism theorem}: the isomorphism class of a single countable structure, being Borel (as an orbit of a Polish group action), is axiomatized by a sentence.

We also note the following easy generalized formulations of Lopez-Escobar's theorem:

\begin{corollary}
For any countable set $Y$ and $n \in \#N$, a subset $\@K \subseteq \Mod_Y^n(\@L)$ is axiomatizable by a formula $\phi(x_0,\dotsc,x_{n-1})$ iff it is Borel and closed under isomorphisms (of $n$-pointed models).
\end{corollary}

\begin{corollary}
\label{thm:lopez-escobar-global}
A class $\@K$ of countable models on \emph{all} countable underlying sets $Y$ is axiomatizable by a sentence (or countable theory) iff it is closed under isomorphisms and $\@K \cap \Mod_Y(\@L) \subseteq \Mod_Y(\@L)$ is Borel for each countable set $Y$.
Similarly for a class $\@K$ of $n$-pointed models.
\end{corollary}

This follows from applying \cref{thm:lopez-escobar} to $Y = 0, 1, 2, \dotsc, \#N$ to get sentences $\phi_Y$, and then combining them into a sentence $\bigvee_{Y \le \omega} (\text{``there are exactly $Y$ many elements''} \wedge \phi_Y)$.

\subsection{Spaces of $\mathcal{L}_{\omega_1\omega}$ types}

\begin{definition}
\label{def:types}
Let $(\@L,\@T)$ be a countable theory, and let $n \in \#N$.
For $\@M |= \@T$ and $\vec{a} \in M^n$, the \defn{($\@L_{\omega_1\omega}$) type of $\vec{a}$} will mean the complete $\@L_{\omega_1\omega}$ theory of the $n$-pointed structure%
\footnote{We will never consider finitary first-order (i.e., $\@L_{\omega\omega}$) types in this paper.}
\begin{align*}
\tp(\@M, \vec{a}) := \set*{\phi \in \@L_{\omega_1\omega}^n}{\phi^\@M(\vec{a})}.
\end{align*}
The \defn{Borel space of ($\@L_{\omega_1\omega}$) $n$-types} $\@S_n(\@T)$ is the set of all $\tp(\@M, \vec{a})$ over all $n$-tuples $\vec{a}$ in all countable models $\@M$ of $\@T$,
equipped with the Borel $\sigma$-algebra
\begin{align*}
\@B(\@S_n(\@T)) := \set*{\sqsqbr\phi}{\phi \in \@L_{\omega_1\omega}^n},
    \quad \text{where } \sqsqbr\phi := \set*{p \in \@S_n(\@T)}{\phi \in p}.
\end{align*}
Note that by the completeness theorem, we thus have an isomorphism of $\sigma$-algebras
\begin{align*}
\sqsqbr{-} : \@L_{\omega_1\omega}^n/\@T := \{\text{formulas $\phi(x_0,\dotsc,x_{n-1})$ mod $\@T$-provable equivalence}\} \cong \@B(\@S_n(\@T));
\end{align*}
an alternate definition of $\@S_n(\@T)$ is the space of principal $\sigma$-ultrafilters in $\@L_{\omega_1\omega}^n/\@T$.%
\footnote{In this paper, we only consider principal types realized in countable models, since they are all we will need.
In general $\@L_{\omega_1\omega}$ theories, there is a more general notion of type, namely an arbitrary $\sigma$-ultrafilter of formulas, which is not necessarily realizable in any model due to the failure of the completeness theorem for uncountable theories.}
\end{definition}


\begin{corollary}[of Lopez-Escobar]
\label{thm:mod-types}
The maps $\tp$ induce an isomorphism of Borel spaces
\begin{align*}
\tp : (\bigsqcup_Y \Mod_Y^n(\@T))/({\cong}) &\cong \@S_n(\@T).
\end{align*}
\end{corollary}

Here, the disjoint union may be taken over any representative family of sets $Y$ of each countable cardinality, e.g., the ordinals $Y \le \omega$.
The left-hand side is then the disjoint union of the quotient Borel spaces $\Mod_Y^n(\@T)/\Sym(Y)$, for $Y \le \omega$, each of which is a quotient of a Polish group action.

Recall (\cref{sec:dst}) that such a quotient is standard Borel iff the $\sigma$-algebra of invariant Borel sets in $\Mod_Y^n(\@T)$ is countably generated.
By Lopez-Escobar, this means that there are countably many formulas $\phi(x_0,\dotsc,x_{n-1})$ separating all non-isomorphic $n$-pointed models.

\begin{definition}
For a set $\@F \subseteq \@L_{\omega_1\omega}$ of formulas, an $n$-pointed model $(\@M, \vec{a})$ is \defn{$\@F$-categorical} if it is the only countable model of $\tp(\@M, \vec{a}) \cap \@F$, up to isomorphism.
\end{definition}

\begin{corollary}[{of Lopez-Escobar; see \cite[11.5.8]{GaoIDST}}]
\label{thm:mod-smooth-cat}
For any countable $\@L_{\omega_1\omega}$ theory $\@T$ and $n \in \#N$, the following are equivalent:
\begin{enumerate}[label=(\roman*)]
\item  The space of $n$-types $\@S_n(\@T)$ is standard Borel.
\item  Isomorphism of $n$-pointed models of $\@T$ is a smooth equivalence relation.
\item  There are countably many formulas $\@F \subseteq \@L_{\omega_1\omega}^n$ such that every $n$-pointed model is $\@F$-categorical.
\end{enumerate}
\end{corollary}

\begin{theorem}[Becker--Kechris]
\label{thm:types-borel}
For a countable theory $\@T$, isomorphism ${\cong} \subseteq \Mod_Y^n(\@T)^2$ of $n$-pointed models is a Borel equivalence relation iff isomorphism ${\cong} \subseteq \Mod_Y^{n+1}(\@T)^2$ of $(n+1)$-pointed models is.
\end{theorem}
\begin{proof}
The forward direction is \cite[7.1.4]{BK}.
For the converse, when $n \ge 1$, the logic action on $\Mod_Y(\@T) \times Y^n$ diagonally embeds into the logic action on $\Mod_Y(\@T) \times Y^{n+1}$ (by duplicating the last coordinate, say).
When $n = 0$,
use instead that
\begin{equation*}
\hspace{-3ex}
\@M \cong \@N \iff \exists a \in M,\, b \in N\, ((\@M,a) \cong (\@N,b))
                    \OR (M = N = \emptyset \AND \forall P \in \@L^0\, (P^\@M \iff P^\@N)).
\hspace{-2ex}
\qedhere
\end{equation*}
\end{proof}

\begin{corollary}
\label{thm:types-smooth}
If $\@S_n(\@T)$ is standard Borel, then so is $\@S_{n+1}(\@T)$.
\end{corollary}
\begin{proof}
Suppose isomorphism ${\cong} \subseteq \Mod_Y^n(\@T)^2$ of $n$-pointed models is smooth; we show that isomorphism ${\cong} \subseteq \Mod_Y^{n+1}(\@T)^2$ of $(n+1)$-pointed models is smooth.
We may assume $Y \le \omega$.
Let $D \subseteq \Mod_Y^n(\@T)$ be a Borel transversal of ${\cong} \subseteq \Mod_Y^n(\@T)^2$.
Then
\begin{align*}
\set[\big]{(\@M,a_0,\dotsc,a_n)}
          {(\@M,a_0,\dotsc,a_{n-1}) \in D \AND \forall b_n < a_n ((\@M,a_0,\dotsc,a_{n-1},b_n) \not\cong (\@M,a_0,\dotsc,a_n))}
\end{align*}
is a Borel transversal of ${\cong} \subseteq \Mod_Y^{n+1}(\@T)^2$.
\end{proof}

\begin{remark}
\label{rmk:types-simplicial}
The Borel spaces $(\@S_n(\@T))_{n < \omega}$ are related by Borel \defn{projection maps}
\begin{equation*}
\dotsb --->{\partial_3} \@S_2(\@T) --->{\partial_2} \@S_1(\@T) --->{\partial_1} \@S_0(\@T)
\end{equation*}
where, for an $n$-type $p \in \@S_n(\@T)$, $\partial_n(p)$ consists of the formulas $\phi(x_0,\dotsc,x_{n-2})$ in $n-1$ variables which belong to $p$ when regarded as having an extra variable $x_{n-1}$; thus
\begin{align}
\label{eq:types-subst}
\partial_n^{-1}(\sqsqbr{\phi(x_0,\dotsc,x_{n-2})}) = \sqsqbr{\phi(x_0,\dotsc,x_{n-1})}.
\end{align}
In other words, via \cref{thm:mod-types}, $\partial_n$ descends from the coordinate projection
\begin{align*}
\pi_n : \bigsqcup_Y \Mod_Y^n(\@T) --> \bigsqcup_Y \Mod_Y^{n-1}(\@T)
\end{align*}
omitting the last coordinate, which is clearly countable-to-1 and maps (invariant) Borel sets to (invariant) Borel sets; thus $\partial_n$ is also countable-to-1 and maps Borel sets onto Borel sets.
Namely,
\begin{align}
\label{eq:types-exists}
\partial_n(\sqsqbr{\phi(x_0,\dotsc,x_{n-1})}) = \sqsqbr{\exists x_{n-1}\, \phi(x_0,\dotsc,x_{n-1})}.
\end{align}

In fact there are other canonical ``projection maps'' between the $\@S_n(\@T)$'s.
Instead of regarding an $(n-1)$-ary formula as $n$-ary via the inclusion of variables $n-1 `-> n$, we may consider a variable substitution from $m$ variables to $n$ variables given by an arbitrary map $s : m -> n$, which induces
\begin{gather*}
\begin{aligned}
\partial_s : \@S_n(\@T) &--> \@S_m(\@T) \\
p &|--> \{\phi(x_0,\dotsc,x_{m-1}) \mid \phi(x_{s(0)}, \dotsc, x_{s(m-1)}) \in p\}
\end{aligned} \\
\shortintertext{with}
\partial_s^{-1}(\sqsqbr{\phi(x_0,\dotsc,x_{m-1})}) = \sqsqbr{\phi(x_{s(0)}, \dotsc, x_{s(m-1)})}.
\end{gather*}
This descends from the corresponding map $\pi_s : \bigsqcup_Y \Mod_Y^n(\@T) -> \bigsqcup_Y \Mod_Y^m(\@T)$, and is again countable-to-1 and maps Borel sets to Borel sets.
When $s$ is the inclusion $n-1 `-> n$, $\partial_s$ is the coordinate projection $\partial_n$ from above.
When $s : n ->> n-1$ is instead the surjection collapsing $n-1 |-> n-2$ (and fixing other elements), then $\partial_s$ is the diagonal embedding duplicating the last two coordinates, and image under $\partial_s$ is given by
\begin{align}
\label{eq:types-eq}
\partial_s(\sqsqbr{\phi(x_0,\dotsc,x_{n-2})}) = \sqsqbr{\phi(x_0,\dotsc,x_{n-2}) \wedge (x_{n-2} = x_{n-1})}.
\end{align}
For a general $s : m -> n$ between finite ordinals $m, n < \omega$, we may write $s$ as a composition of such inclusions and surjections as well as bijective permutations; thus $\partial_s$ may be understood as a combination of the above cases, using the obvious fact
\begin{align*}
\partial_{s \circ t} = \partial_t \circ \partial_s
\quad \text{(also $\partial_\id = \id$)}.
\end{align*}

(Category-theoretically, we get that $n |-> \@S_n(\@T)$ forms a contravariant functor from the category of finite ordinals to the category of Borel spaces and countable-to-1 Borel maps with Borel images.
Such a functor is sometimes called an \emph{augmented symmetric simplicial set} in relation to homotopy theory; we will make use of this connection in \cref{sec:gpd} below.
The family $(\@S_n(\@T))_n$ is also dual to the Lindenbaum--Tarski hyperdoctrine from \cite[\S B.4]{CK}; see \cref{rmk:interp-hyperdoctrine,rmk:interp-types}.)
\end{remark}

\subsection{Interpretations in $\mathcal{L}_{\omega_1\omega}$}
\label{sec:interp}

\begin{definition}
\label{def:interp}
Let $\@L, \@L'$ be languages, assumed to be relational as per \cref{cvt:relational}.

An \defn{interpretation} $\alpha : \@L -> \@L'$ is a map taking each $n$-ary relation $R \in \@L$ to an $n$-ary formula $\alpha(R)(x_0,\dotsc,x_{n-1}) \in \@L'_{\omega_1\omega}$.\cref{ft:interp}
Such $\alpha$ provides a syntactic recipe for constructing $\@L$-structures from $\@L'$-structures: given an $\@L'$-structure $\@M$, its \defn{$\alpha$-reduct} is the $\@L$-structure $\alpha^*(\@M)$ given by
\begin{equation*}
R^{\alpha^*(\@M)} := \alpha(R)^\@M
\end{equation*}
for each relation symbol $R \in \@L$.
(The name ``reduct'' comes from the case when $\alpha$ is the inclusion of a sublanguage $\@L \subseteq \@L'$.)
When the underlying set $Y$ of $\@M$ is countable, this yields a map
\begin{align*}
\alpha^* = \alpha^*_Y : \Mod_Y(\@L') &--> \Mod_Y(\@L)
\end{align*}
which is easily seen to be Borel and equivariant under the logic action:
\begin{align*}
\alpha^*_Z(g \cdot \@M) = g \cdot \alpha^*_Y(\@M)
\end{align*}
for any $\@M \in \Mod_Y(\@L')$ and bijection $g \in \Sym(Y, Z)$.

Given $\alpha : \@L -> \@L'$, we extend $\alpha$ to all $\@L$-formulas inductively in the obvious way:
\begin{align*}
\alpha(R(x_0,\dotsc,x_{n-1})) &:= \alpha(R)(x_0,\dotsc,x_{n-1}), \\
\alpha(\bigvee_i \phi_i) &:= \bigvee_i \alpha(\phi_i), \\
\alpha(\exists x\, \phi) &:= \exists x\, \alpha(\phi),
\end{align*}
etc.
Semantically, this means that $\alpha(\phi)$ is an $\@L'$-formula such that in every $\@L'$-structure $\@M$,
\begin{align*}
\alpha(\phi)^\@M = \phi^{\alpha^*(\@M)}.
\end{align*}

Now given an $\@L_{\omega_1\omega}$ theory $\@T$ and $\@L'_{\omega_1\omega}$ theory $\@T'$, we say that $\alpha$ is an \defn{interpretation of $\@T$ in $\@T'$}, written $\alpha : \@T -> \@T'$ or $\alpha : (\@L,\@T) -> (\@L',\@T')$, if it preserves provable truth:
\begin{align*}
\@T' |- \alpha(\@T).
\end{align*}
For $\@T'$ countable, by the completeness theorem this means equivalently that $\alpha^*_Y$ restricts to a map
\begin{align}
\label{eq:interp-mod}
\alpha^*_Y : \Mod_Y(\@T') --> \Mod_Y(\@T)
\end{align}
for every countable set $Y$, which is again Borel and equivariant under the logic action.

Two interpretations $\alpha, \beta : \@T -> \@T'$ are \defn{$\@T'$-provably equivalent} if $\@T' |- \alpha(R) <-> \beta(R)$ for every symbol $R \in \@L$, or equivalently (if $\@T'$ is countable) $\alpha^* = \beta^*$.
We normally identify interpretations up to provable equivalence; that is, an ``interpretation $\alpha : \@T -> \@T'$'' will really mean a $\@T'$-provable equivalence class of interpretations.

If there exists an interpretation $\@T -> \@T'$, we say that $\@T$ is \defn{interpretable} in $\@T'$, or that $\@T'$ \defn{interprets} $\@T$.
There is an obvious composition of interpretations $\@T -> \@T' -> \@T''$; thus interpretability is a preorder on the class of all theories, while interpretations themselves form a category.
If an interpretation has an inverse, we call it a \defn{bi-interpretation}; if there is a bi-interpretation between $\@T, \@T'$, we call them \defn{bi-interpretable}, written $\@T \cong \@T'$.
Note that this is stronger than the existence of interpretations both ways, which we call \defn{mutually interpretable}.
\end{definition}

See the introduction (\cref{ex:finsub-pt}, \cref{eq:finsub-pt:interp-pt,eq:finsub-pt:interp-LO}) for examples of interpretations $\@T_\pt -> \@T_\finsub \sqcup \@T_\LO$ and $\@T_\LO -> \@T_\sep$.
In the rest of this subsection, we give some equivalent reformulations of the definition of interpretation, which are more abstract, but easier to work with when developing the general theory.

\begin{remark}
\label{rmk:interp-hyperdoctrine}
Conceptually, an interpretation $\alpha : \@T -> \@T'$ may be viewed as a family of maps
\begin{align*}
(\alpha_n : \@L_{\omega_1\omega}^n/\@T --> \@L_{\omega_1\omega}'^n/\@T')_{n < \omega},
\end{align*}
where $\@L_{\omega_1\omega}^n/\@T$ denotes $\@T$-provable equivalence classes of $n$-ary $\@L_{\omega_1\omega}$ formulas (\cref{def:types}), which commutes with variable substitutions and the logical operations:
\begin{eqalign*}
\alpha_n([\phi(x_{s(0)},\dotsc,x_{s(m-1)})]) &= [\alpha_m(\phi)(x_{s(0)},\dotsc,x_{s(m-1)})] \quad \text{for $s : m -> n$}, \\
\alpha_2([x_0 = x_1]) &= [x_0 = x_1], \\
\alpha_n([\bigvee_i \phi_i]) &= \bigvee_i \alpha_n([\phi_i]), \\
\alpha_n([\exists x_n\, \phi(x_0,\dotsc,x_n)]) &= \exists x_n\, \alpha_{n+1}([\phi(x_0,\dotsc,x_n)]),
\end{eqalign*}
etc.
This ensures that $\alpha$ is determined by its values on atomic formulas, i.e., the symbols in $\@L$.

In other words, $\alpha$ is a homomorphism of multi-sorted structures $(\@L_{\omega_1\omega}^n/\@T)_{n < \omega} -> (\@L_{\omega_1\omega}'^n/\@T')_{n < \omega}$, consisting of sequences of Boolean $\sigma$-algebras which are furthermore equipped with operations between the various algebras corresponding to variable substitution (for each $s : m -> n$), $\exists$ (from the $(n+1)$th algebra to the $n$th) and a constant ``$x_0=x_1$'' (in the $2$nd algebra).
These structures are the \defn{Lindenbaum--Tarski hyperdoctrines} of the theories, called ``$\omega_1\omega$-Boolean algebras'' in \cite[\S B.4]{CK}; see there as well as \cite{Jacobs} for more information on hyperdoctrines.
\end{remark}

\begin{example}
\label{ex:interp-mod}
For a theory $(\@L,\@T)$, a countable model $\@M \in \Mod_Y(\@T)$ may be regarded as a special case of an interpretation.
Indeed, note that by the usual definition of the \emph{interpretation} of first-order logic in a structure, $\@M$ amounts to a map from formulas to relations
\begin{eqalign*}
\alpha : \@L_{\omega_1\omega}^n &--> 2^{Y^n} \\
\phi &|--> \phi^\@M,
\end{eqalign*}
preserving countable Boolean connectives, variable substitution, quantifiers and $=$, and $\@T$-truth.

Now the Boolean $\sigma$-algebra $2^{Y^n}$ is canonically isomorphic to the algebra of $n$-ary formulas modulo equivalence for the theory $\@T_Y$ of \defn{$Y$-enumerated sets}, with language $\@L_Y$ consisting of constant symbols $c_y$ for each $y \in Y$ and axioms saying that these form a bijection with $Y$:
\begin{align*}
\@T_Y := \brace[\big]{
    \bigwedge_{y \ne z \in Y} (c_y \ne c_z), \;
    \forall x\, \bigvee_{y \in Y} (x = c_y)}.
\end{align*}
Indeed, every $A \subseteq Y^n$ gives a formula $\phi_A(x_0,\dotsc,x_{n-1}) := \bigvee_{\vec{y} \in A} (\vec{x} = c_{\vec{y}})$; and it is easy to see that $\@T_Y$ has quantifier elimination, using the above axioms to rewrite $\exists x\, \phi(x)$ to $\bigvee_{y \in Y} \phi(c_y)$, whence every $n$-ary formula is equivalent modulo $\@T_Y$ to a quantifier-free formula, hence to some $\phi_A$ since the above axioms also easily yield that the formulas $\vec{x} = c_{\vec{y}}$ are the atoms of the $\sigma$-algebra.
It is also easy to see that these isomorphisms $2^{Y^n} \cong (\@L_Y)_{\omega_1\omega}^n/\@T_Y$ preserve substitution, quantifiers, and $=$.

Thus $\alpha$, i.e., the model $\@M$, may be regarded as an interpretation $\alpha : \@T -> \@T_Y$.
Semantically, $\alpha^*$ turns a model $\@N \in \Mod_Z(\@T_Y)$, i.e., a bijection $g_\@N : y |-> c_y^\@N : Y \cong Z$, into the model $\alpha^*(\@N) = g_\@N \cdot \@M \in \Mod_Z(\@T)$.
\end{example}

\begin{remark}
\label{rmk:interp-types}
The Lindenbaum--Tarski hyperdoctrine of a theory $(\@L,\@T)$ is dual to the sequence of type spaces $(\@S_n(\@T))_{n < \omega}$, with the canonical projection maps $\partial_s$ (induced by arbitrary substitutions $s$) between them, from \cref{rmk:types-simplicial}.
Indeed, the algebra $\@L_{\omega_1\omega}^n/\@T$ of $\@T$-equivalence classes of formulas is isomorphic to the Borel $\sigma$-algebra $\@B(\@S_n(\@T))$ by \cref{def:types}; and by \cref{eq:types-subst,eq:types-exists,eq:types-eq}, variable substitution, $\exists$ and $=$ correspond to preimage and image under the $\partial_s$.

Combined with \cref{rmk:interp-hyperdoctrine}, this yields several equivalent ways to describe the concept of an interpretation $\alpha : (\@L,\@T) -> (\@L',\@T')$ between theories:
\begin{enumerate}[label=(\roman*)]
\item \label{rmk:interp-types:interp}
maps $\alpha : \@L -> \@L'_{\omega_1\omega}$ such that $\@T' |- \alpha(\@T)$, modulo $\@T'$-provable equivalence;
\item \label{rmk:interp-types:homom}
homomorphisms of hyperdoctrines $(\alpha_n : \@L_{\omega_1\omega}^n/\@T -> \@L_{\omega_1\omega}'^n/\@T')_{n < \omega}$;
\item \label{rmk:interp-types:types}
families of Borel maps $(\alpha^*_n : \@S_n(\@T') -> \@S_n(\@T))_{n < \omega}$, commuting with projections $\partial_s : \@S_n -> \@S_m$:
\begin{align*}
\alpha^*_m \circ \partial_s &= \partial_s \circ \alpha^*_n : \@S_n(\@T') -> \@S_m(\@T), \\
\shortintertext{and also obeying}
(\alpha^*_m)^{-1} \circ \partial_s &= \partial_s \circ (\alpha^*_n)^{-1} : \@B(\@S_n(\@T)) -> \@B(\@S_m(\@T'))
\end{align*}
for each $s : m -> n$ (the maps here are image/preimage in the following commutative square);
\begin{equation*}
\begin{tikzcd}
\@S_n(\@T') \dar["\alpha_n^*"'] \rar["\partial_s"] &
\@S_m(\@T') \dar["\alpha_m^*"] \\
\@S_n(\@T) \rar["\partial_s"] &
\@S_m(\@T)
\end{tikzcd}
\end{equation*}
\item \label{rmk:interp-types:mod}
families of Borel maps $(\alpha^*_Y : \Mod_Y(\@T') -> \Mod_Y(\@T))_Y$ for all countable sets $Y$, equivariant under the logic action (of $\Sym(Y,Z)$ for all countable $Y,Z$).
\end{enumerate}
The correspondence between \cref{rmk:interp-types:interp} and \cref{rmk:interp-types:homom} was described in \cref{rmk:interp-hyperdoctrine}.
The maps in \cref{rmk:interp-types:homom} are, via $\@L_{\omega_1\omega}^n/\@T \cong \@B(\@S_n(\@T))$, preimage under the maps in \cref{rmk:interp-types:types}; the conditions in \cref{rmk:interp-types:types} ensure that preimage under those maps preserves substitution, $\exists$, and $=$.
The maps in \cref{rmk:interp-types:types} are obtained from \cref{rmk:interp-types:mod} by passing to $\alpha^*_Y \times \id_{Y^n} : \Mod_Y^n(\@T') = \Mod_Y(\@T') \times Y^n -> \Mod_Y(\@T) \times Y^n = \Mod_Y^n(\@T)$ and then quotienting by the logic action, i.e., passing to types.
\end{remark}

\begin{proof}[Proof that these correspondences are inverses of each other]
Starting from $\alpha : \@T -> \@T'$ as in \cref{rmk:interp-types:interp} and \cref{rmk:interp-types:homom}, we indeed recover $\alpha$ up to $\@T'$-provable equivalence as $(\alpha^*_n)^{-1}$ by the completeness theorem; thus the composite
$\text{\cref{rmk:interp-types:interp}} ->
\text{\cref{rmk:interp-types:mod}} ->
\text{\cref{rmk:interp-types:types}} ->
\text{\cref{rmk:interp-types:homom}} \cong \text{\cref{rmk:interp-types:interp}}$ is the identity.
The passage from \cref{rmk:interp-types:types} to \cref{rmk:interp-types:homom} $\cong$ \cref{rmk:interp-types:interp} is also injective, i.e., a Borel map $\@S_n(\@T') -> \@S_n(\@T)$ is determined by its preimage map between the Borel $\sigma$-algebras, because singletons in $\@S_n(\@T)$ are Borel by the Scott isomorphism theorem.
Finally, \cref{rmk:interp-types:mod} $->$ \cref{rmk:interp-types:types} is injective: given an arbitrary equivariant family $(f_Y : \Mod_Y(\@T') -> \Mod_Y(\@T))_Y$, for each $\@M \in \Mod_Y(\@T')$, the $\@L$-structure $f_Y(\@M)$ is determined by the types $\tp(f_Y(\@M), \vec{a}) = \tp((f_Y \times \id_{Y^n})(\@M, \vec{a}))$ of all finite tuples $\vec{a} \in Y^n$.
\end{proof}

\Cref{rmk:interp-types} allows us to work with $\@L_{\omega_1\omega}$ theories in a syntax-independent way: a theory $\@T$ is completely determined, up to bi-interpretability, \cref{rmk:interp-types:homom} by the hyperdoctrine $(\@L_{\omega_1\omega}^n/\@T)_n$, which is an algebraic structure, or
\cref{rmk:interp-types:types} by the family of type spaces $(\@S_n(\@T))_n$ and projections $\partial_s$ between them, or
\cref{rmk:interp-types:mod} by the family of spaces of models $(\Mod_Y(\@T))_Y$ equipped with the logic action.
Among these, \cref{rmk:interp-types:mod} has the benefit of living entirely in the familiar category of standard Borel spaces.
This makes it possible to not only recover but even abstractly define a theory via its spaces of models:

\begin{theorem}[Lopez-Escobar, Becker--Kechris]
\label{thm:mod-equiv}
We have a dual equivalence of categories from:
\begin{itemize}
\item  countable $\@L_{\omega_1\omega}$ theories $(\@L,\@T)$, and interpretations between them as morphisms; to
\item  families of standard Borel spaces $M_Y$ over all countable sets $Y$, equipped with Borel actions $\Sym(Y, Z) \times M_Y -> M_Z$, and equivariant families of Borel maps as morphisms;%
\footnote{Formally, these are presheaves on the groupoid of sets, enriched in the category of standard Borel spaces.}
\end{itemize}
taking a theory $\@T$ to the family of spaces $\Mod_Y(\@T)$ and an interpretation $\alpha$ to $\alpha^*$.
\end{theorem}
\begin{proof}
Said functor from the former category to the latter is full and faithful by the preceding remark (which boils down to Lopez-Escobar in the form of \cref{thm:mod-types}).
It remains to show it is essentially surjective, i.e., given such a family of spaces $(M_Y)_Y$ equipped with Borel actions of $\Sym(Y,Z)$, to show that they are equivariantly isomorphic to $(\Mod_Y(\@T))_Y$ for some theory $\@T$.

By \cite[2.7.3]{BK}, there is a language $\@L_\#N$ such that $M_\#N$ Borel $\Sym(\#N)$-equivariantly embeds into $\Mod_\#N(\@L_\#N)$, hence by Lopez-Escobar there is a $\@L_\#N$-sentence $\phi_\#N$ such that $\Mod_\#N(\phi_\#N) \cong M_\#N$ as $\Sym(\#N)$-spaces.
It is also easy to find languages $\@L_N$ such that $M_N$ $\Sym(N)$-equivariantly embeds into $\Mod_N(\@L_N)$ for each finite $N < \omega$, hence again by Lopez-Escobar there are sentences $\phi_N$ such that $\Mod_N(\phi_N) \cong M_N$.
For example, for each subgroup $H \le \Sym(N)$, we may create a structure $\@M_H$ in a language $\@L_H$ on the set $N$ such that the $H$-orbits in each $N^n$ become definable, so that $\Aut(\@M_H) = H$; then letting $\@L_N := \bigsqcup_{H \le \Sym(N)} \@L_H$ together with countably many nullary relations, $\Mod_N(\@L_N)$ has $2^{\aleph_0}$ many points with each possible stabilizer $H \le \Sym(N)$; now by partitioning the $\Sym(N)$-orbits of $M_N$ according to stabilizers, we may easily embed it into $\Mod_N(\@L_N)$.

So we have found Borel $\Sym(N)$-equivariant isomorphisms $\Mod_N(\phi_N) \cong M_N$ for each $N \le \omega$.
By unioning the languages (and modifying $\phi_N$ to assert the symbols from the other languages are false), we may assume the $\phi_N$ are over the same language.
By adding to $\phi_N$ a clause ``there exist exactly $N$ elements'', we may assume each $\phi_N$ has only models of size $N$.
Then $\phi := \bigvee_{N \le \omega} \phi_N$ is such that $\Mod_N(\phi) \cong M_N$ for each $N \le \omega$; call these isomorphisms $f_N$.
Then for any other countable set $Y$, define $f_Y : \Mod_Y(\phi) \cong M_Y$ by $f_Y(\@M) := g \cdot f_{\abs{Y}}(g^{-1} \cdot \@M)$, for any bijection $g : \abs{Y} \cong Y$ (where $\abs{Y} \le \omega$ is the cardinality).
Using that $f_{\abs{Y}}$ is $\Sym(\abs{Y})$-equivariant, this is easily seen to be independent of $g$ and to be equivariant with respect to the actions of all $\Sym(Y,Z)$.
\end{proof}

\subsection{Operations on theories}

We also recall here the following simple methods of combining theories; see \cite[\S B.2]{CK}.

\begin{definition}
\label{def:thy-coprod}
Given countably many theories $(\@L_i,\@T_i)$, their \defn{coproduct theory} is the disjoint union $(\bigsqcup_i \@L_i, \bigsqcup_i \@T_i)$.
Clearly, for any set $Y$, we have a canonical $\Sym(Y,Z)$-equivariant isomorphism
\begin{eqalign*}
\Mod_Y(\bigsqcup_i \@T_i) \cong \prod_i \Mod_Y(\@T_i);
\end{eqalign*}
thus, $\Mod_Y(\bigsqcup_i \@T_i)$ has the universal property of the categorical product (in the category of families of standard Borel spaces with Borel $\Sym(Y,Z)$-actions from \cref{thm:mod-equiv}).
Dually, this means $\bigsqcup_i \@T_i$ has the universal property of the coproduct: an interpretation $\bigsqcup_i \@T_i -> \@T'$ into another theory $\@T'$ is equivalently a family of interpretations $\@T_i -> \@T'$ for each $i$.

In particular, this means that $\bigsqcup_i \@T_i$ is the join (least upper bound) of the $\@T_i$ in the preorder of theories under the interpretability relation $->$.
\end{definition}


\begin{definition}
\label{def:thy-prod}
Given countably many theories $(\@L_i,\@T_i)$, their \defn{product theory} $(\bigoplus_i \@L_i, \bigoplus_i \@T_i)$ has the dual universal property that an interpretation $\@T' -> \bigoplus_i \@T_i$ is equivalently a family of interpretations $\@T' -> \@T_i$ for each $i$.
Thus, the space of models should be
\begin{eqalign*}
\Mod_Y(\bigoplus_i \@T_i) \cong \bigsqcup_i \Mod_Y(\@T_i).
\end{eqalign*}
In other words, a model of $\bigoplus_i \@T_i$ should be a specification of a unique index $i$ together with a model of $\@T_i$ (and no other data).
By \cref{thm:mod-equiv}, this suffices to define $\bigoplus_i \@T_i$ up to bi-interpretability.
For an explicit axiomatization of $\bigoplus_i \@T_i$, see \cite[6.2]{CK}.

In particular, $\bigoplus_i \@T_i$ is the meet (greatest lower bound) of the $\@T_i$ with respect to interpretability.
\end{definition}

\section{Structurability of CBERs}
\label{sec:cber}

\subsection{Structurability}
\label{sec:str}

\begin{definition}[see \cite{JKL}, \cite{CK}]
\label{def:str}
Let $E \subseteq X^2$ be a CBER, $\@T$ be a countable $\@L_{\omega_1\omega}$ theory in a countable language $\@L$.
A \defn{(Borel) $\@T$-structuring} of $E$ is a family of models $\@M = (\@M_C)_{C \in X/E}$, where for each $E$-class $C \in X/E$, we have $\@M_C \in \Mod_C(\@T)$; and ``$x |-> \@M_{[x]_E}$ is Borel'', meaning the following conditions which are equivalent by \cref{thm:str} below:
\begin{enumerate}[label=(\roman*)]
\item \label{def:str:enum}
For some (equivalently any) Borel family of enumerations $(g_x : \abs{[x]_E} \cong [x]_E)_{x \in X}$ (as in \cref{thm:lusin-novikov-cber}\cref{thm:lusin-novikov-cber:enum}), the following map is Borel:
\begin{eqalign*}
X &--> \bigsqcup_{N \le \omega} \Mod_N(\@T) \\
x &|--> g_x^{-1} \cdot \@M_{[x]_E}.
\end{eqalign*}
\item \label{def:str:param}
For any countable set $Y$, standard Borel space $Z$, Borel map $f : Z -> X$, and Borel family of bijections $(g_z : Y \cong [f(z)]_E)_{z \in Z}$, the following map is Borel:
\begin{eqalign*}
Z &--> \Mod_Y(\@T) \\
z &|--> g_z^{-1} \cdot \@M_{[f(z)]_E}.
\end{eqalign*}
\item \label{def:str:fiber}
For each $n$-ary relation symbol $R \in \@L$, the following set is a Borel subset of $X^{n+1}$:
\begin{eqalign*}
\~R^\@M &:=
\set[\big]{(x, x_0, \dotsc, x_{n-1}) \in X^{n+1}}{x \mathrel{E} x_0 \mathrel{E} \dotsb \mathrel{E} x_{n-1} \AND R^{\@M_{[x]_E}}(x_0, \dotsc, x_{n-1})}.
\end{eqalign*}
\item \label{def:str:global}
(assuming $\@L$ has no nullary relation symbols)
For each $n$-ary $R \in \@L$, the following is a Borel subset of $X^n$:
\begin{eqalign*}
R^\@M &:=
\set[\big]{(x_0, \dotsc, x_{n-1}) \in X^n}{x_0 \mathrel{E} \dotsb \mathrel{E} x_{n-1} \AND R^{\@M_{[x_0]_E}}(x_0, \dotsc, x_{n-1})}.
\end{eqalign*}
(This is often taken as the main definition, e.g., in \cite{JKL}, \cite{CK}.)
\item[(iii$'$)\rlap{ and (iv$'$)}] \hphantom{ and (iv$'$)}
\phantomitem{(iii$'$)}\label{def:str:fiber'}%
\phantomitem{(iv$'$)}\label{def:str:global'}%
Same as above, but replacing $R$ with arbitrary $\phi(x_0,\dotsc,x_{n-1}) \in \@L_{\omega_1\omega}$ (giving rise to Borel $\~\phi^\@M \subseteq X^{n+1}$, respectively $\phi^\@M \subseteq X^n$ for $n \ge 1$).
\end{enumerate}
Let $\Mod_E(\@T)$ denote the set of $\@T$-structurings of $E$.
If $\Mod_E(\@T) \ne \emptyset$, we call $E$ \defn{$\@T$-structurable}.
\end{definition}

\begin{lemma}
\label{thm:str}
The conditions above are equivalent.
\end{lemma}
\begin{proof}
\cref{def:str:fiber'}$\implies$\cref{def:str:fiber}$\iff$\cref{def:str:global} is clear, as is
\cref{def:str:fiber'}$\implies$\cref{def:str:global'}.

\cref{def:str:global'}$\implies$\cref{def:str:fiber'}:
Regard $\phi(x_0,\dotsc,x_{n-1})$ as having an extra free variable.

\cref{def:str:fiber}$\implies$\cref{def:str:param}:
A generator for the Borel $\sigma$-algebra of $\Mod_Y(\@T) \subseteq \Mod_Y(\@L) = \prod_{R \in \@L^n} 2^{Y^n}$ is $\{\@M \mid R^\@M(\vec{a})\}$ for some $n$-ary $R \in \@L$ and $\vec{a} \in Y^n$; the preimage of this set under the map in \cref{def:str:param} is
\begin{eqalign*}
\set[\big]{z \in Z}{R^{g_z^{-1} \cdot \@M_{[f(z)]_E}}(\vec{a})}
&= \set[\big]{z \in Z}{R^{\@M_{[f(z)]_E}}(g_z(\vec{a}))}
= \set[\big]{z \in Z}{\~R^\@M(f(z),g_z(\vec{a}))}.
\end{eqalign*}

\cref{def:str:param}$\implies$\cref{def:str:enum}, for any $(g_x)_x$:
It suffices to verify Borelness of the map in \cref{def:str:enum} restricted to the union of all $E$-classes of each size $N \le \omega$, which is immediate from \cref{def:str:param} (for $Z =$ said union).

\cref{def:str:enum} for some $(g_x)_x \implies$\cref{def:str:fiber'}:
\begin{align*}
\~\phi^\@M = \set[\big]{(x, x_0, \dotsc, x_{n-1})}{x \mathrel{E} x_0 \mathrel{E} \dotsb \mathrel{E} x_{n-1} \AND R^{g_x^{-1} \cdot \@M_{[x]_E}}(g_x^{-1} \cdot x_0, \dotsc, g_x^{-1} \cdot x_{n-1})}.
&\qedhere
\end{align*}
\end{proof}

\begin{definition}[{see \cite[3.1]{CK}}]
\label{def:str-classbij}
Let $(X,E), (Y,F)$ be two CBERs.
A \defn{(Borel) class-bijective homomorphism} $f : (X,E) -> (Y,F)$ is a Borel map $f : X -> Y$ which restricts to a bijection $f : [x]_E \cong [f(x)]_F$ for each $x \in X$.

Given such $f$, and a $\@T$-structuring $\@M = (\@M_D)_{D \in Y/F}$ of $F$, the \defn{pullback $\@T$-structuring} $f^{-1} \cdot \@M$ of $E$ is defined by
\begin{align*}
(f^{-1} \cdot \@M)_C := (f|C : C \cong f(C))^{-1} \cdot \@M_{f(C)}
    \quad \text{for each $C \in X/E$}.
\end{align*}
The Borelness conditions \labelcref{def:str}\cref{def:str:fiber} or \cref{def:str:param} are easily seen.
\end{definition}

\begin{definition}[{see \cite[\S B.3]{CK}}]
\label{def:str-interp}
Let $(\@L,\@T), (\@L',\@T')$ be two theories, $\alpha : \@T -> \@T'$ be an interpretation, and $(X,E)$ be a CBER with a $\@T'$-structuring $\@M = (\@M_C)_{C \in X/E}$.
Recall from \cref{def:interp} the notion of \emph{$\alpha$-reduct} of a model of $\@T'$.
The \defn{$\alpha$-reduct} of $\@M$ is the $\@T$-structuring $\alpha^*\@M$ where
\begin{eqalign*}
(\alpha^*\@M)_C := \alpha^*(\@M_C)
    \quad \text{for each $C \in X/E$}.
\end{eqalign*}
The Borelness condition \labelcref{def:str}\cref{def:str:param} is again easily seen (just compose the map $Z -> \Mod_Y(\@T')$ with $\alpha^*_Y : \Mod_Y(\@T') -> \Mod_Y(\@T)$ from \cref{eq:interp-mod}).
\end{definition}

Thus, the operation
\begin{align*}
\yesnumber
\label{eq:str}
\{\text{CBERs}\} \times \{\text{theories}\} &--> \{\text{sets}\} \\
(E, \@T) &|--> \Mod_E(\@T)
\end{align*}
is jointly (contravariantly) functorial in both variables, with respect to class-bijective homomorphisms, respectively interpretations.
To verify functoriality, note that we have the obvious identities
\begin{align*}
f^{-1} \cdot g^{-1} \cdot \@K = (g \circ f)^{-1} \cdot \@K, &&
\alpha^* \beta^* \@K = (\beta \circ \alpha)^* \@K, &&
f^{-1} \cdot \alpha^*\@K = \alpha^*(f^{-1} \cdot \@K)
\end{align*}
(as well as $\id^{-1} \cdot \@K = \@K$ and $\id^* \@K = \@K$).

\begin{remark}
\label{rmk:str-homom}
\Cref{def:str}\cref{def:str:enum,def:str:param} are manifestly syntax-independent, depending only on the spaces of models of $\@T$ and not any particular axiomatization.
We now give reformulations of the notion of structuring in terms of the other two ways of presenting a theory from \cref{rmk:interp-types}, namely \cref{rmk:interp-types:homom} the algebra of formulas modulo $\@T$-equivalence and \cref{rmk:interp-types:types} the spaces of types.

Recall from \cref{ex:interp-mod} that a countable model $\@M \in \Mod_Y(\@T)$ is essentially by definition a homomorphism $\phi |-> \phi^\@M : \@L_{\omega_1\omega}^n/\@T -> 2^{Y^n}$ taking formulas to relations for each $n$, preserving $\bigwedge, \exists$, etc.
Thus, a structuring $\@M = (\@M_C)_{C \in X/E} \in \Mod_E(\@T)$ is a family of such homomorphisms
\begin{align*}
\yesnumber
\label{eq:str-homom-classwise}
(\@L_{\omega_1\omega}^n/\@T &--> 2^{C^n})_{n < \omega, C \in X/E} \\
\shortintertext{or}
(\@L_{\omega_1\omega}^n/\@T &--> 2^{[x]_E^n})_{n < \omega, x \in X} \\
\phi &|--> \phi^{\@M_{[x]_E}}
\end{align*}
which is $E$-invariant in $x$.
Now a family of relations $\phi^{\@M_{[x]_E}} \in 2^{[x]_E^n} \subseteq 2^{X^n}$, for each $x \in X$, is equivalently a single relation in $2^{X^{n+1}}$ (since $(2^{X^n})^X \cong 2^{X^{n+1}}$), which is contained in
\begin{align*}
E^n_X :=& \set{(x, x_0, \dotsc, x_{n-1}) \in X^{n+1}}{(x_0,\dotsc,x_{n-1}) \in [x]_E^n} \\
=& \set{(x, x_0, \dotsc, x_{n-1}) \in X^{n+1}}{x \mathrel{E} x_0 \mathrel{E} \dotsb \mathrel{E} x_{n-1}}.
\end{align*}
(Here, the notation $E^n_X$ refers to the fiber product of $n$ copies of $E$ over $X$ via the first projection.)
The family $(\phi^{\@M_{[x]_E}} \in 2^{[x]_E^n})_{x \in X}$ corresponds to the single Borel relation $\~\phi^\@M \subseteq E^n_X$ from \labelcref{def:str}\cref{def:str:fiber'}.
Thus, a structuring $\@M \in \Mod_E(\@T)$ is equivalently a homomorphism
\begin{align*}
\yesnumber
\label{eq:str-homom}
(\@L_{\omega_1\omega}^n/\@T &--> \@B^E(E^n_X))_{n < \omega} \\
\phi &|--> \~\phi^\@M,
\end{align*}
where $\@B^E(E^n_X) \subseteq \@B(E^n_X)$ is the $\sigma$-algebra of Borel sets $R \subseteq E^n_X$ which are $E$-invariant in the first coordinate; these may be thought of as ``Borel families of $n$-ary relations in $E$-classes''.

Note that here, ``homomorphism'' means that countable Boolean operations are preserved, while variable substitutions, quantifiers, and equality operate on the last $n$ coordinates only (corresponding to the usual operations in $2^{C^n}$ from \cref{eq:str-homom-classwise}).
Thus for $s : m -> n$, variable substitution along $s$ maps
\begin{align*}
\@B^E(E^m_X) &--> \@B^E(E^n_X) \\
R &|--> \{(x,x_0,\dotsc,x_{n-1}) \mid R(x,x_{s(0)},\dotsc,x_{s(m-1)})\};
\end{align*}
while $\exists$ is given by image along the projection $E^{n+1}_X -> E^n_X$, and likewise $=$ is given by the image of the diagonal embedding $E = E^1_X -> E^2_X$.
In the following subsection, we will show that the family of $\sigma$-algebras $(\@B^E(E^n_X))_{n < \omega}$ equipped with these operations is isomorphic to the Lindenbaum--Tarski hyperdoctrine of a theory, analogous to the theory of $Y$-enumerated sets from \cref{ex:interp-mod}, so that a $\@T$-structuring, like a countable $\@T$-model, may be understood as an interpretation.

Note also that $\@B^E(E^n_X)$ is isomorphic to the Borel sets in the quotient space $E^n_X/E$ (meaning the quotient by the equivalence relation of $E$-equivalence in the first coordinate).
For $n \ge 1$, we have $E^n_X/E \cong E^{n-1}_X$, along which the $E$-invariant $\~\phi^\@M \subseteq E^n_X$ descends to $\phi^\@M \subseteq E^{n-1}_X$ from \labelcref{def:str}\cref{def:str:global'}.
(For $n = 0$, we instead have $E^n_X/E = X/E$.)
Since $\@L_{\omega_1\omega}^n/\@T$ is isomorphic to the Borel sets $\sqsqbr{\phi}$ in the space of types $\@S_n(\@T)$ (\cref{def:types}), \cref{eq:str-homom} becomes preimage under the Borel maps
\begin{align*}
\yesnumber
\label{eq:str-types}
\tp_\@M^n : E^n_X/E &--> \@S_n(\@T) \\
[(x,x_0,\dotsc,x_{n-1})] &|--> \tp(\@M_{[x]_E},x_0,\dotsc,x_{n-1}),
\end{align*}
with $(\tp_\@M^n)^{-1}(\sqsqbr{\phi}) = \~\phi^\@M$.
Thus a structuring $\@M \in \Mod_E(\@T)$ is also equivalently given by a family of such Borel maps $(\tp_\@M^n)_{n < \omega}$, subject to commutativity with image/preimage under the variable projections $E^n_X/E -> E^m_X/E$ for each $s : m -> n$ (analogous to the conditions in \cref{rmk:interp-types}\cref{rmk:interp-types:types}).
\end{remark}

\subsection{Scott theories of CBERs}
\label{sec:scott}

We now define an operation $E |-> \@T_E$ that associates to each CBER $E$ a canonical theory $\@T_E$ that ``completely encodes'' $E$, thereby allowing the concepts of ``structurability'', ``class-bijective homomorphism'', and (as shown in the next subsection) even ``CBER'' to be subsumed by equivalent model-theoretic concepts.
This construction is essentially from \cite[\S4.2, B.2]{CK}, from which we also adopt the name ``Scott sentence of a CBER'', tweaked to fit our convention of working with theories instead of sentences.
The main novelty of our approach here is to do everything in terms of the following intrinsic definition in terms of $E$, rather than a specific syntactic encoding as in \cite[\S4.2]{CK} (we will also review a version of that encoding below; see \cref{cst:scott-explicit}).

\begin{definition}
\label{def:scott}
Let $(X,E)$ be a CBER.
The \defn{Scott theory} of $E$ is the countable $\@L_{\omega_1\omega}$ theory $(\@L_E,\@T_E)$ defined uniquely up to bi-interpretability via \cref{thm:mod-equiv} by declaring its models on any countable set $Y$ to be bijections between $Y$ and some $E$-class:
\begin{align*}
\Mod_Y(\@T_E) :=& \{u \in X^Y \mid u : Y \cong C \text{ for some $C \in X/E$}\} \\
\yesnumber
\label{eq:scott-borel}
=& \set*{u \in X^Y}{\begin{aligned}
& \forall y \ne z \in Y\, (u(y) \ne u(z)), \\
& \forall y, z \in Y\, (u(y) \mathrel{E} u(z)), \\
& \forall x \mathrel{E} u(y_0)\, \exists y \in Y\, (u(y) = x)
\end{aligned}} \quad \text{for arbitrary $y_0 \in Y$}
\end{align*}
(and $\Mod_Y(\@T_E) := \emptyset$ if $Y = \emptyset$).
This clearly defines a standard Borel subspace of $X^Y$.
For two countable sets $Y, Z$, the logic action $\Sym(Y,Z) \times \Mod_Y(\@T_E) -> \Mod_Z(\@T_E)$ is given by relabeling: for $g : Y \cong Z$ and $u \in \Mod_Y(\@T_E)$,
\begin{eqalign*}
g \cdot u :=& u \circ g^{-1}.
\end{eqalign*}
We have a \defn{tautological $\@T_E$-structuring} $\@H_E$ of $E$, given by the bijections $(\@H_E)_C := \id_C : C \cong C$.
\end{definition}

Informally speaking, the models of $\@T_E$ are just the $E$-classes, with every element labeled.
This is made precise by the following alternate characterizations of $\@T_E$, in terms of \cref{rmk:interp-types}\cref{rmk:interp-types:types,rmk:interp-types:homom}:

\begin{proposition}
\label{thm:scott-types-formulas}
For each $n \in \#N$,
the $n$-types of $\@T_E$ are determined by the Borel isomorphism
\begin{align*}
\yesnumber
\label{eq:scott-types}
\tp^n_{\@H_E} : E^n_X/E &\cong \@S_n(\@T_E) \\
[(x,x_0,\dotsc,x_{n-1})] &|-> \tp(\id_{[x]_E}, x_0,\dotsc,x_{n-1}),
\end{align*}
where $E^n_X/E$ is the space of $n$-tuples in $E$-classes ($\cong E^{n-1}_X$ for $n \ge 1$) from \cref{eq:str-types}.

Thus, the algebra of $n$-ary formulas modulo $\@T_E$ is determined by the isomorphism
\begin{alignat*}{4}
\yesnumber
\label{eq:scott-formulas}
(\@L_E)_{\omega_1\omega}^n/\@T_E &\cong \@B(\@S_n(\@T_E)) &&\cong \@B(E^n_X/E) &&\cong \@B^E(E^n_X), \\
\phi &|-> \sqsqbr{\phi} &&|-> \phi^{\@H_E} &&|-> \~\phi^{\@H_E},
\end{alignat*}
where $\@B^E(E^n_X)$ is the algebra of ``Borel families of $n$-ary relations in $E$-classes'' from \cref{eq:str-homom}.
\end{proposition}
\begin{proof}
For any model $u \in \Mod_Y(\@T_E)$, i.e., $u : Y \cong u(Y) \in X/E$, and $\vec{a} \in Y^n$, the unique preimage of $\tp(u,\vec{a}) \in \@S_n(\@T_E)$ under \cref{eq:scott-types} is easily seen to be $[(x,u(\vec{a}))]$ for any $x \in u(Y)$.
So $\tp^n_{\@H_E}$ is a bijection; since $E^n_X/E, \@S_n(\@T_E)$ are quotients by Polish group actions, this implies Borel isomorphism.
\end{proof}

\begin{corollary}[{see \cite[B.2]{CK}}]
\label{thm:str-interp}
For any CBER $(X,E)$ and theory $(\@L,\@T)$, we have a bijection%
\begin{align*}
\{\text{interpretations } \@T -> \@T_E\} &\cong \Mod_E(\@T) = \{\text{$\@T$-structurings of $E$}\} \\
\alpha &|-> \alpha^*(\@H_E).
\end{align*}
\end{corollary}

\begin{proof}
Note that $\alpha^*(\@H_E)$ is the structuring such that the type map \cref{eq:str-types} is, via \cref{eq:scott-types},
\begin{alignat*}{2}
\tp_{\alpha^*(\@H_E)}^n = \alpha^*_n \circ \tp_{\@H_E}^n : E^n_X/E &\cong \@S_n(\@T_E) &&--> \@S_n(\@T) \\
[(x,x_0,\dotsc,x_{n-1})] &|-> \tp(\id_{[x]_E},\vec{x}) &&|--> \tp(\alpha^*((\@H_E)_{[x]_E}), \vec{x}).
\end{alignat*}
Or via \cref{eq:scott-formulas}, $\alpha^*(\@H_E)$ interprets formulas as
\begin{alignat*}{2}
\@L_{\omega_1\omega}^n/\@T &--> (\@L_E)_{\omega_1\omega}^n/\@T_E &&\cong \@B^E(E^n_X) \\
[\phi(x_0,\dotsc,x_{n-1})] &|--> \alpha(\phi) &&|-> \~{\alpha(\phi)}^{\@H_E}.
\end{alignat*}
Thus the inverse of the bijection takes a structuring $\@M$ to $\alpha : \@T -> \@T_E$ such that $\smash{\~{\alpha(\phi)}^{\@H_E}} = \smash{\~\phi^\@M}$, or $\alpha^*_n \circ \tp_{\@H_E}^n = \tp_\@M^n$; this determines a unique interpretation $\alpha$ by \cref{thm:scott-types-formulas} and \cref{rmk:interp-types}.
\end{proof}

\begin{proposition}[{see \cite[4.7]{CK}}]
\label{thm:str-classbij}
For any CBERs $(X,E), (Y,F)$, we have a bijection
\begin{align*}
\{\text{class-bijective homomorphisms } E -> F\} &\cong \Mod_E(\@T_F) = \{\text{$\@T_F$-structurings of $E$}\} \\
f &|-> f^{-1} \cdot \@H_F.
\end{align*}
\end{proposition}
\begin{proof}
A class-bijective homomorphism $f : E -> F$ is clearly the same thing as a family of bijections
$
(f|C : C \cong f(C) \in Y/F)_{C \in X/E},
$
i.e., a family of models of $\@T_F$
\begin{equation*}
\paren[\big]{f|C = \id_{f(C)} \circ f|C = (f|C)^{-1} \cdot (\@H_F)_{f(C)} = (f^{-1} \cdot \@H_F)_{f(C)} \in \Mod_C(\@T_F)}_{C \in X/E};
\end{equation*}
the Borelness of $f$ is easily seen to correspond to the Borelness of the structuring.
\end{proof}

Note that here, $f^{-1} \cdot \@H_F$ is the structuring whose types are, via \cref{eq:scott-types},
\begin{alignat*}{2}
\tp_{f^{-1} \cdot \@H_F}^n : E^n_X/E &--> F^n_Y/F &&\cong \@S_n(\@T_F) \\
[(x,x_0,\dotsc,x_{n-1})] &|--> [(f(x),f(\vec{x}))] &&|-> \tp(\id_{[f(x)]_F}, f(\vec{x})),
\end{alignat*}
Combining the two preceding results, we have

\begin{corollary}[{see \cite[B.3]{CK}}]
\label{thm:scott-full-faithful}
For any CBERs $(X,E), (Y,F)$, we have a bijection
\begin{align*}
\{\text{class-bijective homomorphisms } E -> F\} &\cong \Mod_E(\@T_F) \cong \{\text{interpretations } \@T_F -> \@T_E\}.
\end{align*}
Namely, each $f : E -> F$ corresponds to the unique $\alpha : \@T_F -> \@T_E$ such that $\alpha^*(\@H_E) = f^{-1} \cdot \@H_F$, or such that $\alpha^*_n : \@S_n(\@T_E) \cong E^n_X/E -> F^n_Y/F \cong \@S_n(\@T_F)$ is the map induced by $f$.
\qed
\end{corollary}

In other words, the construction of Scott theories $E |-> \@T_E$ extends to a contravariant functor
\begin{align*}
\{\text{CBERs, class-bijective homomorphisms}\} &--> \{\text{$\@L_{\omega_1\omega}$ theories, interpretations}\}
\end{align*}
which is a full and faithful embedding (i.e., restricts to a bijection on each hom-set).
We will characterize its essential image (i.e., image-up-to-isomorphism) in the following subsection.

\medskip
We conclude this subsection by briefly recalling the explicit syntactic axiomatization of Scott theories $\@T_E$ from \cite[\S4.2]{CK}.
We do this in order to keep this paper self-contained, and because some additional light may be shed on this axiomatization from our current perspective.

\begin{definition}
\label{def:Tsep}
The \defn{theory of countable separating families} $(\@L_\sep,\@T_\sep)$ is given by
\begin{align*}
    \@L_\sep & := \{ U_i \}_{i \in \#N} \quad \text{(each $U_i$ unary)}, \\
    \@T_\sep & := \{ \forall y \forall z\, [y \neq z \to \bigvee_{i \in \#N} (U_i(y) <-> \neg U_i(z))] \}.
\end{align*}
Note that for any countable set $Y$, we have a canonical bijection
\begin{align*}
\yesnumber
\label{eq:Tsep}
\Mod_Y(\@L_\sep) = \prod_{i \in \#N} 2^Y &\cong (2^\#N)^Y \\
\@M = (U_i^\@M)_{i \in \#N} &|-> (u_\@M : y |-> (U_i^\@M(y))_i), \\
\shortintertext{which restricts to}
\Mod_Y(\@T_\sep) &\cong \{u : Y -> 2^\#N \mid u \text{ injective}\}.
\end{align*}
Moreover, the logic action on the left is easily seen to correspond on the right to the relabeling action $g \cdot u = u \circ g^{-1}$ as in \cref{def:scott} of the Scott theory of a CBER.
\end{definition}

\begin{construction}
\label{cst:scott-explicit}
Now given a CBER $(X,E)$, we may assume without loss that $X \subseteq 2^\#N$ is a Borel subspace.
Then $\Mod_Y(\@T_E)$, as defined abstractly in \cref{eq:scott-borel}, embeds via \cref{eq:Tsep} into $\Mod_Y(\@T_\sep)$, and so by Lopez-Escobar, $\@T_E$ must be axiomatizable over the language $\@L_\sep$ by $\@T_\sep$ plus some additional axioms.
Namely, we need axioms whose interpretations in a model $\@M \in \Mod_Y(\@L_\sep)$ will correspond via the above bijection \cref{eq:Tsep} to the second and third conditions in \cref{eq:scott-borel}.

To axiomatize the second condition $\forall y, z \in Y\, (u(y) \mathrel{E} u(z))$, note that since $E \subseteq 2^\#N \times 2^\#N$ is Borel, we may write ``$u(y) \mathrel{E} u(z)$'' as a countable Boolean combination of assertions of the form ``$u(y)_i = 1$'' or ``$u(z)_i = 1$'' for various $i \in \#N$, which correspond via \cref{eq:Tsep} to $U_i^\@M(y)$ or $U_i^\@M(z)$ respectively.
Thus letting $\Phi_E(y,z)$ be the corresponding quantifier-free Boolean combination of atomic formulas $U_i(y), U_i(z)$, the sentence
\begin{equation}
\label{eq:scott-explicit-homom}
\forall y, z\, \Phi_E(y, z)
\end{equation}
axiomatizes those $\@M \in \Mod_Y(\@L_\sep)$ corresponding via \cref{eq:Tsep} to $u$ obeying $\forall y, z \in Y\, (u(y) \mathrel{E} u(z))$.

To axiomatize the third condition in \cref{eq:scott-borel}, first write
$
E = \bigcup_{i \in \#N} f_i
$
for Borel functions $f_i : X -> X$ by Lusin--Novikov.
Then rewrite the third condition in \cref{eq:scott-borel} as
\begin{alignat*}{2}
&\exists y_0 \in Y\, \forall x \mathrel{E} u(y_0)\, &&\exists y \in Y\, (u(y) = x) \\
\iff&  \exists y_0 \in Y\, \forall i \in \#N &&\exists y \in Y\, (u(y) = f_i(u(y_0))).
\end{alignat*}
Now similarly to before, using that the graphs of $f_i \subseteq 2^\#N \times 2^\#N$ are Borel, find quantifier-free $\@L_\sep$-formulas $\phi_i(y_0, y)$ axiomatizing ``$u(y) = f_i(u(y_0))$''; then
\begin{align}
\label{eq:scott-explicit-classsurj}
\exists y_0\, \bigwedge_{i \in \#N} \exists y\, \phi_i(y_0, y)
\end{align}
works.
So we may axiomatize $\@T_E$ over $\@L_E := \@L_\sep$ via $\@T_\sep \cup \{\cref{eq:scott-explicit-homom}, \cref{eq:scott-explicit-classsurj}\}$.
\end{construction}

\subsection{Lusin--Novikov functions}
\label{sec:TLN}

We now aim to characterize those $\@L_{\omega_1\omega}$ theories bi-interpretable to a Scott theory of some CBER, i.e., the essential image of the full functorial embedding from \cref{thm:scott-full-faithful}.
The key ingredient is to isolate the role played by the formulas $\phi_i$ obtained from Lusin--Novikov in \cref{cst:scott-explicit}.

\begin{definition}
\label{def:TLN}
The \defn{theory of Lusin--Novikov functions} $(\@L_\LN,\@T_\LN)$ is given by
\begin{eqalign*}
    \@L_\LN := \{ & f_i \}_{i \in \#N} \quad \text{(each $f_i$ a unary function)}, \\
    \@T_\LN := \{ & \forall x \forall y \bigvee_i (f_i(x) = y) \}.
\end{eqalign*}
(Following \cref{cvt:relational}, we formally encode each $f_i$ as its graph relation $F_i(x,y) \iff f_i(x) = y$, and add to $\@T_\LN$ axioms saying that each $F_i$ is a graph of a function.)
\end{definition}

\begin{example}
\label{ex:TLN-scott}
In \cref{cst:scott-explicit} of the Scott theory $\@T_E$ of a CBER, we have an interpretation $\alpha : \@T_\LN -> \@T_E$ mapping $f_i |-> \phi_i$ (formally, mapping the graph $f_i(x) = y$ to $\phi_i(x,y)$).
Indeed, note that the formula $\Phi_E$ defined there may be given in terms of the $\phi_i$ as
\begin{eqalign*}
\Phi_E(x,y)  :=  \bigvee_i \phi_i(x,y).
\end{eqalign*}
Then the axiom \cref{eq:scott-explicit-homom} in $\@T_E$ becomes precisely $\alpha(\@T_\LN)$.
\end{example}

\begin{proposition}
\label{thm:interp-LN}
A theory $(\@L,\@T)$ interprets $\@T_\LN$ iff the following equivalent conditions hold:
\begin{enumerate}[label=(\roman*)]
\item
The space of 1-types $\@S_1(\@T)$ is standard Borel, and for any pointed model $(\@M,a) \in \Mod_Y^1(\@T)$, the map $b |-> \tp(\@M,a,b) : Y -> \@S_2(\@T)$ is injective.
\item
For any countable set $Y$, the logic action $\Sym(Y) \actson \Mod_Y^1(\@T)$ is free and smooth.
\item
There is a countable set $\@F \subseteq \@L_{\omega_1\omega}^1$ of formulas in one variable, such that every pointed model of $\@T$ is rigid and $\@F$-categorical.
\end{enumerate}
\end{proposition}
\begin{proof}
These conditions are equivalent by \cref{thm:mod-types,thm:mod-smooth-cat}.

$\Longrightarrow$:
Let $(\@M,a) \in \Mod_Y(\@T)$ be a pointed model, and let $g : \@M \cong \@M$ be an automorphism fixing $a$.
Then for any $b \in Y$, there is a Lusin--Novikov function $f_i : Y -> Y$ that is part of $\@M$ such that $f_i(a) = b$, whence $g(b) = g(f_i(a)) = f_i(g(a)) = f_i(a) = b$.
This shows that $(\@M,a)$ is rigid.

To show categoricity over a countable fragment $\@F$, we may as well assume that the Lusin--Novikov functions $f_i$ are part of the language $\@L$ (and $\@T |- \@T_\LN$); if not, then consider the expansion of $\@T$ with such unary functions $f_i$ defined to be equal to the function graphs $\phi_i$ in the image of the interpretation $\@T_\LN -> \@T$.
Then $\@T$ admits quantifier elimination for formulas with at least one free variable, since every existential $\exists y\, \psi(x_0,\dotsc,x_{n-1},y)$, $n \ge 1$, is $\@T$-provably equivalent to $\bigvee_i \psi(x_0,\dotsc,x_{n-1},f_i(x_0))$.
So we may take $\@F$ to consist of just the atomic formulas.
(To undo the expansion of $\@T$, replace these with existential formulas over the function graph relations $\phi_i$.)

$\Longleftarrow$:
Suppose the logic action $\Sym(N) \actson \Mod_N^1(\@T)$ is free and smooth for each $N \le \omega$; let $D_N \subseteq \Mod_N^1(\@T)$ be a Borel transversal.
For each $i < N$, let
\begin{align*}
F_{N,i} :=& \Sym(N) \cdot (D_N \times \{i\}) \subseteq \Mod_N^2(\@T) \\
=& \{(\@M,a,b) \mid g(b) = i \text{ for the unique $g \in \Sym(N)$ such that $g(\@M,a) \in D_N$}\}.
\end{align*}
These are Borel $\Sym(N)$-invariant and form graphs of functions of the first two variables $\@M,a$ covering $\Mod_N^2(\@T)$ as $i$ varies, hence by Lopez-Escobar, are defined by formulas $\phi_{N,i}(x,y)$ such that $\@T |- \text{``there are $N$ elements''} -> \text{``$\phi_{N,i}$ is a function''} \wedge \forall x \forall y \bigvee_i \phi_{N,i}(x,y)$.
For $N \le i < \omega$, let $\phi_{N,i}(x,y) := (x = y)$ be the identity function, and put $\phi_i := \bigvee_{N \le \omega} (\text{``there are $N$ elements''} \wedge \phi_{N,i})$.
Then the $\phi_i$ form an interpretation $\@T_\LN -> \@T$.
\end{proof}

The above conditions may be verified abstractly for a Scott theory $\@T = \@T_E$ of a CBER $(X,E)$, without resorting to the explicit axiomatization in \cref{cst:scott-explicit}.
Indeed, every (unpointed) model of $\@T_E$, i.e., bijection with an $E$-class, is clearly rigid.
And by \cref{thm:scott-types-formulas}, $\@S_1(\@T_E) \cong E^0_X = X$ is standard Borel.
Note also that $\@S_0(\@T_E) \cong X/E$, with the projection $\partial_1 : \@S_1(\@T_E) -> \@S_0(\@T_E)$ (from \cref{rmk:types-simplicial}) corresponding to the quotient map $X ->> X/E$:
\begin{equation*}
\begin{tikzcd}
\mathllap{E \rightrightarrows {}}
X \dar[phantom, "\cong"{rotate=-90}] \rar[two heads] &
X/E \dar[phantom, "\cong"{rotate=-90}]
\\
\mathllap{\ker(\partial_1) \rightrightarrows {}}
\@S_1(\@T) \rar["\partial_1"] &
\@S_0(\@T)
\end{tikzcd}
\end{equation*}
We may thus recover the CBER $(X,E)$ canonically up to isomorphism from $\@T = \@T_E$, as the kernel of $\partial_1 : \@S_1(\@T) -> \@S_0(\@T)$.

\begin{theorem}
\label{thm:scott-LNsep}
For a countable $\@L_{\omega_1\omega}$ theory $\@T$, the following are equivalent:
\begin{enumerate}[label=(\roman*)]
\item \label{thm:scott-LNsep:types}
$\partial_1 : \@S_1(\@T) -> \@S_0(\@T)$ is surjective, $\@S_1(\@T)$ is standard Borel, and for any model $\@M \in \Mod_Y(\@T)$, the map $\tp_\@M : Y -> \@S_1(\@T)$ is injective.
\item \label{thm:scott-LNsep:mod}
$\@T$ has no empty models, and for any countable set $Y$, the logic action $\Sym(Y) \actson \Mod_Y(\@T)$ is free and the logic action $\Sym(Y) \actson \Mod_Y^1(\@T)$ on pointed models is smooth.
\item \label{thm:scott-LNsep:logic}
Every model of $\@T$ is nonempty rigid, and there is a countable set $\@F \subseteq \@L_{\omega_1\omega}^1$ of formulas in one variable such that every pointed model is $\@F$-categorical.
\item \label{thm:scott-LNsep:LNsep}
$\@T$ has no empty models and interprets $\@T_\LN \sqcup \@T_\sep$.
\item \label{thm:scott-LNsep:scott}
$\@T$ is bi-interpretable with the Scott theory $\@T_E$ of a CBER $E$, namely $E = \ker(\partial_1)$.
\end{enumerate}
Thus we have a dual equivalence of categories
\begin{align*}
\{\text{CBERs, class-bijective homomorphisms}\} &\simeq \{\text{$\@L_{\omega_1\omega}$ theories obeying \cref*{thm:scott-LNsep:types}--\cref*{thm:scott-LNsep:scott}, interpretations}\}
\end{align*}
taking a CBER $E$ to its Scott theory $\@T_E$, with the correspondence on morphisms as in \cref{thm:scott-full-faithful}.
\end{theorem}
\begin{proof}
As in \cref{thm:interp-LN}, the equivalence between \cref{thm:scott-LNsep:types,thm:scott-LNsep:mod,thm:scott-LNsep:logic} follows from \cref{thm:mod-types,thm:mod-smooth-cat}, noting that to say $\partial_1 : \@S_1(\@T) -> \@S_0(\@T)$ is surjective means that every model is nonempty.

These conditions are equivalent to \cref{thm:scott-LNsep:LNsep} by \cref{thm:interp-LN}, the fact that models of $\@T_\sep$ are clearly rigid, and that conversely, if models are rigid, then the family $\@F$ from \cref{thm:scott-LNsep:logic} must form a countable separating family of unary relations in every model by Scott's isomorphism theorem.

As noted above, by \cref{cst:scott-explicit} or an easy check, a Scott theory of a CBER obeys \cref{thm:scott-LNsep:types}--\cref{thm:scott-LNsep:LNsep}.

Finally, suppose the theory $\@T$ obeys \cref{thm:scott-LNsep:types}; we verify \cref{thm:scott-LNsep:scott} for $E = \ker(\partial_1)$.
Note first that $E \subseteq \@S_1(\@T)^2$ is Borel, since its lift in $\Mod_Y^1(\@T)^2$ is the preimage under the projection of ${\cong} \subseteq \Mod_Y(\@T)^2$ which is Borel by \cref{thm:types-borel}.
Clearly $E$ is also countable, since we are assuming models to be countable.
So $E = \ker(\partial_1) \subseteq \@S_1(\@T)^2$ is a CBER.

For any countable model $\@M \in \Mod_Y(\@T)$, by \cref{thm:scott-LNsep:types}, we have an injection $\tp_\@M : Y `-> \@S_1(\@T)$, whose image is clearly the $E$-class $\partial_1^{-1}(\tp(\@M))$.
This defines a map
\begin{align*}
f_Y : \Mod_Y(\@T) &--> \Mod_Y(\@T_E) \subseteq \@S_1(\@T)^Y \\
\@M &|--> \tp_\@M,
\end{align*}
which is easily seen to be Borel
(since $f_Y(\@M)(a) = \tp(\@M,a) \in \sqsqbr{\phi} \iff \phi^\@M(a)$)
and equivariant under bijections $g : Y \cong Z$
(since $f_Y(g \cdot \@M)(a) = \tp(g \cdot \@M, a) = \tp(\@M, g^{-1} \cdot a) = (f_Y(\@M) \circ g^{-1})(a)$).
Thus by \cref{rmk:interp-types}\cref{rmk:interp-types:mod}, $f_Y = \alpha_Y^*$ for an interpretation $\alpha : \@T_E -> \@T$.
Since both $\@T, \@T_E$ have rigid models, the logic actions on both $\Mod_Y(\@T), \Mod_Y(\@T_E)$ are free; thus to check that $f_Y$ is bijective, it suffices to check that it is bijective on orbits, i.e., it induces a bijection $\@S_0(\@T) \cong \@S_0(\@T_E)$.

By \cref{thm:scott-types-formulas}, we have a bijection $\@S_0(\@T) \cong \@S_1(\@T)/E \cong \@S_0(\@T_E)$ taking $C |-> \tp(\id_C)$ for each $E$-class $C \in \@S_1(\@T)/E$.
Let the 1-types in $C$ be realized in $\@M \in \Mod_Y(\@T)$; then $\tp_\@M : Y \cong C \subseteq \@S_1(\@T)$.
So $\tp_\@M : \tp_\@M \cong \id_C$ as models of $\@T_E$, and so $\tp(\id_C) = \tp(\tp_\@M) = \tp(f_Y(\@M))$.
\end{proof}

\subsection{Structurability via interpretations}
\label{sec:str-interp}

Using \cref{thm:scott-LNsep}, as well as the universal properties of Scott theories from \cref{sec:scott}, we may now translate questions of structurability into purely model-theoretic terms.

To avoid having to restate this assumption repeatedly, below we will assume that all theories under consideration have no empty models.

\begin{corollary}
\label{thm:str-impl-interp}
For two $\@L_{\omega_1\omega}$ theories $\@T, \@T'$ (in respective languages), the following are equivalent:
\begin{enumerate}[label=(\roman*)]
\item \label{thm:str-impl-interp:str}
Every $\@T$-structurable CBER is $\@T'$-structurable.
\item \label{thm:str-impl-interp:interp}
There exists an interpretation $\@T' -> \@T \sqcup \@T_\LN \sqcup \@T_\sep$.
\end{enumerate}
\end{corollary}
\begin{proof}
We have
\begin{align*}
\text{\cref{thm:str-impl-interp:str}}
&\iff  \forall \text{ CBERs } E,\, (\@T -> \@T_E \implies \@T' -> \@T_E)
    &&\text{by \cref{thm:str-interp}} \\
&\iff  \forall \text{ theories } \@T_E <- \@T_\LN \sqcup \@T_\sep,\, (\@T -> \@T_E \implies \@T' -> \@T_E)
    &&\text{by \cref{thm:scott-LNsep}};
\end{align*}
since the least theory $\@T_E$ (under the interpretability preorder) interpreting both $\@T_\LN \sqcup \@T_\sep$ and $\@T$ is their least upper bound $\@T \sqcup \@T_\LN \sqcup \@T_\sep$ (recall \cref{def:thy-coprod}), this is equivalent to \cref{thm:str-impl-interp:interp}.
\end{proof}

Recall that in \cref{ex:finsub-pt}, we showed how the well-known fact that selecting a finite nonempty set in each class of a CBER ($\@T_\finsub$-structurability) implies smoothness ($\@T_\pt$-structurability) amounts to an interpretation $\@T_\pt -> \@T_\finsub \sqcup \@T_\LN \sqcup \@T_\sep$.
We will exhibit many more such concrete instances of \cref{thm:str-impl-interp} in \cref{sec:examples}.

We note that \cref{thm:str-impl-interp}, as a statement only about \emph{existence} of structurings, contains strictly less information than \cref{thm:scott-LNsep} and \cref{thm:str-interp}, which precisely characterize the set of all possible structurings.
This becomes relevant for problems in Borel combinatorics, rather than just CBERs, where the input data is something richer than an equivalence relation:

\begin{definition}
Given two theories $\@T, \@T'$, an interpretation $\alpha : \@T -> \@T'$, and a $\@T$-structuring $\@M$ of a CBER $E$, an \defn{$\alpha$-expansion} of $\@M$ is a $\@T'$-structuring $\@M'$ such that $\alpha^*(\@M') = \@M$.
\end{definition}

For example, consider the problem of Borel coloring a (locally countable) Borel graph $G \subseteq X^2$.
Letting $E$ be the connectedness relation generated by $G$, this problem may be described as starting with a structuring $\@M$ of $E$ by the theory $\@T_\graph$ of graphs, and seeking an expansion of $\@M$ to a structuring by the theory $\@T_\!{graph+color}$ of graphs equipped with a coloring, i.e., an $\alpha$-expansion for $\alpha : \@T_\graph `-> \@T_\!{graph+color}$ the inclusion.
See \cref{thm:kst-lfcolor} for a concrete instance of this.

\begin{corollary}
\label{thm:str-expan-interp}
For any interpretation $\alpha : \@T -> \@T'$, the following are equivalent:
\begin{enumerate}[label=(\roman*)]
\item \label{thm:str-expan-interp:str}
Every $\@T$-structuring of a CBER admits an $\alpha$-expansion to a $\@T'$-structuring.
\item \label{thm:str-expan-interp:interp}
There exists an interpretation $\beta : \@T' -> \@T \sqcup \@T_\LN \sqcup \@T_\sep$ such that $\beta \circ \alpha : \@T -> \@T \sqcup \@T_\LN \sqcup \@T_\sep$ is the inclusion (up to provable equivalence as usual).
\end{enumerate}
\end{corollary}
\begin{proof}
We again have
\begin{align*}
\text{\cref{thm:str-expan-interp:str}}
&\iff  \forall \text{ CBERs } E,\, \forall \text{ interpretations } \gamma : \@T -> \@T_E,\, \exists \gamma' : \@T' -> \@T_E\, (\gamma' \circ \alpha = \gamma)
    &&\text{by \labelcref{thm:str-interp}} \\
&\iff  \forall \text{ theories } \@T_E <- \@T_\LN \sqcup \@T_\sep,\, \forall \gamma : \@T -> \@T_E,\, \exists \gamma' : \@T' -> \@T_E\, (\gamma' \circ \alpha = \gamma)
    &&\text{by \labelcref{thm:scott-LNsep}}.
\end{align*}
If this holds, then taking $\@T_E := \@T \sqcup \@T_\LN \sqcup \@T_\sep$ with $\gamma$ the inclusion yields \cref{thm:str-expan-interp:interp}.
Conversely, given $\beta$ from \cref{thm:str-expan-interp:interp}, then for any other $\@T_E, \gamma$ as above, we may combine the interpretation $\gamma$ and the interpretation $\@T_\LN \sqcup \@T_\sep$ into an interpretation from the coproduct $\delta : \@T \sqcup \@T_\LN \sqcup \@T_\sep -> \@T_E$ (recall again \cref{def:thy-coprod}); then $\gamma' := \delta \circ \beta$ works.
\end{proof}

Theories of the form $\@T \sqcup \@T_\LN \sqcup \@T_\sep$ featured above play a special role in structurability.
Note that up to \emph{mutual} interpretability, they are just all the Scott theories of CBERs by \cref{thm:scott-LNsep}, since $\sqcup$ is join in the preorder $->$.
But up to \emph{bi-}interpretability, we may characterize them as follows.

\begin{definition}[{see \cite[4.13]{CK}}]
A class-bijective homomorphism $f : (X,E) -> (Y,F)$ between CBERs which is also injective is called an \defn{invariant (Borel) embedding}, denoted $f : E \sqle^i F$.

Given a class $\@K$ of CBERs, we say $E \in \@K$ is \defn{invariantly universal in $\@K$} if it is $\sqle^i$-largest among (i.e., contains invariant isomorphic copies of all other) CBERs in $\@K$.
By the Borel Schröder--Bernstein theorem \cite[15.7]{Kcdst}, such an $E$ is unique up to Borel isomorphism if it exists.

In particular, there is an \defn{invariantly universal CBER} among all CBERs, denoted $\#E_\infty$.

More generally, for any theory $\@T$, there is an invariantly universal $\@T$-structurable CBER, denoted $\#E_{\infty\@T}$.%
\footnote{Note that in this paper, unlike many other works including \cite{JKL}, \cite{CK}, $\@T$ denotes an arbitrary countable $\@L_{\omega_1\omega}$ theory and \emph{not} the class of trees.}
A CBER $E$ which is isomorphic to some $\#E_{\infty\@T}$ is called \defn{universally structurable}.
In that case, as follows from \cref{thm:einfT} below, we may always take $\@T = \@T_E$.
\end{definition}

\begin{remark}
We may characterize those class-bijective homomorphisms $f : (X,E) -> (Y,F)$ which are invariant embeddings model-theoretically, in terms of the corresponding $\@T_F$-structurings $\@M$ of $E$ from \cref{thm:str-classbij}, or the corresponding interpretations $\alpha : \@T_F -> \@T_E$ from \cref{thm:scott-full-faithful}, as follows.
Note that since $f$ is class-bijective, it is an invariant embedding iff it is also a reduction, i.e., descends to an injection $X/E `-> Y/F$.
Via \cref{thm:str-classbij}, this means that
\begin{enumerate}[label=(\roman*)]
\item
Distinct $E$-classes $C \ne D \in X/E$ receive non-isomorphic $\@T_F$-structures $\@M_C \not\cong \@M_D$.
\end{enumerate}
Since $E$-classes are isomorphism types of $\@T_E$-models, this means
\begin{enumerate}[resume*]
\item
The interpretation $\alpha$ induces an injection $\alpha^*_0 : \@S_0(\@T_E) -> \@S_0(\@T_F)$.
\end{enumerate}
Or dually, between algebras of sentences $\@B(\@S_0(\@T_E)) \cong (\@L_E)_{\omega_1\omega}^0/\@T_E$,
\begin{enumerate}[resume*]
\item
$\alpha$ is a surjection on provable equivalence classes of sentences.
\end{enumerate}
We may also replace sentences with arbitrary formulas here, or 0-types with $n$-types, since $f$ being injective clearly implies the same for the induced maps $E^n_X -> F^n_Y$ which descend to the induced maps $\alpha^*_n : \@S_n(\@T_E) -> \@S_n(\@T_F)$ by \cref{thm:scott-full-faithful}.
We call such $\alpha : \@T_F -> \@T_E$ a \defn{surjective interpretation}.
\end{remark}

\begin{remark}
\label{thm:interp-schroeder-bernstein}
It follows that we have a Schröder--Bernstein theorem for interpretations:
if $\@T, \@T'$ are both Scott theories of CBERs, and there exist surjective interpretations $\@T -> \@T'$ and $\@T' -> \@T$, then there exists a bi-interpretation $\@T \cong \@T'$.
\end{remark}

\begin{proposition}
\label{thm:einf-LNsep}
$\@T_\LN \sqcup \@T_\sep$ is the Scott theory of the invariantly universal CBER $\#E_\infty$.
\end{proposition}
\begin{proof}
For any other CBER $E$, the interpretation $\@T_\LN \sqcup \@T_\sep -> \@T_E$ given by the explicit \cref{cst:scott-explicit} of $\@T_E$ is obviously surjective (on atomic formulas, hence on all formulas), since it is the identity on the $\@T_\sep$ part.

Or in other words, we may obviously find a $(\@T_\LN \sqcup \@T_\sep)$-structuring $\@M$ of $E$ such that distinct $E$-classes $C, D$ receive non-isomorphic $\@M_C \not\cong \@M_D$, by taking the $\@T_\sep$ part of the structuring to come from an embedding into $2^\#N$ (and the $\@T_\LN$ part from Lusin--Novikov applied to $E$).
\end{proof}

\begin{corollary}
\label{thm:einfT-LNsep}
For any theory $\@T$, $\@T \sqcup \@T_\LN \sqcup \@T_\sep$ is the Scott theory of the invariantly universal $\@T$-structurable CBER $\#E_{\infty\@T}$.
\end{corollary}
\begin{proof}
The inclusion $\@T -> \@T \sqcup \@T_\LN \sqcup \@T_\sep$ shows that $\@T \sqcup \@T_\LN \sqcup \@T_\sep$ is the Scott theory of a $\@T$-structurable CBER $\#E_{\infty\@T}$.
For any other $\@T$-structurable CBER $E$, we get an interpretation $\@T -> \@T_E$ by \cref{thm:str-interp}; combined with the surjective interpretation $\@T_\LN \sqcup \@T_\sep ->> \@T$ given by the preceding result, we get a surjective interpretation $\@T \sqcup \@T_\LN \sqcup \@T_\sep ->> \@T_E$, corresponding to an invariant embedding $E \sqle^i \#E_{\infty\@T}$.
\end{proof}

\begin{example}
\label{ex:einf-aper}
Let $\@T_\inf$ be the \defn{theory of infinite sets} (in the empty language):
\begin{eqalign*}
\@T_\inf := \bigwedge_n \exists x_0 \dotsb \exists x_{n-1} \bigwedge_{i \neq j} (x_i \neq x_j).
\end{eqalign*}
Then $\@T_\inf \sqcup \@T_\LN \sqcup \@T_\sep$ is the Scott theory of the invariantly universal \defn{aperiodic} (i.e., every class is infinite) CBER, denoted $\#E_{\infty\!{aper}}$ in \cref{fig:cber-interp,fig:cber-interp-big}.
\end{example}

\begin{remark}
More generally, for any theory $\@T$ and any Scott theory $\@T_E$, the coproduct $\@T \sqcup \@T_E$ is also a Scott theory, namely of the \defn{universal $\@T$-structurable CBER over $E$} denoted $E \ltimes \@T$ from \cite[4.1]{CK}, with the defining property that a class-bijective homomorphism $F -> E \ltimes \@T$ is the same thing as a class-bijective homomorphism $F -> E$ together with a $\@T$-structuring of $F$.

In particular, if $\@T = \@T_F$ is also a Scott theory, then $\@T_E \sqcup \@T_F$ is the Scott theory of the \defn{class-bijective product} $E \otimes F$ from \cite[4.17]{CK}, the product in the category of CBERs and class-bijective homomorphisms; this is clear from the full embedding given by \cref{thm:scott-full-faithful}.
\end{remark}

The following is the model-theoretic counterpart of \cite[4.13]{CK}:

\begin{corollary}
\label{thm:einfT}
For a countable $\@L_{\omega_1\omega}$ theory $\@T$, the following are equivalent:
\begin{enumerate}[label=(\roman*)]
\item \label{thm:einfT:self}
$\@T \cong \@T \sqcup \@T_\LN \sqcup \@T_\sep$, i.e., $\@T$ is the Scott theory of $\#E_{\infty\@T}$.
\item \label{thm:einfT:other}
$\@T \cong \@T' \sqcup \@T_\LN \sqcup \@T_\sep$ for some theory $\@T'$, i.e., $\@T$ is the Scott theory of some universally structurable CBER.
\item \label{thm:einfT:surj}
$\@T_\LN \sqcup \@T_\sep -> \@T$, and for any $\@T'$ interpreting $\@T$, there is a surjective interpretation $\@T -> \@T'$.
\end{enumerate}
\end{corollary}
\begin{proof}
\cref{thm:einfT:self}$\implies$\cref{thm:einfT:other} is obvious, and
\cref{thm:einfT:other}$\implies$\cref{thm:einfT:surj} by \cref{thm:einfT-LNsep}, since if $\@T -> \@T'$ then $\@T'$ is the Scott theory of a $\@T$-structurable CBER (by \cref{thm:scott-LNsep} and \cref{thm:str-interp}).
Now suppose \cref{thm:einfT:surj}; then from $\@T -> \@T \sqcup \@T_\LN \sqcup \@T_\sep$ we get a surjective such interpretation, and we clearly also have a surjective $\@T \sqcup \@T_\LN \sqcup \@T_\sep ->> \@T$ by combining the identity with the given $\@T_\LN \sqcup \@T_\sep -> \@T$, whence $\@T \cong \@T \sqcup \@T_\LN \sqcup \@T_\sep$ by the Schröder--Bernstein property \labelcref{thm:interp-schroeder-bernstein}.
\end{proof}

\section{Examples of (non-)interpretations}
\label{sec:examples}


Many fundamental theorems about CBERs can be naturally understood in terms of structurability by particular $\@L_{\omega_1\omega}$ theories, and the proofs of these theorems often amount to interpretations between these theories.
In this section, we give several examples of such interpretations to illustrate this perspective.
These examples include the Feldman--Moore theorem, its generalization to $\omega$-coloring the intersection graph on finite subsets of $E$-classes, and the Slaman--Steel marker lemma.

Moreover, we will show several \emph{non-interpretability} results between these theories, thereby making precise the idea that certain kinds of Borel combinatorial structures are ``more powerful'' than others.
For example, we will show that an $\omega$-coloring of finite subsets is ``more powerful'' than the Feldman--Moore theorem, despite both being available on every CBER; see \cref{ex:color3-cex}.

The theories we will consider are listed in \cref{tbl:theories}.
(Recall our \cref{cvt:relational}, that function symbols occurring in examples are implicitly encoded as their graph relations.)

\begin{table}[htbp]
\centering
\makeatletter
\def\makename#1{$\mathsf{#1}$}
\def\maketheory#1{\maketheory@#1}
\def\maketheory@#1{\begingroup\def\arraystretch{1}$\begin{array}[t]{>{}l}#1\end{array}$\endgroup}
\def\thyrule{\cmidrule{2-3}}
\makeatother
\begin{tabular}{
    >{\collectcell\makename}l<{\endcollectcell}
    l
    >{\collectcell\maketheory}l<{\endcollectcell}
}
\toprule
     \ang{name} \; \text{(ref)} & $ \mathcal{L}_{\ang{\!{name}}}$ & {\mathcal{T}_{\ang{\!{name}}}}
\\ \thyrule
    \pt \; \text{(\labelcref{ex:finsub-pt})}
    & constant $c$ (unary rel $C$)
    & {\emptyset \; (\exists! x\, C(x))}
\\ \thyrule
    \finsub \; \text{(\labelcref{ex:finsub-pt})}
    & unary relation $D$
    & {
       \bigvee_{n \ge 1} \exists z_0, \dotsc, z_{n-1}\, \forall w\, ( D(w) <-> \bigvee_{i < n} (w = z_i))}
\\ \thyrule
    \LO \; \text{(\labelcref{ex:finsub-pt})}
    & binary relation $<$
    & {\forall x\, (x \not< x), \;
       \forall x \forall y \forall z\, ((x < y < z) -> (x < z)) \\
       \forall x \forall y\, ((x = y) \vee (x < y) \vee (y < x))}
\\ \thyrule
    \sep \; \text{(\labelcref{def:Tsep})}
    & unary relations $\{U_i\}_{i \in \#N}$
    & {\forall x \forall y\, [ x \neq y \rightarrow \bigvee_i ( U_i(x) \leftrightarrow \neg U_i(y))]}
\\ \thyrule
    \LN \; \text{(\labelcref{def:TLN})}
    & unary functions $\{f_i\}_{i \in \#N}$
    & {\forall x \forall y \bigvee_i (f_i(x) = y)}
\\ \thyrule
    \LNbij \; \text{\labelcref{eq:LNFM}}
    & unary functions $\{f_i\}_{i \in \#N}$
    & {\forall x \forall y \bigvee_i (f_i(x) = y), \;
       \forall y \bigwedge_i \exists ! x \, (f_i(x) = y)}
\\ \thyrule
    \LNinv \; \text{\labelcref{eq:LNFM}}
    & unary functions $\{f_i\}_{i \in \#N}$
    & {\forall x \forall y \bigvee_i (f_i(x) = y), \;
       \forall x \bigwedge_i (f_i(f_i(x)) = x)}
\\ \thyrule
    \text{\begin{tabular}[t]{@{}>{}l<{}@{}} $\Gamma$ (a monoid) \\ \quad\labelcref{eq:LNFM} \end{tabular}}
    & unary functions $\{a_\gamma\}_{\gamma \in \Gamma}$
    & {\forall x \forall y \bigvee_\gamma (a_\gamma(x) = y) \\
       \forall x\, (a_1(x) = x), \;
       \forall x \bigwedge_{\gamma, \gamma'} (a_{\gamma \gamma'}(x) = a_\gamma(a_{\gamma'}(x)))}
\\ \thyrule
    \FMfun \; \text{\labelcref{eq:LNFM}}
    & unary functions $\{a_\gamma\}_{\gamma \in \#N^{<\omega}}$
    & { \@T_{{\#N^{<\omega}}} }
\\ \thyrule
    \FMbij \; \text{\labelcref{eq:LNFM}}
    & unary functions $\{a_\gamma\}_{\gamma \in \#F_\omega}$
    & { \@T_{{\#F_\omega}} }
\\ \thyrule
    \FMinv \; \text{\labelcref{eq:LNFM}}
    & unary functions $\{a_\gamma\}_{\gamma \in \#Z_2^{*\omega}}$
    & { \@T_{{\#Z_2^{*\omega}}} }
\\ \thyrule
    \coloreq2 \; \text{(\labelcref{thm:FM-color2})}
    & binary relations $\{C_k\}_{k \in \#N}$
    & {\bigwedge_k \forall x \forall y \, (C_k(x,y) <-> C_k(y,x)) \\
        \forall x \forall y \, [ x \neq y \to \bigvee_k (C_k(x,y) \wedge \bigwedge_{l \neq k} \neg C_l(x,y)) ] \\
    \forall x_0,x_1, y_0,y_1 \\ \qquad [ \bigvee_{i,j \in 2} (x_i = y_j) \wedge \bigvee_{i \in 2} \bigwedge_{j \in 2} (x_i \neq y_j) \\ \qquad \to \bigwedge_l \neg (C_l(x_0,x_1) \wedge C_l(y_0,y_1)) ] }
\\ \thyrule
    \colorlt\omega \; \text{(\labelcref{thm:KM-color})}
    & $n$-ary relations $\{C_{nk}\}_{\substack{2 \le n \in \#N \\ k \in \#N}}$
    & {\bigwedge_{n,k} \bigwedge_{\sigma \in \Sym(n)} \forall x_0, \dots, x_{n-1} \\ \qquad [C_{nk}(\vec{x}) <-> C_{nk}(x_{\sigma(0)},\dots,x_{\sigma(n-1)})] \\
        \bigwedge_n \forall x_0, \dotsc, x_{n-1} \, [\bigwedge_{i \neq j < n} (x_i \ne x_j) \\ \qquad \to \bigvee_k (C_{nk}(\vec{x}) \wedge \bigwedge_{l \ne k} \neg C_{nl}(\vec{x}))] \\
    \bigwedge_n \forall x_0,\dots,x_{n-1}, y_0,\dots,y_{n-1} \\ \qquad [ \bigvee_{i,j < n} (x_i = y_j) \wedge \bigvee_{i < n} \bigwedge_{j < n} (x_i \neq y_j) \\ \qquad \to \bigwedge_k  \neg (C_{nk}(\vec{x}) \wedge C_{nk}(\vec{y})) ] }
\\ \thyrule
    \text{\begin{tabular}[t]{@{}>{}l<{}@{}} $Y$ (a set) (\labelcref{ex:interp-mod}) \end{tabular}}
    & constants $\{c_y\}_{y \in Y}$
    & {\forall x \bigvee_{y \in Y} (x = c_y), \;
       \bigwedge_{y \neq z \in Y} (c_y \neq c_z)}
\\ \thyrule
    \enum \; \text{(\labelcref{thm:enum-LNpt})}
    & $\bigoplus_{Y \leq \omega} \@L_Y$
    & {\bigoplus_{Y \leq \omega} \@T_Y}
\\ \thyrule
    \inf \; \text{(\labelcref{ex:einf-aper})}
    & $\emptyset$
    & {\bigwedge_n \exists x_0 \dotsb \exists x_{n-1} \bigwedge_{i \neq j} (x_i \neq x_j)}
\\ \thyrule
    \marker \; \text{(\labelcref{thm:marker})}
    & unary relations $\{A_i\}_{i \in \#N}$
    & {\bigwedge_i \forall x (A_{i+1}(x) \rightarrow A_i(x)) \\
       \bigwedge_i \exists x \, A_i(x), \;
       \forall x \bigvee_i \neg A_i(x)}
\\ \thyrule
    lfgraph \; \text{(\labelcref{thm:kst-lfcolor})}
    & binary relation $G$
    & {
       \forall x\, \neg (xEx), \;
       \forall x \forall y\, (xEy \to yEx) \\
       \forall x \bigvee_n \exists y_0 \dotsb y_{n-1} \forall z\, (xGz \to \bigvee_{i < n} (z = y_i))}
\\ \thyrule
    \text{\begin{tabular}[t]{@{}>{}l<{}@{}} $\!{lfgraph{+}\$\omega color}$ \\ \quad(\labelcref{thm:kst-lfcolor}) \end{tabular}}
    & \begin{tabular}[t]{@{}>{}l<{}@{}} binary relation $G$ \\ unary relations $\{C_i\}_{i \in \omega}$ \end{tabular}
    & {\@T_\!{lfgraph} \\
       \forall x \bigvee_i (C_i(x) \wedge \bigwedge_{j \neq i} \neg C_j(x)) \\
       \forall x \forall y\, [xGy \to \bigwedge_i \neg (C_i(x) \wedge C_i(y))]}
\\ \bottomrule
\end{tabular}
\caption{Theories of some common combinatorial structures (with reference to first mention).}
\label{tbl:theories}
\end{table}

\subsection{Variants of the Feldman--Moore theorem}
\label{sec:FM-variants}

A key result in the theory of CBERs is the Feldman--Moore theorem \cite{FM}, which states that every CBER is generated by a Borel action of a countable group.
In fact, the usual proof of Feldman--Moore yields an \emph{a priori} stronger conclusion, namely that the Borel functions covering $E = \bigcup_i f_i$ given by the Lusin--Novikov theorem \labelcref{thm:lusin-novikov-cber}\cref{thm:lusin-novikov-cber:LN} can be upgraded to a family of Borel involutions that still cover $E$.
Indeed, this stronger statement is nowadays often understood, especially in the context of Borel combinatorics, as the real content of the Feldman--Moore theorem (see e.g., \cite[3.4]{Kcber}).

We may distinguish between these variants of the Feldman--Moore and Lusin--Novikov theorems, by restating their conclusions as structurability by different $\@L_{\omega_1\omega}$ theories.
Recall from \cref{ex:TLN-scott} that Lusin--Novikov applied to a CBER $E$ yields structurability by the theory $\@T_\LN$ in the language with countably many unary functions $\@L_\LN = \{f_i\}_{i \in \#N}$ asserting that they cover all pairs of elements.
We may strengthen this by requiring the $f_i$ to be bijections, or even involutions.
On the other hand, the original statement of Feldman--Moore is slightly different in flavor, in that it gives a transitive action on each $E$-class of a countable group (which we may assume to be the free group $\#F_\omega$, without loss of generality).
We may strengthen this to an action of a group generated by involutions (which we may assume to be the free product $\#Z_2^{*\omega}$), or weaken it to only an action of a monoid (e.g., the free monoid $\#N^{<\omega}$).
All told, we have 6 theories capturing variants of Lusin--Novikov and Feldman--Moore:
\begin{align*}
\@T_\LN = \@T_\LNfun &:= \text{covering family of functions},
& \@T_\FMfun = \@T_{{\#N^{< \omega}}} &:= \text{transitive $\#N^{< \omega}$ action}, \\
\yesnumber\label{eq:LNFM}
\@T_\LNbij &:= \text{covering family of bijections},
& \@T_\FMbij = \@T_{{\#F_\omega}} &:= \text{transitive $\#F_\omega$ action}, \\
\@T_\LNinv &:= \text{covering family of involutions},
& \@T_\FMinv = \@T_{{\#Z_2^{*\omega}}} &:= \text{transitive $\#Z_2^{*\omega}$ action}.
\end{align*}
Here by a \emph{covering} family of functions $f_i$ on a set $Y$, we mean that their graphs cover $Y^2$; and by $\@T_\Gamma$ in general for a countable monoid $\Gamma$, we mean the theory of sets equipped with a transitive $\Gamma$-action.
See \cref{tbl:theories} for precise axiomatizations of each of these theories.

Our somewhat pedantic naming scheme for these theories is that ``$\LN_{\ang{\!{property}}}$'' refers to a covering family of functions each with said property, while ``$\FM_{\ang{\!{property}}}$'' means that only the closure under composition/inverse of said functions is required to cover, i.e., we have a transitive action of the group/monoid generated by said functions.
Thus, the Lusin--Novikov theorem says that every CBER is structurable by $\@T_\LNfun$ ($= \@T_\LN$ from \cref{def:TLN}); while the Feldman--Moore theorem asserts, depending on the variant, structurability by $\@T_\FMbij$, $\@T_\FMinv$, or $\@T_\LNinv$.

Note that these theories are not \emph{a priori} equivalent in strength: for instance, it is not clear if a covering family of bijections may be canonically turned into involutions.
The following diagram shows the \emph{a priori} obvious interpretations (solid arrows) between the above 6 theories:
\begin{equation}
\label{diag:LNFM}
\begin{tikzcd}[column sep=4em]
\mathllap{\@T_{{\#Z_2^{*\omega}}} ={}}
\@T_\FMinv \rar &
\@T_\LNinv
\\
\mathllap{\@T_{{\#F_\omega}} ={}}
\@T_\FMbij \rar \uar&
\@T_\LNbij \uar \lar[shift right=2]  \dar[dashed, bend right=30]
\\
\mathllap{\@T_{{\#N^{<\omega}}} ={}}
\@T_\FMfun \rar \uar &
\@T_\LNfun \mathrlap{{}= \@T_\LN} \uar \lar[shift right=2]
\end{tikzcd}
\end{equation}
The vertical arrows on the right are identity interpretations; e.g. a covering family of involutions is already a covering family of bijections.
The two right-to-left interpretations are also essentially identity interpretations; e.g., a transitive $\#F_\omega$-action is already a covering family of bijections.
(Note that we do not have a similar interpretation $\@T_\LNinv \to \@T_\FMinv$, as compositions of involutions may no longer be involutions.)
The vertical arrows on the left are induced by surjective homomorphisms between the acting groups/monoids; e.g., a transitive $\#Z_2^{*\omega}$-action becomes a transitive $\#F_\omega$-action, by regarding $\#Z_2^{*\omega}$ as a quotient of $\#F_\omega$.
And the left-to-right interpretations define a transitive action by closing a covering family of functions under composition (and inverses for $\@T_\LNbij$); e.g., given a covering family of functions $f_i$, we get a transitive $\#N^{< \omega}$-action by defining the action of an element $s = s_0 s_1 \dots s_{n-1} \in \#N^{< \omega}$ to be the function $f_{s_0} \circ f_{s_1} \circ \dots \circ f_{s_{n-1}}$.

The proof of the Feldman--Moore theorem \cite{FM} (see also \cite[7.1.4]{GaoIDST}) now amounts to:

\begin{proposition}[Feldman--Moore]
\label{thm:FM}
There is an interpretation $\@T_\LNinv \to \@T_\LN \sqcup \@T_\sep$.
\end{proposition}

\begin{proof}
We must uniformly define a model of $\@T_\LNinv$, i.e., a covering family of involutions, given a family of Lusin--Novikov functions $(f_i)_{i \in \#N} |= \@T_\LN$ and a family of separating unary predicates $(U_i)_{i \in \#N} |= \@T_\sep$.
First we define from the $f_i$'s a covering family of partial injections $(f_{ij})_{i, j \in \#N}$ by
\begin{align*}
f_{ij} := f_i \cap f_j^{-1},
\quad \text{i.e.,} \quad
    f_{ij}(x) = y \coloniff (f_i(x) = y) \wedge (f_j(y) = x).
\end{align*}
Now we use the separating family of subsets $U_k$ to define a covering family of partial injections $(f_{ijk})_{i, j, k \in \#N}$, each with disjoint domain and range:
\begin{align*}
    f_{ijk}(x) = y \coloniff (f_{ij}(x) = y) \wedge U_k(x) \wedge \neg U_k(y).
\end{align*}
Finally, we extend the domain of each $f_{ijk}$ to get a covering family of involutions
\begin{align*}
    g_{ijk}(x) = y \coloniff (f_{ijk}(x) = y) \vee (f_{jik}(y) = x) \vee (x \not \in \dom(f_{ijk}) \wedge y \not \in \dom(f_{jik}) \wedge x = y),
\end{align*}
where of course $x \in \dom(f) \coloniff \exists z (f(x) = z)$.
\end{proof}

Thus, when combined with $\@T_\sep$, all 6 theories in \cref{diag:LNFM} become mutually interpretable.
This follows abstractly from \cref{thm:str-impl-interp} and the fact that all 6 theories structure every CBER (by the Feldman--Moore theorem); the point of the above \namecref{thm:FM} is that the \emph{proof} of Feldman--Moore is already essentially an interpretation.

Less obviously, the following shows that we may in fact interpret $\@T_\LNbij$ in $\@T_\LN$ alone, without $\@T_\sep$ (dashed arrow in \cref{diag:LNFM}).
In other words, the original statement of the Feldman--Moore theorem \cite{FM}, in terms of Borel group actions, may be proved using only the Lusin--Novikov theorem, without also using a countable separating family of Borel sets.
(We will show in the following subsection that this is false for $\@T_\LNinv$, i.e., the stronger version of Feldman--Moore yielding Borel involutions \emph{cannot} be proved using just Lusin--Novikov.)
It follows that the bottom 4 theories in \cref{diag:LNFM} are mutually interpretable.

\begin{proposition}
\label{thm:FMbij}
    There is an interpretation $\@T_\LNbij \to \@T_\LN$.
\end{proposition}

\begin{proof}
We must define a covering family of bijections given a covering family of functions $(f_i)_{i \in \#N} |= \@T_\LN$.
As in \cref{thm:FM}, start by taking the covering family of partial bijections $f_{ij} := f_i \cap f_j^{-1}$.
It now suffices to show how, from a single partial bijection $f$, we may define two total bijections $g, h$ such that $f \subseteq g \cup h$ (as graphs of functions); we may then apply this to each of the $f_{ij}$'s.

Informally, we consider the orbits $\{\dotsc, f^{-1}(x), x, f(x), f^2(x), \dotsc\}$ of $f$.
Each orbit is either
\begin{enumerate}[label=(\roman*)]
\item \label{thm:FMbij:periodic}
finite and periodic, in which case $f$ is already a bijection on that orbit; or
\item \label{thm:FMbij:finline}
finite and aperiodic, i.e., $\{x, f(x), \dotsc, f^n(x)\}$ where $x \not\in \im(f)$ and $f^n(x) \not\in \dom(f)$; or
\item \label{thm:FMbij:infline}
bi-infinite, with $f$ already a bijection on that orbit; or
\item \label{thm:FMbij:inf+}
forward-infinite, i.e., $\{x, f(x), f^2(x), \dotsc\}$ where $x \not\in \im(f)$; or
\item \label{thm:FMbij:inf-}
backward-infinite, i.e., $\{\dotsc, f^{-2}(x), f^{-1}(x), x\}$ where $x \not\in \dom(f)$.
\end{enumerate}
On orbits of types \cref{thm:FMbij:periodic,thm:FMbij:infline}, we take $g$ to be $f$; on orbits of type \cref{thm:FMbij:finline}, we take $g$ to be $f$ extended with $f^n(x) |-> x$.
Then $g$ already covers $f$ on these orbits, so we may take $h$ to be the identity.
On orbits of type \cref{thm:FMbij:inf+}, we take $g$ to be the involution swapping $x$ with $f(x)$, $f^2(x)$ with $f^3(x)$, etc., and $h$ to be the involution fixing $x$ and swapping $f(x)$ with $f^2(x)$, $f^3(x)$ with $f^4(x)$, etc; then $f \subseteq g \cup h$ on these orbits.
On orbits of type \cref{thm:FMbij:inf-}, we perform the reverse construction.

Formally, the above-described $g, h$ are defined in terms of $f$ by the formulas
\begin{align*}
g(x) = y
\coloniff&  \Bigl( \sqbr[\big]{\bigvee_{n \in \#N} (x \not\in \dom(f^n)) <-> \bigvee_{m \in \#N} (x \not\in \im(f^m))} \\
    & \qquad \wedge \sqbr[\big]{(f(x) = y) \vee \bigvee_{n \ge 0} ((f^n(y) = x) \wedge (x \not\in \dom(f) \wedge y \not\in \im(f))} \Bigr) \\
{}\vee&  \Bigl( \bigwedge_{n \in \#N} (x \in \dom(f^n)) \wedge \bigvee_{m \in \#N} \sqbr[\big]{ \paren[\big]{x \in \im(f^m) \setminus \im(f^{m+1})}
    \wedge \paren[\big]{f^{(-1)^m}(x) = y} } \Bigr) \\
{}\vee&  \Bigl( \bigwedge_{m \in \#N} (y \in \im(f^n)) \wedge \bigvee_{n \in \#N} \sqbr[\big]{ \paren[\big]{y \in \dom(f^n) \setminus \dom(f^{n+1})}
    \wedge \paren[\big]{f^{(-1)^n}(y) = x} } \Bigr),
\displaybreak[0]\\
h(x) = y
\coloniff&  \Bigl( \sqbr[\big]{\paren[\big]{\bigvee_{n \in \#N} (x \not\in \dom(f^n)) <-> \bigvee_{m \in \#N} (x \not\in \im(f^m))} \vee (x \not\in \dom(f) \cap \im(f))}
    \wedge (x = y) \Bigr) \\
{}\vee&  \Bigl( \bigwedge_{n \in \#N} (x \in \dom(f^n)) \wedge \bigvee_{m \in \#N} \sqbr[\big]{ \paren[\big]{x \in \im(f^m) \setminus \im(f^{m+1})}
    \wedge \paren[\big]{f^{(-1)^m}(y) = x} } \Bigr) \\
{}\vee&  \Bigl( \bigwedge_{m \in \#N} (y \in \im(f^n)) \wedge \bigvee_{n \in \#N} \sqbr[\big]{ \paren[\big]{y \in \dom(f^n) \setminus \dom(f^{n+1})}
    \wedge \paren[\big]{f^{(-1)^n}(x) = y} } \Bigr).
\end{align*}
The resulting interpretation $\@T_\LNbij -> \@T_\LN$ is thus given by $\omega \times \omega$ many copies $g_{ij}, h_{ij}$ of the above formulas for each $i, j < \omega$, with $f$ above replaced by the partial bijections $f_{ij}$.
\end{proof}

\subsection{Edge-colorings and the intersection graph on finite subsets}
\label{sec:FM-color}

The Feldman--Moore theorem is sometimes stated with a more combinatorial flavor, namely, that any CBER $E \subseteq X^2$ has a Borel edge-coloring $c : E \setminus (=_X) -> \omega$ of the complete graph on each $E$-class; see e.g., \cite[3.5]{Kcber}.
Indeed, such a coloring $c$ can be obtained from a countable family of Borel involutions $g_i$ covering $E$ by taking
$c(x, y) := \text{the least $i \in \omega$ such that $g_i(x) = y$}$.
Conversely, the color classes of any such coloring $c$ (extended by the identity to make them total) clearly yield a covering family of involutions.

We formally encode such an edge-coloring $c$ of a complete graph into a theory $\@T_{\coloreq2}$, in a language with binary relations $C_k$ consisting of all edges colored $k$, i.e., $c(x,y) = k \iff C_k(x,y)$; see \cref{tbl:theories} for a precise axiomatization.
The equivalence to Feldman--Moore now amounts to

\begin{proposition}
\label{thm:FM-color2}
$\@T_{\coloreq2}$ and $\@T_\LNinv$ are mutually interpretable.
\qed
\end{proposition}

A well-known lemma due to Kechris--Miller \cite[7.3]{KM} shows how to extend a $\@T_{\coloreq2}$-structuring of a CBER $E$ to a Borel $\omega$-coloring $c$ of the intersection graph on \textit{all} finite subsets of $E$-classes.
We encode such a coloring into a theory $\@T_{\colorlt\omega}$, formally consisting again of $n$-ary relations $C_{nk}$ satisfied by subsets of cardinality $n$ that are colored $k$, i.e., $c(\{x_0,\dots,x_{n-1}\}) = k \iff C_{nk}(x_0,\dots,x_{n-1})$; see \cref{tbl:theories} for a precise axiomatization.
The Kechris--Miller argument amounts to

\begin{proposition}[Kechris--Miller]
\label{thm:KM-color}
There is an interpretation $\@T_{\colorlt\omega} \to \@T_{\coloreq2} \sqcup \@T_\LO$.
\end{proposition}

\begin{proof}
Let $c: Y^2 \setminus (=_Y) \to \omega$ be an edge coloring of the complete graph on a set $Y$ equipped with a linear order.
We define the color of a finite set $\{x_0, \dots, x_{n-1}\}$ by first arranging the elements in increasing order (i.e., choose the unique $\sigma \in \Sym(n)$ such that $x_{\sigma(i)} < x_{\sigma(j)}$ for all $i < j$), and then recording the colors of all pairs, keeping track of this order:
\begin{align*}
    c(\{x_0, \dots, x_{n-1}\}) := ( c(x_{\sigma(i)}, x_{\sigma(j)}) )_{(i,j) \in n^2} \in \#N^{n^2}.
\end{align*}
To verify that this is indeed a coloring, let $(a_0 < \dots < a_{n-1}), (b_0 < \dots < b_{n-1})$ be tuples that intersect (so $a_i = b_j$ for some $i,j \in n$) but have the same color (so $c(a_i, a_j) = c(b_i,b_j)$ for all $i,j \in n$). Then $a_i = b_j \implies c(a_i, a_j) = c(b_i,b_j) = c(b_i,a_i) \implies a_j = b_i$, so $i = j$. But then for any $k \in n$, $c(a_i, a_k) = c(b_i,b_k) = c(a_i,b_k) \implies a_k = b_k$, so in fact $\vec{a} = \vec{b}$.

Fixing bijections $b_n: \#N^{n^2} \cong \#N$ for each $n$, we can therefore interpret $\@T_{\colorlt\omega}$ in $\@T_{\coloreq2} \sqcup \@T_\LO$ by
\begin{align*}
    C_{n, b_n(\gamma)}(x_0, \dots, x_{n-1})
    \coloniff \bigvee_{\sigma \in \Sym(n)} \paren[\big]{\bigwedge_{i < j} (x_{\sigma(i)} < x_{\sigma(j)}) 
    \wedge \bigwedge_{i,j \in n} (c(x_{\sigma(i)}, x_{\sigma(j)}) = \gamma_{ij})}
\end{align*}
for every $n \in \#N, \gamma \in \#N^{n^2}$.
\end{proof}

Together with \cref{thm:FM-color2,thm:FM,eq:finsub-pt:interp-LO}, this yields $\@T_{\colorlt\omega} -> \@T_\LN \sqcup \@T_\sep$, i.e., every CBER $E$ is $\@T_{\colorlt\omega}$-structurable, i.e., has a Borel $\omega$-coloring of the intersection graph of finite subsets of $E$-classes, as the result is usually stated.

Note that the argument above uses $\@T_\LO$ rather than $\@T_\sep$; the former theory is strictly weaker than the latter, by \cref{ex:Zline-cex} below.
Nonetheless, in the presence of $\@T_{\coloreq2}$, they turn out to be mutually interpretable:

\begin{proposition}
\label{thm:sep-colorLO}
    There is an interpretation $\@T_\sep \to \@T_{\coloreq2} \sqcup \@T_\LO$.
\end{proposition}

\begin{proof}
    Define for each $k \in \#N$ a formula 
    \begin{align*}
        U_k(x) \coloniff \exists z ((x < z) \wedge (c(x,z) = k)).
    \end{align*}
    It suffices to show that the $U_k$'s separate points in any model $(c, \leq)$ of $\@T_{\coloreq2} \sqcup \@T_\LO$ (where $c$ is the coloring of pairs).
    So, let $x \neq y$ and say without loss of generality that $x < y$.
    Then for $k = c(x,y)$ we have $U_k(x)$ (choose $z = y$), but $\neg U_k(y)$, as $c(y,z) = k = c(x,y) \implies z = x < y$.
\end{proof}

In fact, it is ``almost'' true that $\@T_\sep -> \@T_{\colorlt\omega}$ (which implies the above by \cref{thm:KM-color}).
The only counterexample is a model of $\@T_{\colorlt\omega}$ with only 2 indistinguishable points; see \cref{ex:LO-color-cex}.
But in models of size $3$ or larger, we may define a countable separating family from a coloring of finite subsets, or indeed a coloring of subsets of size $\le 3$, whose theory we call $\@T_{\colorle3}$ (which can be axiomatized similarly to $\@T_{\colorlt\omega}$ in \cref{tbl:theories}).

\begin{proposition}
\label{thm:sep-color}
    There is an interpretation $\@T_\sep \to \@T_{\colorle3} \sqcup \{\exists x_0, x_1, x_2 \bigwedge_{i \neq j \in 3} (x_i \neq x_j) \}$.
\end{proposition}

\begin{proof}
    Define for each $i,j \in \#N$ a formula
    \begin{align*}
        U_{ij}(x) := \exists z, z' ( (c(x,z) = i) \wedge (c(x,z,z') = j) ).
    \end{align*}
    Let $c$ be an $\omega$-coloring of the subsets of size $\le 3$ of some underlying set with $\ge 3$ elements.
    Let $x \neq y$ and fix any $w \neq x,y$.
    Put $i = c(x,w)$ and $j = c(x,w,y)$, so that $U_{ij}(x)$.
    If $U_{ij}(y)$ as well, then there exist $z,z'$ such that $c(y,z) = i$ and $c(y,z,z') = j$.
    But then $c(y,z,z') = j = c(x,w,y) \implies \{z,z'\} = \{x,w\}$, so $\{x,w\}$ intersects $\{y,z\}$, and $c(x,w) = i = c(y,z) \implies \{x,w\} = \{y,z\}$, contradicting that $x,y,w$ are all distinct.
\end{proof}

\begin{corollary}
\label{thm:color3}
$\@T_{\colorle3}$ and $\@T_{\colorlt\omega}$ are mutually interpretable.
\end{corollary}
\begin{proof}
Clearly $\@T_{\colorle3} -> \@T_{\colorlt\omega}$; conversely, we have $\@T_{\colorlt\omega} -> \@T_{\colorle3}$ using the composite of $\@T_{\colorlt\omega} -> \@T_{\coloreq2} \sqcup \@T_\LO$ from \cref{thm:KM-color}, $\@T_\LO -> \@T_\sep$ from \cref{ex:finsub-pt:LO}, and $\@T_\sep -> \@T_\colorle3$ from \cref{thm:sep-color} in models of size $\ge 3$, and the trivial $\@T_\colorlt\omega$-coloring in models of size $\le 2$.
\end{proof}

\Cref{fig:cber-interp-big} shows the interpretability relations between the theories we have considered thus far.
Note that these theories all describe structures available ``for free'' on every CBER.
The strongest such theory is $\@T_\LN \sqcup \@T_\sep$, by \cref{thm:str-impl-interp}.
It follows from \cref{thm:sep-colorLO} that this theory is mutually interpretable with $\@T_\coloreq2 \sqcup \@T_\LO$, which hence also interprets every theory available for free on every CBER.
Moreover, by \cref{thm:sep-color}, the theory $\@T_\colorlt\omega$ of countable colorings of finite subsets is ``almost'' equivalent in strength as well; we indicate this with the `$\approx$' in \cref{fig:cber-interp-big}.

We now verify that no other interpretability relations (not implied by \cref{fig:cber-interp-big} and transitivity) hold between these theories.
Our primary tool for showing the non-existence of an interpretation $\alpha : \@T \to \@T'$ is to examine the possible automorphisms of models:
recall \cref{eq:interp-mod} that $\alpha$ induces a ``reduct'' map $\alpha^*: \Mod(\@T') \to \Mod(\@T)$ equivariant under the logic action, hence in particular
\begin{equation*}
\Aut(\@M) \subseteq \Aut(\alpha^* \@M)
    \quad \forall \@M |= \@T'.
\end{equation*}
So to show $\@T \not\to \@T'$, it suffices to produce a model $\@M$ of $\@T'$ along with an automorphism $g \in \Aut(\@M)$ that cannot be an automorphism of any model of $\@T$.
More generally, it suffices to have a model $\@M$ of $\@T'$ such that for any model $\@N$ of $\@T$ on the same set, there is some $g \in \Aut(\@M) \setminus \Aut(\@N)$.

We minimize the number of non-interpretations needed to verify \cref{fig:cber-interp-big} by noting that it is enough to show non-interpretations from weak theories to strong theories. That is, to prove $\@T_1 \not\to \@T_2$, it suffices to prove $\@T_1' \not\to \@T_2'$, for any $\@T_1' -> \@T_1$ and any $\@T_2' <- \@T_2$.
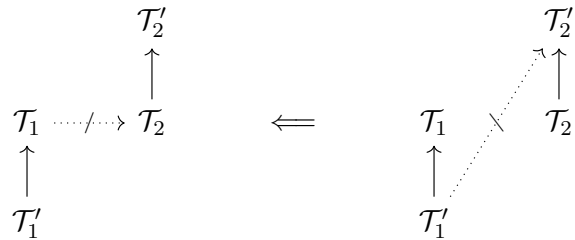
\begin{center}
\begin{tikzcd}
                        & \@T'_2          &            &                                      & \@T'_2          \\
\@T_1 \arrow[r, dotted, "/"{anchor=center,sloped}] & \@T_2 \arrow[u] & \impliedby & \@T_1                                & \@T_2 \arrow[u] \\
\@T'_1 \arrow[u]        &                 &            & \@T'_1 \arrow[u] \arrow[ruu, dotted, "/"{anchor=center,sloped}] &                
\end{tikzcd}
\end{center}
Here and in \cref{fig:cber-interp-big}, dotted arrows between theories indicate non-existence of interpretations. 
We also do not need to check non-interpretations from coproduct theories, as (by the universal property of the coproduct, see \cref{def:thy-coprod}) an interpretation $\@T_1 \sqcup \@T_2 \to \@T$ is equivalent to a pair of interpretations $\@T_1 \to \@T <- \@T_2$. 
It suffices then to verify the six non-interpretations shown in \cref{fig:cber-interp-big} (below the dashed line, and excluding $\@T_\marker \oplus \@T_\pt$ for which see \cref{sec:examples-other}):

\begin{proposition}
\label{ex:LO-color-cex}
    There is no interpretation $\@T_\LO \to \@T_\colorlt\omega$.
\end{proposition}
\begin{proof}
    We may trivially color the subsets of $2 = \{0,1\}$ so that we have an automorphism flipping $0$ and $1$, which clearly is not an automorphism of any linear order on $2$.
\end{proof}

The preceding model of size $2$ is the only possible witness to $\@T_\LO -/> \@T_\colorlt\omega$, by \cref{thm:sep-color}.
Since finite equivalence classes of CBERs are usually considered ``trivial'', in order for a non-interpretation to be viewed as comparing the relative strengths of Borel combinatorial structures, it would be preferable to have infinite models as counterexamples, which we provide in the following.

\begin{proposition} \label{ex:color2-cex}
There is no interpretation $\@T_\LO \to \@T_{\coloreq2}$.
\end{proposition}
\begin{proof}
The complete graph on $\#Q - \{0\}$ with the coloring $c(x,y) := xy$ has an automorphism $x \mapsto -x$, which is not an automorphism of any linear order on $\#Q - \{0\}$.
\end{proof}

\begin{remark} \label{ex:color3-cex}
Since the above counterexample is an infinite model, it follows from \cref{thm:sep-color} that also $\@T_\colorle3 -/> \@T_\coloreq2$.
\end{remark}

\begin{proposition} \label{ex:Zline-cex}
$\@T_\FMinv \sqcup \@T_\LO$ does not interpret $\@T_\sep$ or $\@T_{\coloreq2}$.
\end{proposition}
\begin{proof}
The connected graph $xGy \coloniff |x - y| = 1$ on $(\#Z, \leq)$ with the edge coloring $c(n,m) := \min(n,m) \bmod 2$ yields a model of $\@T_\FMinv \sqcup \@T_\LO$, where the $i$th generator in $\#Z_2^{*\omega}$ flips edges of color $i$, and it has a nontrivial automorphism $h: x \mapsto x+2$.
But models of $\@T_\sep$ are rigid, so there is no interpretation $\@T_\sep \to \@T_\FMinv \sqcup \@T_\LO$.

Moreover, $h$ cannot be an automorphism of an edge coloring $c'$ of the complete graph on $\#Z$, since we would have $c'(h(0), h(2)) = c'(2,4) \neq c'(0,2)$.
So there is also no interpretation $\@T_\coloreq2 \to \@T_\FMinv \sqcup \@T_\LO$.
\end{proof}

\begin{proposition} \label{ex:Ztrans-cex}
There is no interpretation $\@T_\FMinv \to \@T_\LN \sqcup \@T_\LO$.
\end{proposition}
\begin{proof}
Consider $\#Z$ equipped with the usual linear order $\leq$ and the translation action $a: (\#Z,+) \actson \#Z$.
Then $\@M = (\#Z,a,\leq)$ is a model of $\@T_\LN \sqcup \@T_\LO$. Note that automorphisms of $\@M$ are translations by elements of $\#Z$ and therefore all automorphism orbits are infinite.
On the other hand, for any transitive action $b : \#Z_2^{*\omega} \actson \#Z$, there is some order-$2$ generator $\gamma \in \#Z_2^{*\omega}$ with non-trivial action $b_\gamma \neq \id$.
So in particular, there exists $x \in \#Z$ with $b_\gamma(x) \neq x$, but then $b_\gamma$ is not preserved by the automorphism of $\@M$ that translates by $b_\gamma(x) - x$.
\end{proof}

\begin{proposition}
There is no interpretation $\@T_\LN \to \@T_\sep$.
\end{proposition}
\begin{proof}
$\@T_\sep$ has $2^\#N$ as a model, whereas all models of $\@T_\LN$ are countable.
\end{proof}

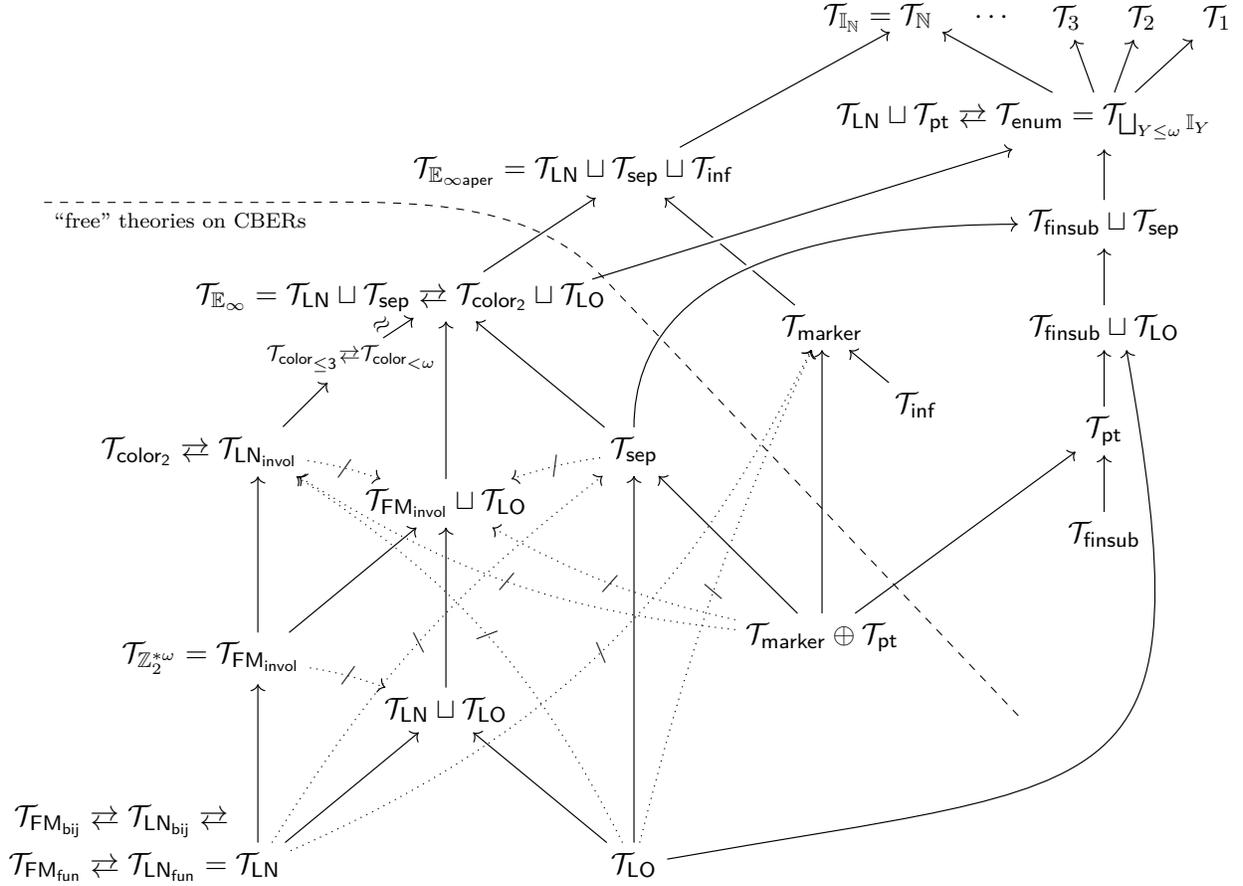
\begin{figure}
\centering
\hphantom{$\@T_\FMbij \rightleftarrows \@T_\LNbij \;\;$}
\begin{tikzpicture}[x=2.5cm, y=2.75cm]

\node[theory] (LNsep)
{
    \mathllap{
    \@T_{\#E_\infty} =
    {}}
    \@T_\LN \sqcup \@T_\sep \rightleftarrows
    \@T_{\coloreq2} \sqcup \@T_\LO
};

\node[theory] (color)
    at ($(LNsep) + (-0.5, -0.3)$)
{
    \scriptstyle
    \@T_{\colorle3} \rightleftarrows
    \@T_{\colorlt\omega}
}
    edge[
        interp,
        shorten <=-.5ex,
        shorten >=-.5ex,
        "\approx"{sloped, pos=0},
    ] (LNsep);

\node[theory] (LNinv)
    at ($(LNsep) + (-1, -.75)$)
{
    \mathllap{
    \@T_{\coloreq2} \rightleftarrows
    {}}
    \@T_\LNinv
}
    edge[interp] (color);

\node[theory] (FMLO)
    at ($(LNsep) + (0, -1)$)
{
    \@T_\FMinv \sqcup \@T_\LO
}
    edge[interp] (LNsep);

\node[theory] (FMinv)
    at ($(LNinv) - (LNsep) + (FMLO)$)
{
    \mathllap{
    \@T_{\#Z_2^{*\omega}} =
    {}}
    \@T_\FMinv
}
    edge[interp] (LNinv)
    edge[interp] (FMLO);

\node[theory] (LNLO)
    at ($(FMLO) - (LNsep) + (FMLO)$)
{
    \@T_\LN \sqcup \@T_\LO
}
    edge[interp] (FMLO);

\node[theory] (LN)
    at ($(FMinv) - (FMLO) + (LNLO)$)
{
    \mathllap{\smash{\begin{aligned}[b]
    \@T_\FMbij &\rightleftarrows
    \@T_\LNbij \rightleftarrows \\
    \@T_\FMfun &\rightleftarrows
    \@T_\LNfun =
    {}\end{aligned}}}
    \@T_\LN
}
    edge[interp] (FMinv)
    edge[interp] (LNLO);


\node[theory] (sep)
    at ($(LNsep) + (1, -.75)$)
{
    \@T_\sep
}
    edge[interp] (LNsep);

\node[theory] (LO)
    at (sep |- LN)
{
    \@T_\LO
}
    edge[interp] (sep)
    edge[interp] (LNLO);


\node[theory] (LNsepinf)
    at ($(LNsep) + (1, 0.6)$)
{
    \mathllap{
    \@T_{\#E_{\infty\!{aper}}} =
    {}}
    \@T_\LN \sqcup \@T_\sep \sqcup \@T_\inf
}
    edge[interp, <-] (LNsep);

\node[theory] (marker)
    at ($(sep) + (1, 0.6)$)
{
    \@T_\marker
}
    edge[interp] (LNsepinf);

\node[theory] (marker or pt)
    at ($(marker) + (0, -1.5)$)
{
    \@T_\marker \oplus \@T_\pt
}
    edge[interp] (sep)
    edge[interp] (marker);


\node[theory] (inf)
    at ($(LNsepinf)!1.5!(marker)$)
{
    \@T_\inf
}
    edge[interp, shorten <=-.0ex] (marker);


\node[theory] (pt)
    at ($(marker or pt) + (1.5, 1)$)
{
    \@T_\pt
}
    edge[interp, <-] (marker or pt);

\node[theory] (finsub)
    at ($(pt) + (0, -0.5)$)
{
    \@T_\finsub
}
    edge[interp] (pt);

\node[theory] (LOpt)
    at ($(pt) + (0, 0.5)$)
{
    \@T_\finsub \sqcup \@T_\LO
}
    edge[interp, <-] (pt);
\node also (LO)
    edge[interp, out=10, in=-80, looseness=1.75] (LOpt.-45);

\node[theory] (seppt)
    at ($(LOpt) + (0, 0.5)$)
{
    \@T_\finsub \sqcup \@T_\sep
}
    edge[interp, <-] (LOpt);
\node also (sep)
    edge[interp, out=90, in=180, looseness=1, cross] (seppt);

\node[theory] (N)
    at ($(seppt) + (-1, 1)$)
{
    \mathllap{
    \@T_{\#I_\#N} =
    {}}
    \@T_\#N
}
    edge[interp, <-] (LNsepinf);

\node[theory] (ndots)
    at ($(N) + (0.4, 0)$)
{
    \dotsm
};

\node[theory] (n3)
    at ($2*(ndots) - (N)$)
{
    \@T_3
};

\node[theory] (n2)
    at ($2*(n3) - (ndots)$)
{
    \@T_2
};

\node[theory] (n1)
    at ($2*(n2) - (n3)$)
{
    \@T_1
};

\node[theory] (enum)
    at ($(seppt) + (0, 0.5)$)
{
    \mathllap{
    \@T_\LN \sqcup \@T_\pt \rightleftarrows
    {}}
    \@T_\enum
    = \@T_{\bigsqcup_{Y \le \omega} \#I_Y}
}
    edge[interp, <-] (seppt)
    edge[interp, <-, shorten >=-2ex, cross] (LNsep.north east)
    edge[interp] (N)
    edge[interp] (n3)
    edge[interp] (n2)
    edge[interp] (n1);


\useasboundingbox;

\draw[dashed, use Hobby shortcut, shorten >=-5em]
    ($(LNsep.north west) + (-8em,2.4em)$)
        node(s)[coordinate]{}
        node[below right]{\scriptsize``free'' theories on CBERs}
    -- ([out angle=0] s -| LNsep.west)
    .. ([tension in=0.25, tension out=5]$(LNsep.north east) + (-1ex,-.5ex)$)
    .. ([tension in=5]$(marker or pt.north east) + (1ex,0ex)$);


\node also (LN)
    edge[noninterp, bend left=10] (sep)
    edge[noninterp, bend right=20] (marker);

\node also (FMinv)
    edge[noninterp, bend left=5] (LNLO);

\node also (LNinv)
    edge[noninterp, bend left=5] (FMLO);

\node also (LO)
    edge[noninterp, bend right=15] (LNinv)
    edge[noninterp, bend left=5] (marker);

\node also (sep)
    edge[noninterp, bend right=5] (FMLO);

\node also (marker or pt)
    edge[noninterp, bend left=13] (LNinv)
    edge[noninterp, bend left=8] (FMLO);

\end{tikzpicture}
\caption{Interpretability relations between some $\@L_{\omega_1\omega}$ theories commonly used to structure CBERs (flipped version of right half of \cref{fig:cber-interp}), along with non-interpretabilities witnessed by models (see text).}
\label{fig:cber-interp-big}
\end{figure}

\subsection{Other examples}
\label{sec:examples-other}

Another important result in the theory of CBERs is the Slaman--Steel marker lemma \cite[Lemma~1]{SlStr} (see also \cite[7.1.5]{GaoIDST}), which says that every aperiodic CBER $E$ (i.e., every $E$-class is infinite) has a \defn{vanishing sequence of markers}: a decreasing sequence of Borel subsets $X \supseteq A_0 \supseteq A_1 \supseteq A_2 \supseteq \dotsb$ which are nonempty in each $E$-class but have $\bigcap_{n \in \omega} A_n = \emptyset$.
We axiomatize such a sequence by a theory $\@T_\marker$; see \cref{tbl:theories}.
The proof of the marker lemma amounts to:

\begin{proposition}[Slaman--Steel]
\label{thm:marker}
There is an interpretation
$\@T_\marker -> \@T_\LN \sqcup \@T_\sep \sqcup \@T_\inf$.
\end{proposition}

Here, as in \cref{ex:einf-aper}, $\@T_\inf$ asserts that the underlying set is infinite; see again \cref{tbl:theories}.

\begin{proof}
Let $\@M = (M, (f_i)_{i \in \#N}, (U_i)_{i \in \#N}) |= \@T_\LN \sqcup \@T_\sep \sqcup \@T_\inf$.
Recall from \cref{ex:finsub-pt:LO} that the $U_i$'s correspond to an injection $u: M `-> 2^\#N$, along which we can pull back the lexicographical ordering to $M$:
\begin{align*}
    x <_\!{lex} y & \coloniff \bigvee_i (\neg U_i(x) \wedge U_i(y) \wedge \bigwedge_{j < i} (U_j(x) <-> U_j(y))).
\end{align*}
Similarly, for any $n \in \omega$ we may define a lexicographical pre-order comparing the first $n$ bits:
\begin{align*}
    x <^n_\!{lex} y & \coloniff \bigvee_{i < n} (\neg U_i(x) \wedge U_i(y) \wedge \bigwedge_{j < i} (U_j(x) <-> U_j(y))).
\end{align*}

Now, there are two cases to consider.
In Case $1$, $M$ contains a $\leq_\lex$-least element, i.e., an element $y \in M$ such that $\@M \models \forall z (y \leq_\lex z)$.
Then the Lusin--Novikov functions $f_i$ give an enumeration of $M$, namely $(f_i(y))_{i \in \omega}$, which we may disjointify into a bijective enumeration $(g_i(y))_{i \in \omega}$ using $\@T_\inf$:
\begin{align}
\label{eq:marker-enum}
    g_i(y) = x  &\coloniff  \bigvee_{j < \omega} \paren[\big]{(f_j(y) = x) \wedge \text{``$j$ is the $i$th natural s.t.\ $\bigwedge_{k < j} (f_j(y) \ne f_k(y))$''}}
\end{align}
The final segments of this enumeration then form a vanishing sequence of markers $(A_n)_n$:
\begin{align*}
A_n(x)  \coloniff  \exists y\, (\forall z\, (y \leq_\lex z) \wedge \bigvee_{i \ge n} (g_i(y) = x)).
\end{align*}

In Case $2$, $M$ contains no $\leq_\lex$-least element.
Then there is some $\leq_\lex$-least element $\xi \in \overline{u(M)}$, which is not in $u(M)$.
We may then define the marker sequence $(A_n)_n$ by taking the basic clopen neighborhoods of $\xi \in 2^\#N$, pulled back along $u$.
The $n$th basic neighborhood of $\xi$ consists of strings agreeing with $\xi$ up to the first $n$ bits; this holds for $u(x)$ iff $u(x)|n \leq_\lex u(z)|n$ for every $z \in M$.
So
\begin{align*}
    A_n(x) \coloniff \forall z\, (x \le^n_\lex z).
\end{align*}

Putting the two cases together, we can interpret $\@T_\marker$ in $\@T_\LN \sqcup \@T_\sep \sqcup \@T_\inf$ by
\begin{align*}
    A_n(x) \coloniff
    & \exists y\, (\forall z\, (y \leq_\lex z) \wedge \bigvee_{i \ge n} (g_i(y) = x)) \\
    & \vee [\neg \exists y\, \forall z\, (y \leq_\lex z) \wedge \forall z\, (x \le^n_\lex z)].
    \qedhere
\end{align*}
\end{proof}

Note that in this argument, we only used the Lusin--Novikov functions in Case $1$ in order to define an enumeration from a point.
If we stop Case $1$ after defining the $\leq_\lex$-least point, the argument now shows that from $\@T_\sep$ alone, we may define \emph{either} a marker sequence, \emph{or} a single point, i.e., a model of the product theory $\@T_\marker \oplus \@T_\pt$ (recall \cref{def:thy-prod}),
thereby allowing us to encompass the main content of the marker lemma within the ``free'' region of \cref{fig:cber-interp-big}:

\begin{corollary}[of proof]
There is an interpretation $\@T_\marker \oplus \@T_\pt \to \@T_\sep$.
\qed
\end{corollary}

The formulas \cref{eq:marker-enum} in Case $1$ above show more generally that we may define from a single distinguished point $y$ and a family of Lusin--Novikov functions $(f_i)_i$ a bijective $\omega$-enumeration, provided the underlying set is infinite; in other words, $\@T_\LN \sqcup \@T_\pt \sqcup \@T_\inf$ interprets the theory $\@T_\omega$ of $\omega$-enumerated sets from \cref{ex:interp-mod}.
It is also easily seen that if the underlying set has finite size $Y < \omega$, then the same formulas yield an interpretation of $\@T_Y$.
Thus, letting $\@T_\enum := \bigoplus_{Y \leq \omega} \@T_Y$ be the \defn{theory of enumerated sets}, Case $1$ of the above argument essentially shows that

\begin{proposition}
\label{thm:enum-LNpt}
There is an interpretation $\@T_\enum \to \@T_\LN \sqcup \@T_\pt$.
\qed
\end{proposition}

This formalizes a common pattern in many Borel combinatorics arguments, of which the above proof of the marker lemma is an example: one attempts to define some kind of structure (axiomatized by a theory $\@T$) on all classes of a CBER, which fails on a set on which the CBER is smooth (so one has only defined an interpretation of $\@T \oplus \@T_\pt$), on which one instead enumerates every class using Lusin--Novikov (yielding an interpretation of $\@T \oplus \@T_\enum$ by composing with the preceding result), allowing one to easily define any structure in a Borel way.
This last step is made precise by

\begin{proposition}
\label{thm:enum-arb}
$\@T_\enum$ interprets any theory $\@T$ with models of every countable cardinality.
\end{proposition}
\begin{proof}
For each $Y \le \omega$, take a model of $\@T$ on $Y$, which gives by \cref{ex:interp-mod} an interpretation $\@T -> \@T_Y$; these combine into an interpretation into the product theory $\@T -> \bigoplus_{Y \le \omega} \@T_Y = \@T_\enum$.
\end{proof}

It follows that $\@T_\enum$ is mutually interpretable with $\@T_\LN \sqcup \@T_\pt$, as shown in \cref{fig:cber-interp-big}.

\begin{remark}
The four theories immediately below $\@T_\enum$ in \cref{fig:cber-interp-big} all yield smoothness on CBERs, as discussed in the introduction (see \cref{ex:finsub-pt}).
As noted there, in the absence of $\@T_\LN$, they are strictly increasing in strength, and strictly weaker than $\@T_\enum$, by arguments like those in the preceding subsection: for example, $\@T_\finsub \sqcup \@T_\LO \rightleftarrows \@T_\pt \sqcup \@T_\LO$ by \cref{ex:finsub-pt}, but does not interpret $\@T_\enum$ since it has non-rigid models.
\end{remark}

\cref{fig:cber-interp-big} also shows how $\@T_\marker$ and $\@T_\marker \oplus \@T_\pt$ relate to the other theories describing ``free'' structures on CBERs considered in the preceding subsections.
Note that no such ``free'' theory interprets $\@T_\marker$, since not every CBER is aperiodic.
Conversely, none of the free theories in \cref{fig:cber-interp-big} are interpreted by $\@T_\marker$ (hence also not by $\@T_\marker \sqcup \@T_\pt$): it suffices to show

\begin{proposition}
There are no interpretations $\@T_\LO \to \@T_\marker$ or $\@T_\LN \to \@T_\marker$.
\end{proposition}

\begin{proof}
Note that in a model $\@M = (M, A_0, A_1, \dotsc)$ of $\@T_\marker$, there may be points that are not in the first marker set $A_0$.
Any permutation $g \in \Sym(M \setminus A_0)$ extends to an automorphism of $\@M$ (e.g., define $g(x) = x$ for $x \in A_0$).
If there are at least $2$ points $x \neq y$ in $M \setminus A_0$, then the automorphism swapping $x,y$ shows that $\@T_\LO \not\to \@T_\marker$.
Similarly, if there are $3$ distinct points $x, y, z \in M \setminus A_0$, then an automorphism $g$ fixing $x$ and swapping $y,z$ cannot be an automorphism of a model of $\@T_\LN$, since a Lusin--Novikov function $f_i$ cannot map $x$ to both $y, z$.
\end{proof}

\begin{remark}
A model $\@M |= \@T_\marker$ is essentially just a labeled countable partition (the differences of adjacent marker sets) with infinitely many nonempty pieces.
These pieces are the automorphism orbits of $\@M$; within each piece, an automorphism can act arbitrarily.
So more generally, in order for a theory $\@T$ to be interpretable in $\@T_\marker$, it must admit a model $\@M$ with a countable partition within which all permutations are automorphisms of $\@M$.
\end{remark}

Using this observation, we can also see that, while the marker lemma can essentially (modulo $\@T_\pt$) be proved from $\@T_\sep$, it cannot be proved from any of the other ``free'' theories in \cref{fig:cber-interp-big}:

\begin{proposition}
There is no interpretation $\@T_\marker \oplus \@T_\pt \to \@T_\FMinv \sqcup \@T_\LO$.
\end{proposition}

\begin{proof}
Consider the structure $\@M = (\#Z, \leq, c)$ described in \cref{ex:Zline-cex}, which has only two automorphism orbits and therefore cannot have a model of $\@T_\marker$ as a reduct.
Moreover, $\@M$ has automorphisms with no fixed points, so cannot have a model of $\@T_\pt$ as a reduct.
\end{proof}

\begin{proposition}
There is no interpretation $\@T_\marker \oplus \@T_\pt \to \@T_\coloreq2$.
\end{proposition}

\begin{proof}
Consider the complete graph on $\#Z_2^{\oplus \omega}$ with edge coloring $c(x,y) := x + y$.
This is a model of $\@T_\coloreq2$ with only one automorphism orbit, since translation by each $a \in \#Z_2^{\oplus \omega}$ is an automorphism, so its reducts can model neither $\@T_\pt$ nor $\@T_\marker$.
\end{proof}

\begin{remark}
Similarly, there is no interpretation $\@T_\marker \oplus \@T_\pt \to \@T_\finsub$ (as shown in \cref{fig:cber-interp-big}), since any model $\@M$ of $\@T_\finsub$ has only two automorphism orbits, and if the finite set $D$ contains at least two points then $\@M$ will have automorphisms with no fixed points.

It is also easy to see $\@T_\LO -/> \@T_\pt$.
Together with the aforementioned non-interpretations, it follows that (the transitive closure of) \cref{fig:cber-interp-big} gives all interpretabilities between the shown theories.
\end{remark}

We close this section with a simple example of a different nature: a Borel combinatorial construction encoded not by structurability via a theory $\@T$, but rather \emph{expandability} along an interpretation $\@T -> \@T'$, as in \cref{thm:str-expan-interp}.

A basic result in Borel combinatorics, due to Kechris--Solecki--Todorcevic \cite[4.5]{KST}, shows that every locally finite Borel graph has a Borel $\omega$-coloring (of the vertices).
We can formalize this by defining a theory $\@T_\!{lfgraph}$ of locally finite graphs (in the language $\@L_\graph$ with an edge relation $G$), as well as an expanded theory $\@T_\!{lfgraph+\omega color}$ of \emph{$\omega$-colored} locally finite graphs; see \cref{tbl:theories}.
The result now states that given a Borel locally finite graph $G \subseteq X^2$, which we may treat as a $\@T_\!{lfgraph}$-structuring of any CBER $E \supseteq G$ (e.g., the connectedness relation of $G$), there exists an expansion to a $\@T_\!{lfgraph+\omega color}$-structuring.
Via \cref{thm:str-expan-interp}, this is witnessed by

\begin{proposition}[Kechris--Solecki--Todorcevic]
\label{thm:kst-lfcolor}
There is an interpretation $\@T_\!{lfgraph+\omega color} \to \break \@T_\!{lfgraph} \sqcup \@T_\sep$, whose restriction to the language $\@L_\graph$ is the identity $\@T_\!{lfgraph} -> \@T_\!{lfgraph} \sqcup \@T_\sep$.
\end{proposition}

\begin{proof}
Let $\@M = (G, U_i)_{i \in \omega} |= \@T_\!{lfgraph} \sqcup \@T_\sep$, a locally finite graph $G$ together with a separating family $(U_i)_i$.
Then for any $xGy$ there is some least $i \in \#N$ such that $U_i$ separates $x$ and $y$, call it $i(x,y)$.
For each $x \in M$, let $A_x = \set{i(x,y)}{yGx,\, x \in U_{i(x,y)}}$, a finite subset of $\#N$ by $\@T_\!{lfgraph}$.
If $xGy$, then clearly $i(x,y) \in A_x \triangle A_y$.
So fixing a bijection $b: \@P_\!{fin}(\#N) \cong \#N$, we can define a coloring of $G$ by
\begin{align*}
C_{b(A)}(x)  \coloniff
\bigwedge_{i \in \#N} \paren[\big]{(i \in A) <-> \exists y G x\, [U_i(x) \wedge \neg U_i(y) \wedge \bigwedge_{j < i} (U_j(x) <-> U_j(y))]}
\end{align*}
(where ``$i \in A$'' denotes the propositional constant $\top$ if $i \in A$, else $\bot$).
\end{proof}

\begin{remark}
The above interpretation is only interesting given the last condition on the restriction to $\@L_\graph$, which ensures that it specifies a coloring of the \emph{original} graph.
(There is trivially an interpretation $\@T_\!{lfgraph+\omega color} -> \emptyset$, since we may take the empty graph.)
Thus, it does not make sense to draw this interpretation as part of \cref{fig:cber-interp-big}.
\end{remark}

\section{Structurability of groupoids}
\label{sec:gpd}

We now generalize the correspondence given by \cref{thm:scott-LNsep} between CBERs and theories interpreting $\@T_\LN \sqcup \@T_\sep$, to theories interpreting just $\@T_\LN$.
We will show that these correspond to locally countable Borel groupoids, which admit a corresponding theory of ``structurability''.
Since this is a less well-studied concept than for CBERs, and since the correspondence between $\@T_\LN$ and groupoids is a bit more involved, we will approach things in a different order here than in \cref{sec:cber}.
First we will introduce groupoids and their ``Scott theories'', and show that they are precisely the theories interpreting $\@T_\LN$, and then we will introduce structurability and discuss some examples.

\subsection{Groupoids}

\begin{definition}
\label{def:gpd}
A \defn{groupoid} $(X,G,\dom,\cod,\id,{\circ},{}^{-1})$ is a category with inverses, consisting of:
\begin{itemize}
\item  a collection $X$ of \defn{objects};
\item  a collection $G$ of \defn{morphisms} or \defn{arrows};
\item  two maps $\dom, \cod : G \rightrightarrows X$ (\defn{domain} and \defn{codomain}), where for $g \in G$ with $\dom(g) = x$ and $\cod(g) = y$, we write $g : x -> y$, and we write $G(x,y) := \dom^{-1}(x) \cap \cod^{-1}(y)$ for the \defn{hom-set} of all such $g : x -> y$;
\item  a map $x |-> \id_x = 1_x : X -> G$ (\defn{identity}), such that $1_x : x -> x$;
\item  a map ${\circ} : G \times_X G := \{(g,h) \in G \times G \mid \dom(g) = \cod(h)\} -> G$ (\defn{composition}), taking $g : y -> z$ and $h : x -> y$ to $g \circ h : x -> z$, obeying the associativity and identity laws;
\item  a map $g |-> g^{-1} : G -> G$ (\defn{inverse}), such that for $g : x -> y$, $g^{-1} \circ g = 1_x$ and $g \circ g^{-1} = 1_y$.
\end{itemize}
A \defn{standard Borel groupoid} is a groupoid such that $X,G$ are standard Borel spaces and $\dom, \cod, {\circ}, 1, {}^{-1}$ are Borel maps.
A \defn{locally countable Borel groupoid} is a standard Borel groupoid such that $\dom$ is countable-to-1; equivalently, each hom-set as well as each connected component is countable.%
\footnote{Note that this conflicts with another common usage of ``locally'' in category theory, to mean that each hom-set obeys said condition (e.g., \emph{locally small category}).}

A \defn{functor} $f : (X,G) -> (Y,H)$ between groupoids is a homomorphism of groupoids, i.e., a pair of maps $f : X -> Y$ and $f : G -> H$ preserving all of the groupoid structure.

For background on category theory and groupoids, see \cite{MacLane}, \cite{Leinster}; for groupoids in the topological and Borel contexts, see \cite{Ramsay}, \cite{Alvarez}, \cite{Carderi}, \cite{Bowen}, \cite{Cgpd}, \cite{TW}.
\end{definition}

\begin{example}
\label{ex:gpd-cber}
Each CBER $(X,E)$ is a locally countable Borel groupoid, with $\dom, \cod : E \rightrightarrows X$ given by the projections.
Up to isomorphism, these are precisely the groupoids such that each hom-set has at most one element (sometimes called \emph{thin} groupoids), or equivalently, each isotopy group $G(x,x)$ is trivial.
\end{example}

\begin{example}
\label{ex:gpd-grp}
A one-object locally countable groupoid is just a countable group.
\end{example}

\begin{example}
\label{ex:grp-action}
More generally, given a countable group $\Gamma$ and a Borel action $\Gamma \actson X$, the \defn{action groupoid} $\Gamma \ltimes X$ on the space of objects $X$ has, informally, morphisms $x -> y$ consisting of all group elements $\gamma \in \Gamma$ such that $\gamma \cdot x = y$.
Formally, such a morphism is encoded by the pair $(\gamma,x)$ with $y = \gamma \cdot x$.
We thus take $\Gamma \ltimes X := \Gamma \times X$, with $\dom : \Gamma \ltimes X -> X$ given by the projection and $\cod$ given by the action map.
The groupoid operations are given by
\begin{align}
\label{eq:grp-action}
\id_x &:= (1_\Gamma,x), &
(\delta,\gamma x) \circ (\gamma,x) &:= (\delta\gamma,x), &
(\gamma,x)^{-1} &:= (\gamma^{-1},\gamma x).
\end{align}
When $X = 1$ with the trivial $\Gamma$-action, this recovers the previous example of a one-object groupoid.

In general, for any countable group action $\Gamma \actson X$, we have two canonical functors
\begin{equation*}
\begin{tikzcd}
& (X, \Gamma \ltimes X)
    \dlar["{\gamma \mapsfrom (\gamma,x)}"']
    \drar["{(\gamma,x) \mapsto (x,\gamma x)}"] \\
(1, \Gamma) && (X, \#E_\Gamma^X)
\end{tikzcd}
\end{equation*}
The latter functor is bijective on objects and surjective on morphisms, hence exhibits the orbit equivalence relation $\#E_\Gamma^X$ as a quotient of the action groupoid; it is an isomorphism iff the action is free.
The first functor is sometimes called the \emph{cocycle associated with the action $\Gamma \actson X$}.
\end{example}

\begin{example}
\label{ex:gpd-action}
A \defn{(left) action} of an arbitrary (Borel) groupoid $(X,G)$ consists of a (Borel) space $Y$ equipped with a (Borel) map $p : Y -> X$, thought of as a ``bundle'' over $X$, as well as a map
\begin{equation*}
a : G \times_X Y := \{(g,y) \in G \times Y \mid \dom(g) = p(y)\} --> Y
\end{equation*}
taking $g : x -> x' \in G$ and $y \in p^{-1}(x)$ to $a(g,y) = g \cdot y \in p^{-1}(x')$ and obeying the usual associativity and identity laws.
A \defn{right action} is defined analogously using $Y \times_X G := \{(y,g) \mid p(y) = \cod(g)\}$.

Given a left action $a$ on $p : Y -> X$ as above, the \defn{action groupoid} has space of objects $Y$ and space of morphisms $G \ltimes Y := G \times_X Y$, with domain map $p$ and codomain map $a$ and groupoid operations as in \cref{eq:grp-action}.
We again have canonical functors
\begin{equation*}
\begin{tikzcd}
& (Y, G \ltimes Y)
    \dlar["{g \mapsfrom (g,y)}"']
    \drar["{(g,y) \mapsto (y,gy)}"] \\
(X, G) && (Y, \#E_G^Y)
\end{tikzcd}
\end{equation*}
where $\#E_G^Y$ is the connectedness relation of the groupoid $G \ltimes Y$, called the \defn{orbit equivalence relation} of the action; the latter functor to it is bijective on objects and surjective on morphisms.

The first functor $p : (Y,G \ltimes Y) -> (X,G)$ is a \defn{(discrete) fibration} of groupoids, meaning that it restricts to a bijection $\dom^{-1}(y) \cong \dom^{-1}(p(y))$ for each object $y \in Y$, i.e., each morphism in the codomain has a unique lift given any lift of its domain.
In fact, the data of a fibration over $G$ is essentially equivalent to an action of $G$, in that for any bundle $p : Y -> X$, the action groupoid construction yields a bijection between actions of $G$ on $Y$, and isomorphism classes of groupoids $H$ on $Y$ equipped with an extension of $p$ to a fibration $p : (Y,H) -> (X,G)$.
\end{example}

\subsection{Simplicial nerves}
\label{sec:simplicial}

\begin{definition}
A \defn{symmetric simplicial set} $S = (S_n)_{0 < n < \omega}$ consists of a sequence of sets $S_1, S_2, S_3, \dotsc$, equipped with, for each function $s : m -> n$ between $0 < m, n < \omega$, a map
\begin{align*}
\partial_s : S_n -> S_m
\end{align*}
which is contravariantly functorial: $\partial_{s \circ t} = \partial_t \circ \partial_s$ and $\partial_\id = \id$.
In short, $S$ is a contravariant functor from the category of positive finite ordinals (with arbitrary maps) to the category of sets.
A \defn{simplicial map} $f : S -> S'$ between symmetric simplicial sets is a family of maps $(f_n : S_n -> S'_n)_n$ commuting with the $\partial_s$, i.e., a natural transformation between functors.

We may think of each $n+1 = \{0,\dotsc,n\}$ as the vertices of a combinatorial $n$-simplex, and a function $s : m+1 -> n+1$ as a simplicial map.
For example, the inclusion $s : 2+1 `-> 3+1$:
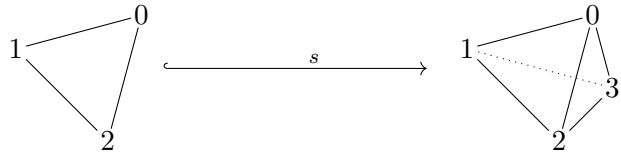
\begin{center}
\begin{tikzpicture}[every node/.append style={inner sep=1pt}]
\coordinate(m);
\node(m0) at (45:1cm) {0};
\node(m1) at (165:1cm) {1} edge (m0);
\node(m2) at (285:1cm) {2} edge (m0) edge (m1);
\begin{scope}[xshift=6cm]
\coordinate(n);
\node(n0) at (45:1cm) {0};
\node(n1) at (165:1cm) {1} edge (n0);
\node(n2) at (285:1cm) {2} edge (n0) edge (n1);
\node(n3) at (165:-1cm) {3} edge (n0) edge (n2) edge[dotted] (n1);
\end{scope}
\draw[commutative diagrams/hookrightarrow, shorten <=1cm, shorten >=1.5cm] (m) to["$\scriptstyle s$"{above,inner sep=2pt}] (n);
\end{tikzpicture}
\end{center}
Now for a symmetric simplicial set $S$, we think of $S$ as a complex obtained by gluing $n$-simplices, where each $S_{n+1}$ is the set of $n$-simplices in $S$.
We thus call $S_1$ the \defn{vertices} of $S$, $S_2$ the \defn{edges}, etc.
For $s : m+1 -> n+1$, the map $\partial_s : S_{n+1} -> S_{m+1}$ restricts each $n$-simplex along $s$ to an $m$-simplex.
For example, when $s = (0,1,2) \in (3+1)^{2+1}$ is the above inclusion, $\partial_{012} : S_{3+1} -> S_{2+1}$ takes each 3-simplex in $S$ to its face spanned by vertices $0,1,2$.
In general, when $s : n `-> n+1$ is an injection, we call $\partial_s : S_{n+1} -> S_n$ a \defn{face map}.
When $s : n+1 ->> n$ is a surjection, we call $\partial_s : S_n -> S_{n+1}$ a \defn{degeneracy map}; these provide a way to regard an $(n-1)$-simplex as a degenerate or ``flat'' $n$-simplex where one of the edges (namely between the two vertices collapsed by $s$) is really a point.
Note that every $s : m -> n$ is a composite of one-step injections, surjections, and permutations; thus the face and degeneracy maps, along with vertex permutations, determine all of the $\partial_s$.
\end{definition}

For background on (symmetric) simplicial sets, which play a fundamental role in abstract homotopy theory, see \cite{GJ}, \cite{Cisinski}.
For our purposes in this paper, simplicial sets are relevant, because on the one hand, they are precisely the structure on the (positive-arity) type spaces $\@S_n(\@T)$ of a theory required to determine it up to bi-interpretability; see \cref{rmk:types-simplicial,rmk:interp-types}.
On the other hand, a simplicial set can also encode a groupoid, via the following well-known construction:

\begin{definition}
\label{def:nerve}
For each $0 < n < \omega$, let $\#I_n = n^2$ denote the indiscrete equivalence relation on $n$.
Given a groupoid $(X,G)$, its \defn{simplicial nerve} $\@N(G)$ is the symmetric simplicial set with
\begin{align*}
\@N(G)_n := \{\text{functors } \#I_n -> G\}
\end{align*}
and $\partial_s : \@N(G)_n -> \@N(G)_m$ for $s : m -> n$ given by precomposition with $s$.
See e.g., \cite[\S1.4]{Cisinski}.

Thus, the vertices $\@N(G)_1$ are functors $\#I_1 -> G$, which are just objects $x \in X$.
The edges $\@N(G)_2$ are functors $f : \#I_2 -> G$, which are uniquely determined by the single morphism $f(0,1) \in G$.
The triangles $f : \#I_3 -> G$ are commuting triangles
\begin{equation}
\label{eq:nerve-triangle}
\begin{tikzcd}
& f(1) \drar["{f(1,2) = hg^{-1}}"] \\
f(0) \urar["{g = f(0,1)}"] \ar[rr, "{f(0,2) = h}"] && f(2)
\end{tikzcd}
\end{equation}
which are uniquely determined by a pair of morphisms $f(0,1), f(0,2)$ sharing a domain.

More generally, it is easily seen that for each $n \ge 2$, we have a bijection
\begin{align*}
(\partial_{01},\dotsc,\partial_{0n}) : \@N(G)_{n+1} &\cong (\@N(G)_2)^n_{\@N(G)_1} \cong G^n_X = \{(g_1,\dotsc,g_n) \in G^n \mid \dom(g_1) = \dotsb = \dom(g_n)\} \\
f &|-> (f(0,1), \dotsc, f(0,n))
\end{align*}
between $n$-simplices and $n$-tuples of morphisms sharing a domain ($G^n_X$ denotes the $n$-fold fiber product of $\dom : G -> X$).
A symmetric simplicial set $S$ for which $(\partial_{01},\dotsc,\partial_{0n}) : S_{n+1} -> (S_2)^n_{S_1}$ is a bijection for each $n$ is said to satisfy the \defn{Grothendieck--Segal condition}.%
\footnote{There are different versions of the Grothendieck--Segal condition, all equivalent for groupoids, where instead of $n$ morphisms sharing a domain, we may specify $f : \#I_n -> G$ by its restriction to any spanning tree of $\#I_n$.}
\end{definition}

\begin{proposition}[Grothendieck]
\label{thm:nerve-segal}
The operation $\@N$ is an equivalence of categories
\begin{align*}
\@N : \{\text{groupoids}\} \simeq \{\text{symmetric simplicial sets satisfying the Grothendieck--Segal condition}\}.
\end{align*}
\end{proposition}

The inverse-up-to-isomorphism takes a symmetric simplicial set $S$ to the groupoid with objects $S_1$, morphisms $S_2$, identity given by the degeneracy map $\partial_0 : S_2 -> S_1$ (where $0$ here denotes the constant map $0 : 2 -> 1$), and inverse and composition both given by the face map $\partial_{12} : S_2 \times_{S_1} S_2 \cong S_3 -> S_2$ (by taking $h = \id$ in the above triangle \cref{eq:nerve-triangle} to get $g^{-1}$, and then replacing $g$ with $g^{-1}$ to get $hg$).
For details, see \cite[1.4.11]{Cisinski}.

All of the above also makes sense in the Borel context, yielding an equivalence between locally countable Borel groupoids and \defn{locally countable standard Borel symmetric simplicial sets} $S$, meaning that each $S_n$ is standard Borel and each face map is Borel and countable-to-1, which satisfy the Grothendieck--Segal condition.

\subsection{Scott theories of groupoids}

\begin{definition}
\label{def:gpd-scott}
Let $(X,G)$ be a locally countable Borel groupoid.
The \defn{Scott theory} of $G$ is the countable $\@L_{\omega_1\omega}$ theory $(\@L_G,\@T_G)$ defined uniquely up to bi-interpretability via \cref{thm:mod-equiv} by declaring its models on any nonempty countable set $Y$ to be the standard Borel space
\begin{align*}
\Mod_Y(\@T_G) :=& \{u : (Y,\#I_Y) -> (X,G) \mid u \text{ is a fibration}\},
\end{align*}
where $\#I_Y$ is the indiscrete equivalence relation on $Y$; put also $\Mod_\emptyset(\@T_G) := \emptyset$.
The logic action $\Sym(Y,Z) \times \Mod_Y(\@T_G) -> \Mod_Z(\@T_G)$ is given by relabeling: for $h : Y \cong Z$ and $u \in \Mod_Y(\@T_G)$,
\begin{equation*}
h \cdot u := u \circ h^{-1}.
\end{equation*}
\end{definition}

\begin{remark}
\label{rmk:gpd-scott-torsor}
Recalling (\cref{ex:gpd-action}) that fibrations over $G$ correspond to actions, the fibrations from an indiscrete equivalence relation correspond to the free transitive $G$-actions, sometimes called \emph{principal $G$-actions} or \emph{$G$-torsors}; $\@T_G$ is the theory of all such.

For each $x \in X$, we have a canonical model $(\@H_G)_x$ of $\@T_G$ on $\dom^{-1}(x)$, given by the fibration
\begin{align*}
(\@H_G)_x := \cod_x : (\dom^{-1}(x), \#I_{\dom^{-1}(x)}) &--> (X, G) \\
\dom^{-1}(x)^2 \ni (g_0,g_1) &|--> g_1 g_0^{-1} : \cod(g_0) -> \cod(g_1).
\end{align*}
Any other $u \in \Mod_Y(\@T_G)$ is isomorphic to such a canonical model; namely, for any $y_0 \in Y$,
\begin{align*}
u_{y_0} : Y &--> \dom^{-1}(u(y_0)) \\
y &|--> u(y_0,y)
\end{align*}
is an isomorphism $u_{y_0} : u \cong (\@H_G)_{u(y_0)}$.
Moreover, all isomorphisms $h : u \cong (\@H_G)_x$ are of this form for a unique $y_0 \in Y$, namely $y_0 = h^{-1}(1_x)$; this fact is the \defn{Yoneda lemma}.
\end{remark}

Below in \cref{def:gpd-str}, we will define a ``structuring'' of a groupoid $G$, and then explain how the models $(\@H_G)_x$ together form a ``tautological $\@T_G$-structuring'' $\@H_G$, with the same universal properties as in the case of CBERs (\cref{def:scott}).

Following \cref{thm:scott-types-formulas}, we now compute the $n$-types of $\@T_G$:

\begin{proposition}
\label{thm:gpd-scott-types}
The $n$-types of $\@T_G$ are determined by the Borel isomorphism
\begin{align*}
\tp_{\@H_G}^n : G^n_X/G &\cong \@S_n(\@T_G) \\
[(x,g_0,\dotsc,g_{n-1})] &|-> \tp(\cod_x, g_0,\dotsc,g_{n-1}),
\end{align*}
where $G^n_X$ is the set of $n$-tuples in some $\dom^{-1}(x)$, equipped with the right translation $G$-action:
\begin{gather*}
G^n_X = \{(x,g_0,\dotsc,g_{n-1}) \in X \times G^n \mid x = \dom(g_0) = \dotsb = \dom(g_{n-1})\}, \\
(x,g_0,\dotsc,g_{n-1}) \cdot (g : x' -> x) = (x',g_0g,\dotsc,g_{n-1}g).
\end{gather*}
(The first coordinate $x$ is only needed for the case $n = 0$, yielding $G^0_X/G = X/G$.)
\end{proposition}
\begin{proof}
For any model $u \in \Mod_Y(\@T_G)$, i.e., fibration $u : \#I_Y -> G$, and tuple $\vec{a} \in Y^n$, we have $u_{y_0} : (u,\vec{a}) \cong (\cod_{u(y_0)},u_{y_0}(\vec{a}))$ for any $y_0 \in Y$ by the preceding remark; and all isomorphisms between such $n$-pointed models are realized by a right translation by the Yoneda lemma, hence they become identified in $G^n_X/G$.
Thus $\tp_{\@H_G}^n$ is a bijection; since $G^n_X/G, \@S_n(\@T_G)$ are quotients by Polish group actions (the former being a quotient by a CBER $\#E_G^{G^n_X}$), this implies Borel isomorphism.
\end{proof}

\begin{remark}
\label{rmk:gpd-scott-nerve}
For $n \ge 1$, we have a bijection
\begin{alignat*}{2}
\@N(G)_n = \{\text{functors } \#I_n -> G\} &\cong G^{n-1}_X &&\cong G^n_X/G \\
f &|-> (f(0),f(0,1),\dotsc,f(0,n-1)) &&|-> [(f(0),f(0,0),\dotsc,f(0,n-1))] \\
((i,j) |-> g_jg_i^{-1}) &\mathrlap{{}<---{\hphantom{\textstyle\quad\;(f(0),f(0,1),\dotsc,f(0,n-1))}}} &&\relbar\mapsfromchar [(x,g_0,\dotsc,g_{n-1})]
\end{alignat*}
where the first bijection witnesses the Grothendieck--Segal condition of the nerve (\cref{def:nerve}) and the second bijection is given by quotienting out the first factor of $G$.
It is easily seen that the simplicial boundary maps $\partial_s : \@N(G)_n -> \@N(G)_m$ of $\@N(G)$ correspond to the variable substitutions $G^n_X -> G^m_X$ for each $s : m -> n$; in other words, we have an isomorphism of symmetric simplicial sets.
Combined with the preceding result, we get a composite isomorphism
\begin{align*}
(\@N(G)_n &\cong \@S_n(\@T_G))_{n \ge 1} \\
f &|-> \tp(\cod_{f(0)},f(0,0),\dotsc,f(0,n-1)).
\end{align*}
In particular, we get that the positive-arity type spaces $(\@S_n(\@T_G))_{n \ge 1}$ of a Scott theory $\@T_G$ satisfy the Grothendieck--Segal condition.
\end{remark}

\begin{theorem}
\label{thm:gpd-scott-LN}
For a countable $\@L_{\omega_1\omega}$ theory $\@T$, the following are equivalent:
\begin{enumerate}[label=(\roman*)]
\item \label{thm:gpd-scott-LN:types}
$\partial_1 : \@S_1(\@T) -> \@S_0(\@T)$ is surjective, $\@S_1(\@T)$ is standard Borel, and for any pointed model $(\@M,a) \in \Mod_Y^1(\@T)$, the map $b |-> \tp(\@M,a,b) : Y -> \@S_2(\@T)$ is injective.
\item \label{thm:gpd-scott-LN:mod}
$\@T$ has no empty models, and for any countable set $Y$, the logic action $\Sym(Y) \actson \Mod_Y^1(\@T)$ is free and smooth.
\item \label{thm:gpd-scott-LN:logic}
$\@T$ has no empty models, and there is a countable set $\@F \subseteq \@L_{\omega_1\omega}^1$ of formulas in one variable, such that every pointed model of $\@T$ is rigid and $\@F$-categorical.
\item \label{thm:gpd-scott-LN:LN}
$\@T$ has no empty models and interprets $\@T_\LN$.
\item \label{thm:gpd-scott-LN:scott}
$\@T$ is bi-interpretable with the Scott theory $\@T_G$ of a locally countable Borel groupoid $G$, namely the groupoid whose nerve is $(\@S_n(\@T))_{n \ge 1}$.
\end{enumerate}
\end{theorem}
\begin{proof}
\cref{thm:gpd-scott-LN:types}--\cref{thm:gpd-scott-LN:LN} are equivalent by \cref{thm:interp-LN}.

\cref{thm:gpd-scott-LN:scott}$\implies$\cref{thm:gpd-scott-LN:types}:
By \cref{thm:gpd-scott-types}, $\partial_1 : \@S_1(\@T_G) -> \@S_0(\@T_G)$ is isomorphic to the projection $G^1_X/G ->> G^0_X/G$ (i.e., to $X ->> X/G$ by \cref{rmk:gpd-scott-nerve}).
By \cref{rmk:gpd-scott-nerve}, $\@S_1(\@T_G) \cong G^0_X = X$ is standard Borel.
By \cref{rmk:gpd-scott-torsor}, every pointed model of $\@T_G$ is isomorphic to some $\cod_x \in \Mod_{\dom^{-1}(x)}(\@T_G)$ equipped with some $a \in \dom^{-1}(x)$, and we may assume $a = 1_x$ by otherwise applying the isomorphism $(-)a^{-1} : \cod_x \cong \cod_{\cod(a)}$; then the model $(\cod_x,a) = (\cod_x,1_x)$ is rigid by the Yoneda lemma.

\cref{thm:gpd-scott-LN:types}$\implies$\cref{thm:gpd-scott-LN:scott}:
Note first that rigidity of pointed models implies that the symmetric simplicial set $(\@S_n(\@T))_{n \ge 1}$ obeys the Grothendieck--Segal condition.
Indeed, to check that the map
\begin{align*}
(\partial_{01},\dotsc,\partial_{0n}) : \@S_{n+1}(\@T) &--> \@S_2(\@T)^n_{\@S_1(\@T)} \\
\intertext{from \cref{def:nerve} is a bijection, since this map commutes with the projection $\partial_0$ to $\@S_1(\@T)$ in the first coordinate, it suffices to restrict both domain and codomain to those types realized in a fixed pointed model $(\@M,a) \in \Mod_Y^1(\@T)$; but since pointed models are rigid, said restriction becomes simply the canonical bijection}
\{a\} \times Y^n &--> (\{a\} \times Y)^n.
\end{align*}
Also by \cref{thm:types-smooth}, $\@S_n(\@T)$ is standard Borel for all $n \ge 1$.
Thus by \cref{thm:nerve-segal}, $(\@S_n(\@T))_{n \ge 1}$ is isomorphic to the nerve $\@N(G)$ of a locally countable Borel groupoid $(X,G)$, with objects $X = \@S_1(\@T)$ and morphisms $G = \@S_2(\@T)$.
Together with \cref{rmk:gpd-scott-nerve}, we get an isomorphism with $(\@S_n(\@T_G))_{n \ge 1}$:
\begin{equation*}
\begin{tikzcd}[isom/.append style={phantom,"\cong"{sloped}}]
\dotsb &
\@S_3(\@T) \dar[isom,"\scriptstyle\text{\labelcref{thm:nerve-segal}}\;"{left}] \rar[shift left=2] \rar \rar[shift right=2] &
\@S_2(\@T) \dar[isom] \rar[shift left] \rar[shift right] &
\@S_1(\@T) \dar[isom] \rar[two heads, "\partial_1"] &
\@S_0(\@T)
\\
\dotsb &
\@N(G)_3 \dar[isom,"\scriptstyle\text{\labelcref{rmk:gpd-scott-nerve}}\;"{left}] \rar[shift left=2] \rar \rar[shift right=2] &
\@N(G)_2 = G \dar[isom] \rar[shift left,"\dom"] \rar[shift right,"\cod"']&
\@N(G)_1 = X \dar[isom] \rar[two heads] &
X/G \dar[isom,"\scriptstyle\;\text{\labelcref{thm:gpd-scott-types}}"{right}]
\\
\dotsb &
\@S_3(\@T_G) \rar[shift left=2] \rar \rar[shift right=2] &
\@S_2(\@T_G) \rar[shift left] \rar[shift right] &
\@S_1(\@T_G) \rar[two heads] &
\@S_0(\@T_G)
\end{tikzcd}
\end{equation*}
By \cref{thm:gpd-scott-types}, we also have a canonical isomorphism $X/G \cong \@S_0(\@T_G)$ (of nonstandard Borel spaces), commuting with the two face maps $G \rightrightarrows X$ and $\@S_2(\@T_G) \rightrightarrows \@S_1(\@T_G)$, so that both $X/G, \@S_0(\@T_G)$ are the spaces of connected components of the respective simplicial sets.
But the surjection $\partial_1 : \@S_1(\@T) ->> \@S_0(\@T)$ also exhibits $\@S_0(\@T)$ as the connected components of $(\@S_n(\@T))_{n \ge 1}$, since any two $1$-types projecting to the same $0$-type are realized in isomorphic models, hence may be amalgamated into a $2$-type realized in a single model.
We thus get a Borel isomorphism of type spaces $(\@S_n(\@T) \cong \@S_n(\@T_G))_n$ for all $n < \omega$, hence a bi-interpretation $\@T \cong \@T_G$ by \cref{rmk:interp-types}.
\end{proof}

\begin{remark}
\label{rmk:gpd-scott-rigid}
If $(X,G)$ is a locally countable Borel groupoid, then $G$ is a CBER (up to isomorphism) iff all models of $\@T_G$ are rigid, since a non-identity endomorphism $g : x -> x \in G$ gives rise to a non-identity right translation automorphism of the canonical model $\cod_x$ (\cref{rmk:gpd-scott-torsor}).

Thus for a theory $\@T$ interpreting both $\@T_\LN$ and $\@T_\sep$, the groupoid $G$ constructed in the above proof must be a CBER, hence by \cref{thm:scott-LNsep} must be the same CBER $E$ constructed in that earlier proof.
Note however that we previously constructed $E$ on $X = \@S_1(\@T)$ as the kernel of $\partial_1 : \@S_1(\@T) ->> \@S_0(\@T)$, whereas the above proof instead shows that the two maps $\@S_2(\@T) \rightrightarrows \@S_1(\@T)$ yield an injection $\@S_2(\@T) `-> \@S_1(\@T)^2$ whose image is then $E$.
(For a general groupoid, said image will not be injective, but will still be $\ker(\partial_1)$, which yields the connectedness relation of the groupoid $G$.)
\end{remark}

\begin{remark}
The simplicial set of types $(\@S_n(\@T))_{n \ge 1}$ and (weakenings of) the Grothendieck--Segal condition have appeared before in the finitary ($\@L_{\omega\omega}$) model theory literature, in the context of \emph{$n$-amalgamation} of types; see \cite{GKK}, \cite{Kruckman}.

In our context, we may give a homotopy-theoretic explanation of $(\@S_n(\@T))_{n \ge 1}$ as follows, using the equivalence of 2-categories between theories and standard Borel groupoids of models from \cite{Cscc} (see the following subsection).
Given a theory $\@T$ obeying the conditions of \cref{thm:gpd-scott-LN}, let $\@T_1$ be $\@T$ expanded with a constant, so that $\Mod(\@T_1) \cong \Mod^1(\@T)$ with standard Borel $\cong$-quotient $\@S_1(\@T)$.
The inclusion $\@T `-> \@T_1$ induces an essentially surjective forgetful functor between groupoids of models $\Mod(\@T_1) ->> \Mod(\@T)$, thereby covering $\Mod(\@T)$ with the essentially discrete standard Borel groupoid $\Mod(\@T_1) \simeq \@S_1(\@T)$.
The Čech nerve of this functor (see \cite[6.1.2.11]{Lurie}) is $(\@S_n(\@T))_{n \ge 1}$, from which we recover $\Mod(\@T)$ as the realization (2-categorical quotient).
\end{remark}

\subsection{Connection to essential countability and imaginaries}
\label{sec:esscount}

In this subsection, which is tangential to the rest of the paper, we discuss connections between \cref{thm:gpd-scott-LN} above and some other notions and results that have appeared in the countable model theory literature.

\begin{definition}
\label{def:esscount}
For a countable $\@L_{\omega_1\omega}$ theory $\@T$, the isomorphism relation ${\cong} \subseteq \Mod_Y(\@T)$ is said to be \defn{essentially countable} if it is Borel and there is a Borel set $D \subseteq \Mod_Y(\@T)$ picking at least one and only countably many elements from each isomorphism class; see \cite[\S4]{Kcber}.
\end{definition}

\begin{theorem}[{Hjorth--Kechris \cite[4.3]{HK}}]
\label{thm:hjorth-kechris}
For a countable $\@L_{\omega_1\omega}$ theory $\@T$, the following are equivalent:
\begin{enumerate}[label=(\roman*)]
\item \label{thm:hjorth-kechris:iso}
For any countable set $Y$, the isomorphism relation ${\cong} \subseteq \Mod_Y(\@T)$ is essentially countable.
\item \label{thm:hjorth-kechris:logic}
There is a countable set $\@F$ of $\@L_{\omega_1\omega}$ formulas (of various arities), such that for any countable model $\@M \in \Mod_Y(\@T)$, there is a finite tuple $\vec{a} \in Y^n$ such that $(\@M,\vec{a})$ is $\@F$-categorical.
\end{enumerate}
\end{theorem}

\begin{remark}
\label{rmk:hjorth-kechris}
By the proof of \cite[4.3(iii)$\implies$(ii)]{HK}, if the above conditions hold, then in fact \cref{thm:hjorth-kechris:logic} may be ``uniformly witnessed'': there are formulas $\phi_n(x_0,\dotsc,x_{n-1}) \in \@F$ for each $n$, such that
\begin{enumerate}[label=(\alph*)]
\item  any $n$-pointed model $(\@M,\vec{a})$ of $\@T$ satisfying $\phi_n^\@M(\vec{a})$ must be $\@F$-categorical; and
\item  any model $\@M$ contains an $n$-tuple $\vec{a}$ (for some $n$) satisfying $\phi_n^\@M(\vec{a})$.
\end{enumerate}
Moreover, whenever $(\@M,\vec{a})$ satisfies (a), then so does $(\@M,\vec{b})$ for any tuple $\vec{b}$ containing all the elements of $\vec{a}$.
\end{remark}

Note the similarity between \cref{thm:hjorth-kechris} and our \cref{thm:interp-LN,thm:gpd-scott-LN}.
There are two essential differences between the two characterizations.
First, in our results we only consider single-pointed models, while \cref{thm:hjorth-kechris} considers models with some finite tuple fixed.
Second, we impose rigidity in addition to categoricity of models with respect to some countable fragment $\@F$.
We now give a detailed explanation of these two differences: the first will amount to a change of language (see \cref{rmk:esscaut}).
The second, our rigidity requirement, is more substantial:

\begin{remark}
\label{rmk:caut}
For a theory $\@T$ and $n$-pointed model $(\@M,\vec{a}) \in \Mod_Y^n(\@T)$, we have
\begin{align*}
\Aut(\@M) = \bigcup_{\vec{b} \in Y^n} \Iso((\@M,\vec{a}),(\@M,\vec{b})),
\end{align*}
where each nonempty set of isomorphisms $\Iso((\@M,\vec{a}),(\@M,\vec{a}))$ is a coset of $\Aut(\@M,\vec{a}) \subseteq \Aut(\@M)$.
Thus, the following are equivalent, for any theory $\@T$:
\begin{enumerate}[label=(\roman*)]
\item \label{rmk:caut:aut}
Every countable model of $\@T$ has only countably many automorphisms.
\item \label{rmk:caut:logic}
Every countable model $\@M \in \Mod_Y(\@T)$ becomes rigid after fixing a finite tuple $\vec{a} \in Y^n$.
\end{enumerate}
\end{remark}

\begin{lemma}
\label{thm:caut}
If the above conditions hold for a theory $\@T$, then for each $n < \omega$, the class of $n$-pointed rigid models is definable by an $\@L_{\omega_1\omega}$ formula $\phi_n(x_0,\dotsc,x_{n-1})$.
\end{lemma}
\begin{proof}
For any countable set $Y$ and $n < \omega$, the isomorphism relation ${\cong} \subseteq \Mod_Y^n(\@T)^2$ is Borel by Lusin--Novikov, since the logic action $\Sym(Y) \times \Mod_Y^n(\@T) -> \Mod_Y^n(\@T)$ is countable-to-1, since each $n$-pointed model has only countably many automorphisms.
Now the set of rigid $n$-pointed models is
\begin{align*}
\{(\@M,\vec{a}) \mid \forall b, c \in Y\, (b \ne c \implies (\@M,\vec{a},b) \not\cong (\@M,\vec{a},c))\} \subseteq \Mod_Y^n(\@T),
\end{align*}
which is Borel and obviously isomorphism-invariant, so the claim follows from Lopez-Escobar.
\end{proof}

\begin{remark}
\label{rmk:esscaut}
It follows from this lemma and \cref{rmk:hjorth-kechris} that if a theory $\@T$ has both essentially countable $\cong$ and also countable automorphism groups, i.e., the analogue of our characterization \labelcref{thm:interp-LN} of theories interpreting $\@T_\LN$, with pointed models replaced by models with some fixed finite tuple, then we may find formulas $\phi_n(x_0,\dotsc,x_{n-1})$ uniformly defining these tuples:
\begin{enumerate}[label=(\alph*)]
\item  any $n$-pointed model $(\@M,\vec{a}) \in \Mod_Y^n(\@T)$ satisfying $\phi_n^\@M(\vec{a})$ is rigid and $\@F$-categorical; and
\item  any model $\@M$ contains an $n$-tuple $\vec{a}$ satisfying $\phi_n^\@M(\vec{a})$.
\end{enumerate}
Thus, any model $\@M \in \Mod_Y^n(\@T)$ becomes rigid and $\@F$-categorical after fixing an element of (not the underlying set $Y$ but) the ``uniformly definable set''
\begin{align*}
Y \times \bigsqcup_n \phi_n^\@M =: A^\@M.
\end{align*}
\end{remark}

Note that we may recover the underlying set $Y$ from this set $A^\@M$, by quotienting out the second coordinate.
Thus, if we also lift each of the relations $R^\@M$ of $\@M$ (for $R \in \@L$) to a relation on $A^\@M$, we obtain a structure $\alpha^*(\@M)$ ``uniformly defined'' from $\@M$, from which $\@M$ can in turn be ``uniformly'' recovered up to definable isomorphism, and such that $\alpha(\@M)$ obeys the conditions of \cref{thm:interp-LN} characterizing theories interpreting $\@T_\LN$.
Note, however, that this operation $\@M |-> \alpha^*(\@M)$ is not given by an interpretation $\alpha$ in the sense we have considered in this paper (\cref{def:interp}), since $\alpha^*(\@M)$ lives not on the original underlying set $Y$, but rather on the above set $A^\@M$, which is ``uniformly defined'' from $\@M$ by the formal expression
\begin{align*}
A := \top(x) \times \bigsqcup_n \phi_n(x_0,\dotsc,x_n)
\end{align*}
(here $\top(x)$ denotes a tautology in one variable, whose interpretation in $\@M$ is the underlying set $Y$).

Recall that a (model-theoretic) \defn{imaginary sort} over a theory $\@T$ is an expression such as this one, that denotes a set uniformly definable from any model $\@M |= \@T$, but that may be external to the underlying set of $\@M$; see \cite[Ch.~7]{Hodges}.
In $\@L_{\omega_1\omega}$, the natural notion of imaginary is built from formulas by taking formal Cartesian products, countable disjoint unions, and definable quotients; see \cite[\S4]{Cscc}.
A (model-theoretic) \defn{interpretation} $\alpha : (\@L,\@T) -> (\@L',\@T')$ between two (one-sorted) $\@L_{\omega_1\omega}$ theories maps the single (base) sort of $\@T$ to an imaginary sort $A$ of $\@T'$, and each $\@L$-formula $\phi(x_0,\dotsc,x_{n-1})$ to a definable subsort of the $n$-fold Cartesian product of imaginaries $A^n$; see \cite[\S10]{Cscc} (the notion was essentially introduced in \cite{HMM}).
We will also call these \defn{imaginary interpretations}, to disambiguate from the more restrictive notion from \cref{def:interp} we have considered thus far, consisting of imaginary interpretations which fix the base sort, which we will call \defn{one-sorted interpretations} for contrast.

\begin{corollary}[of \cref{thm:hjorth-kechris}, \cref{thm:interp-LN}, and \cref{thm:gpd-scott-LN,thm:scott-LNsep}]
\label{thm:esscaut-LN}
For a countable (one-sorted) $\@L_{\omega_1\omega}$ theory $\@T$, the following are equivalent:
\begin{enumerate}[label=(\roman*)]
\item \label{thm:esscaut-LN:iso}
Every countable model of $\@T$ has only countably many automorphisms, and for any countable set $Y$, the isomorphism relation ${\cong} \subseteq \Mod_Y(\@T)^2$ is essentially countable.
\item \label{thm:esscaut-LN:logic}
There is a countable set $\@F$ of $\@L_{\omega_1\omega}$ formulas, such that for any countable model $\@M \in \Mod_Y(\@T)$, there is a finite tuple $\vec{a} \in Y^n$ such that $(\@M,\vec{a})$ is rigid and $\@F$-categorical.
\item \label{thm:esscaut-LN:formulas}
There is a countable set $\@F$ of $\@L_{\omega_1\omega}$ formulas, and $\phi_n(x_0,\dotsc,x_{n-1}) \in \@F$ for each $n < \omega$, such that
\begin{enumerate*}[label=(\alph*)]
\item  any $n$-pointed model $(\@M,\vec{a}) \in \Mod_Y^n(\@T)$ satisfying $\phi_n^\@M(\vec{a})$ is rigid and $\@F$-categorical; and
\item  any model $\@M$ contains an $n$-tuple $\vec{a}$ satisfying $\phi_n^\@M(\vec{a})$.
\end{enumerate*}
\item \label{thm:esscaut-LN:LN}
$\@T$ admits an imaginary interpretation from $\@T_\LN$, taking the single base sort of $\@T_\LN$ to an imaginary sort of $\@T$ which is nonempty in every model and from which the base sort of $\@T$ may be recovered as a definable quotient.
\item \label{thm:esscaut-LN:scott}
$\@T$ admits an imaginary bi-interpretation with the Scott theory $\@T_G$ of a locally countable Borel groupoid $G$.
\end{enumerate}
If all models of $\@T$ are rigid, then $\@T_\LN$ above may be replaced with $\@T_\LN \sqcup \@T_\sep$, and the groupoid $G$ will be a CBER.
\end{corollary}
\begin{proof}
\cref{thm:esscaut-LN:iso}$\iff$\cref{thm:esscaut-LN:logic}$\iff$\cref{thm:esscaut-LN:formulas} by \cref{thm:hjorth-kechris,rmk:caut,rmk:esscaut}.

\cref{thm:esscaut-LN:formulas}$\implies$\cref{thm:esscaut-LN:LN}:
Let the interpretation take the base sort of $\@T_\LN$ to $\top(x) \times \bigsqcup_n \phi_n(x_0,\dotsc,x_{n-1})$, as in \cref{rmk:esscaut}.

\cref{thm:esscaut-LN:LN}$\implies$\cref{thm:esscaut-LN:scott}:
Let $\alpha : \@T_\LN -> \@T$ be an imaginary interpretation, taking the base sort of $\@T_\LN$ to the imaginary $A$, which has a definable equivalence relation ${\sim} \subseteq A^2$ whose quotient is definably isomorphic to the base sort of $\@T$.
Consider the expanded language $\@L' := \@L_\LN \sqcup \@L \sqcup \{\sim\}$.
We may extend $\alpha$ to an interpretation $\alpha' : \@L' -> (\@L,\@T)$, by taking $\sim$ to the definable equivalence relation $\sim$, and $\alpha'(R) \subseteq A^n$ for each $n$-ary $R \in \@L$ to be a subsort whose interpretation in each $\@M |= \@T$ is the lift of $R^\@M \subseteq ((A/{\sim})^\@M)^n$.
Then for each $\@M |= \@T$, $\alpha'^*(\@M)$ will be an $\@L'$-structure with a definable equivalence relation $\sim$ and $\sim$-invariant relations for each $R \in \@L$, such that $\@M$ is recovered as the quotient structure $\alpha'^*(\@M)/{\sim}$; this quotient structure is in turn specified by the interpretation $\beta : \@L -> \@L'$ taking the base sort to the quotient by $\sim$.
So letting
\begin{eqalign*}
\@T' := \@T_\LN \sqcup \beta(\@T) \sqcup \{\text{``$\sim$ is an equivalence relation and each $R \in \@L$ is $\sim$-invariant''}\},
\end{eqalign*}
we have that $\alpha, \beta$ form an imaginary bi-interpretation $\@T' <-> \@T$, where $\@T'$ includes $\@T_\LN$ and has every model nonempty, hence is in turn one-sorted bi-interpretable with a Scott theory $\@T_G$ of a locally countable Borel groupoid $G$ by \cref{thm:gpd-scott-LN}.

\cref{thm:esscaut-LN:scott}$\implies$\cref{thm:esscaut-LN:iso}:
$\@T_G$ clearly obeys \cref{thm:esscaut-LN:iso} (its groupoid of models is Borel equivalent to $G$, by \cref{rmk:gpd-scott-torsor}); and \cref{thm:esscaut-LN:iso} is preserved under imaginary bi-interpretations $\@T <-> \@T_G$ (since such a bi-interpretation induces an Borel equivalence of groupoids of models; see \cite{Cscc}).

If every model of $\@T$ is rigid, then models of the new theory $\@T'$ defined above will still be rigid (since they are bi-interpretable with models of $\@T$); thus by \cref{thm:scott-LNsep} we have $\@T_\LN \sqcup \@T_\sep -> \@T'$ and so $\@T' \cong \@T_E$ for a CBER $E$ (which must be the same as $G$ by \cref{rmk:gpd-scott-rigid}).
\end{proof}

\begin{remark}
A less direct ``soft'' proof of \cref{thm:esscaut-LN:iso}$\implies$\cref{thm:esscaut-LN:scott} above may be given by taking $(X,G)$ to be the locally countable Borel subgroupoid of the groupoid of models $\bigsqcup_{N \le \omega} (\Sym(N) \ltimes \Mod_N(\@T))$ restricted to a countable complete section, by essential countability (\cref{def:esscount}).
We then have Borel equivalences of groupoids between the models of $\@T_G$, $G$, and the models of $\@T$, whence $\@T_G, \@T$ are imaginary bi-interpretable by \cite{Cscc}.

Thus, in some sense the abstract results of this paper may be regarded as a specialization of \cite{Cscc}, which concerns arbitrary $\@L_{\omega_1\omega}$ theories and their standard Borel groupoids of models, to the case of theories with countable automorphism groups and essentially countable isomorphism relations, which correspond instead to locally countable Borel groupoids.
In that case, a simpler recovery of the theory from the groupoid of models is possible, using the Scott theory, rather than the Becker--Kechris topological realization argument as in \cite{Cscc}.
\end{remark}

\subsection{Structurability}
\label{sec:gpd-str}

Let $(X,G)$ be a locally countable Borel groupoid, $\@T$ be a countable $\@L_{\omega_1\omega}$ theory.
Imitating \cref{thm:str-interp}, we would like a \defn{$\@T$-structuring} of $G$ to correspond to an interpretation $\alpha : \@T -> \@T_G$, with the \defn{tautological $\@T_G$-structuring} $\@H_G$ corresponding to $\alpha = \id_{\@T_G}$.

To reformulate this in explicit combinatorial terms, note that (by \cref{rmk:interp-types}) $\alpha$ amounts to a Borel way to put models of $\@T$ on models of $\@T_G$, which by \cref{rmk:gpd-scott-torsor} are up to isomorphism just the $G$-torsors $\dom^{-1}(x)$ for each $x \in X$, which is equivariant with respect to isomorphisms $\dom^{-1}(x) \cong \dom^{-1}(x')$, which by the Yoneda lemma (\cref{rmk:gpd-scott-torsor}) are given by right multiplication by $g : x' -> x \in G$.
We thus adopt the following (cf.\ \cref{def:str})

\begin{definition}
\label{def:gpd-str}
A \defn{$\@T$-structuring} of a locally countable Borel groupoid $(X,G)$ is a family of models $\@M = (\@M_x)_{x \in X}$, where each $\@M_x \in \Mod_{\dom^{-1}(x)}(\@T)$, which is right-$G$-equivariant:
\begin{align*}
\@M_x = \@M_{x'} \cdot (g : x -> x')
\end{align*}
(where $\@M_{x'} \cdot g$ denotes the logic action of $(-)g : \dom^{-1}(x') -> \dom^{-1}(x)$), and ``Borel'', meaning:
\begin{enumerate}[label=(\roman*)]
\item \label{def:gpd-str:enum}
For some (equivalently any) Borel family of enumerations $(h_x : \abs{\dom^{-1}(x)} \cong \dom^{-1}(x))_{x \in X}$,
the following map is Borel:
\begin{eqalign*}
X &--> \bigsqcup_{N \le \omega} \Mod_N(\@T) \\
x &|--> h_x^{-1} \cdot \@M_x.
\end{eqalign*}
\item \label{def:gpd-str:param}
For any countable set $Y$, standard Borel space $Z$, Borel map $f : Z -> X$, and Borel family of bijections $(h_z : Y \cong \dom^{-1}(f(z)))_{z \in Z}$, the following map is Borel:
\begin{eqalign*}
Z &--> \Mod_Y(\@T) \\
z &|--> h_z^{-1} \cdot \@M_{f(z)}.
\end{eqalign*}
\item \label{def:gpd-str:fiber}
For each $n$-ary relation symbol $R \in \@L$ (or more generally formula $\phi$),
\begin{align*}
\~R^\@M :=
\set[\big]{(x,g_0,\dotsc,g_{n-1}) \in G^n_X}{R^{\@M_x}(g_0,\dotsc,g_{n-1})}
\end{align*}
is a Borel subset of $G^n_X$ (recall \cref{thm:gpd-scott-types}).
\item \label{def:gpd-str:global}
For each $n$-ary $R \in \@L$ (assuming $\@L$ has no nullary relation symbols), or more generally formula $\phi$ in at least one variable,
\begin{align*}
R^\@M :=
\set[\big]{\text{functors } f : \#I_n -> G}{R^{\@M_{f(0)}}(f(0,0), \dotsc, f(0,n-1))}
\end{align*}
is a Borel subset of $G^{\#I_n}$ (recall \cref{rmk:gpd-scott-nerve}).
\end{enumerate}
Let $\Mod_G(\@T)$ denote the set of $\@T$-structurings of $G$.
If $\Mod_G(\@T) \ne \emptyset$, we call $G$ \defn{$\@T$-structurable}.
\end{definition}

It is easily seen that the above Borelness conditions are equivalent, following \cref{thm:str}.

\begin{definition}[cf.\ \cref{def:scott}]
The \defn{tautological $\@T_G$-structuring} $\@H_G$ of $G$ is given by $(\@H_G)_x := \cod_x : \dom^{-1}(x) -> G$ as in \cref{rmk:gpd-scott-torsor}.
\end{definition}

\begin{definition}[cf.\ \cref{def:str-classbij}]
\label{def:gpd-str-fib}
For a Borel fibration $f : (X,G) -> (Y,H)$ between locally countable Borel groupoids, and a $\@T$-structuring $\@M$ of $H$, the \defn{pullback $\@T$-structuring} $f^{-1} \cdot \@M$ of $G$ is defined by
\begin{align*}
(f^{-1} \cdot \@M)_x := (f|\dom^{-1}(x) : \dom^{-1}(x) \cong \dom^{-1}(f(x)))^{-1} \cdot \@M_{f(x)}
    \quad \text{for each $x \in X$}.
\end{align*}
\end{definition}

\begin{definition}[cf.\ \cref{def:str-interp}]
\label{def:gpd-str-interp}
For an interpretation $\alpha : \@T -> \@T'$ between theories, and a $\@T'$-structuring $\@M$ of $G$, the \defn{$\alpha$-reduct} $\alpha^* \@M$ is the $\@T$-structuring given by
\begin{align*}
(\alpha^* \@M)_x := \alpha^*(\@M_x)
    \quad \text{for each $x \in X$}.
\end{align*}
\end{definition}

Following the discussion preceding \cref{def:gpd-str} and the analogous results in \cref{sec:scott}, we now easily have:

\begin{corollary}[cf.\ \cref{thm:str-interp}]
\label{thm:gpd-str-interp}
For any locally countable Borel groupoid $(X,G)$ and theory $(\@L,\@T)$, we have a bijection
\begin{align*}
\{\text{interpretations } \@T -> \@T_G\} &\cong \Mod_G(\@T) = \{\text{$\@T$-structurings of $G$}\} \\
\alpha &|-> \alpha^*(\@H_G).
\qed
\end{align*}
\end{corollary}

\begin{proposition}[cf.\ \cref{thm:str-classbij}]
\label{thm:gpd-str-fib}
For any locally countable Borel groupoids $(X,G), (Y,H)$, we have a bijection
\begin{align*}
\{\text{fibrations } G -> H\} &\cong \Mod_G(\@T_H) = \{\text{$\@T_H$-structurings of $G$}\} \\
f &|-> f^{-1} \cdot \@H_H.
\qed
\end{align*}
\end{proposition}

\begin{corollary}[of \cref{thm:gpd-scott-LN}; cf.\ \cref{thm:scott-full-faithful}]
\label{thm:gpd-scott-equiv}
We have a dual equivalence of categories
\begin{align*}
\{\text{locally countable Borel groupoids, fibrations}\} &\simeq \{\text{$\@L_{\omega_1\omega}$ theories obeying \labelcref{thm:gpd-scott-LN}, interpretations}\}
\end{align*}
taking a groupoid $G$ to its Scott theory $\@T_G$, and a fibration $f : G -> H$ to the unique $\alpha : \@T_H -> \@T_G$ such that $\alpha^*(\@H_G) = f^{-1} \cdot \@H_H$, or such that $\alpha^*_n : \@S_n(\@T_G) \cong G^n_X/G -> H^n_Y/H \cong \@S_n(\@T_H)$ (via \labelcref{thm:gpd-scott-types}) is the map induced by $f$.
\qed
\end{corollary}

\begin{corollary}[cf.\ \cref{thm:str-impl-interp}]
\label{thm:gpd-str-impl-interp}
For two $\@L_{\omega_1\omega}$ theories $\@T, \@T'$, the following are equivalent:
\begin{enumerate}[label=(\roman*)]
\item
Every $\@T$-structurable locally countable Borel groupoid is $\@T'$-structurable.
\item
There exists an interpretation $\@T' -> \@T \sqcup \@T_\LN$.
\end{enumerate}
More generally, expandability of every $\@T$-structuring along an interpretation $\alpha : \@T -> \@T'$ is equivalent to factorizability of $\alpha$ through the inclusion $\@T -> \@T \sqcup \@T_\LN$ (cf.\ \cref{thm:str-expan-interp}).
\qed
\end{corollary}

\subsection{Some examples}

We conclude with some simple examples of structurings of groupoids.
There is undoubtably much of the general theory that remains to be developed (e.g., much of \cite{CK} for groupoids); our goal here is just to demonstrate the utility of the concept, by showing how it relates to some familiar notions.

\begin{example}
\label{ex:gpd-str-freegrp}
If $G$ is a locally countable Borel groupoid on one object $X = 1 = \{*\}$, i.e., a countable group, then a $\@T$-structuring $\@M$ of $G$ is just a single model $\@M_{*} \in \Mod_G(\@T)$ which is invariant under right translation.

For example, if $\@T$ is the theory of directed trees, i.e., connected acyclic graphs with a chosen orientation of each edge, then $G$ is $\@T$-structurable (``treeable'') iff it is a free group, in which case a $\@T$-structuring $\@M$ amounts to a choice of free generating set, namely all those $g \in G$ for which there is an edge $1 -> g$ in $\@M$, with $\@M$ then given by the Cayley graph.
(If $\@T$ is the theory of undirected trees, then a free product of copies of $\#Z_2$ would also be $\@T$-structurable.)

More generally, if $\@T$ is the theory of connected graphs, then a $\@T$-structuring of a countable group is just a Cayley graph.
\end{example}

\begin{example}
\label{ex:gpd-str-tree}
For the theory of trees $\@T$ as above, we may also consider ``treeings'' $\@M$ of a general locally countable Borel groupoid $(X,G)$.
By definition, this consists of a Borel right-translation-invariant family $(\@M_x)_{x \in X}$ of trees on each fiber $\dom^{-1}(x)$ of $\dom : G -> X$.

We may represent this in more concrete (but less structurally transparent) terms, by interpreting the edge relation $E$ as the set of morphisms $E^\@M \subseteq G^1_X = G$ from \cref{def:gpd-str}\cref{def:gpd-str:global} and \cref{rmk:gpd-scott-nerve}; the correspondence with the fiberwise relations $E^{\@M_x} \subseteq \dom^{-1}(x)^2$ is given by
\begin{eqalign*}
E^{\@M_x}(g : x -> y,\, h : x -> z)
\iff  E^{\@M_y}(1_y : y -> y,\, hg^{-1} : y -> z)
\iff  E^\@M(hg^{-1}).
\end{eqalign*}
To say that the tree $\@M_x$ on each fiber is connected means that $1_x \in \dom^{-1}(x)$ is joined by a path
\begin{eqalign*}
1_x = g_0 \mathrel{(E^{\@M_x})^{\pm1}} g_1 \mathrel{(E^{\@M_x})^{\pm1}} \dotsb \mathrel{(E^{\@M_x})^{\pm1}} g_n = g
\end{eqalign*}
to any $g \in \dom^{-1}(x)$, which means that we have morphisms $h_i = (g_{i+1} g_i^{-1})^{\pm1}$ with
\begin{eqalign*}
g = (g_n g_{n-1}^{-1}) (g_{n-1} g_{n-2}^{-1}) \dotsm (g_1 g_0^{-1}) = h_{n-1}^{\pm1} h_{n-2}^{\pm1} \dotsm h_0^{\pm1}
    &&\text{where } h_0, \dotsc, h_{n-1} \in E^\@M;
\end{eqalign*}
in other words, $E^\@M \subseteq G$ generates the groupoid $G$.
To say that each $\@M_x$ is moreover acyclic means that each $g \in G$ may be so written in a unique way, i.e., $G$ is the \defn{free groupoid} (or \emph{path groupoid}, or \emph{fundamental groupoid}) generated by the directed multigraph $E^\@M$.
See e.g., \cite{Alvarez}, \cite{Carderi}.

Likewise, a ``graphing'' (structuring by connected graphs) of $G$ just amounts to any subgraph generating $G$, not necessarily freely.
\end{example}

\begin{remark}
\label{rmk:gpd-str-tree}
In the literature (see e.g., \cite[\S30]{KM}), the term \defn{treeable group} usually refers to not just a group structurable by trees as above, i.e., a free group, but more generally a group $G$ admitting a probability-measure-preserving free action generating a treeable CBER.

This is equivalent to the existence of a free groupoid $H$ with an invariant measure and a fibration $H -> G$, by \cref{ex:gpd-action} and the coinduced action (Bernoulli shift) construction (see \cref{ex:str-gpd-freeact}).
\end{remark}

\begin{example}
\label{ex:gpd-str-FM}
Every locally countable Borel groupoid $G$ is $\@T_\LN$-structurable, by \cref{thm:gpd-scott-LN}.
For example, if $G$ is a countable group, a sequence of Lusin--Novikov functions is given by left multiplication $g_n (-)$ by each group element $g_n \in G$.

However, not every countable group $G$ is $\@T_\FMinv$-structurable, e.g., $\#Z$, by the argument in \cref{ex:Ztrans-cex}.
In other words, we do not have a ``Feldman--Moore theorem for group(oid)s'', in the form of a translation-invariant transitive action of $\#Z_2^{*\omega}$.

Nonetheless, by \cref{thm:FMbij}, we \emph{do} have a ``Feldman--Moore theorem for groupoids'', in the form of a translation-invariant transitive action of $\#F_\omega$.
That is, for a locally countable Borel groupoid $(X,G)$, we have a Borel action $\#F_\omega \actson G$, whose orbits are the $\dom$-fibers, such that
\begin{align*}
\gamma \cdot gh = (\gamma \cdot g)h
    \quad \text{for $g : x -> y$ and $h : x' -> x$}.
\end{align*}
As in \cref{ex:gpd-str-tree}, such a family of actions on the $\dom$-fibers may by right-invariance be represented more concretely by just the action on the identities $1_x$, yielding for each $\gamma \in \#F_\omega$ a Borel family of morphisms $\gamma \cdot \id_X := \{\gamma \cdot 1_x\}_{x \in X} \subseteq G$ in each $\dom$-fiber.
In fact it is easily seen that each such $\gamma \cdot \id_X \subseteq G$ also picks one morphism in each $\cod$-fiber, i.e., is a \defn{(total) Borel bisection} of $G$; and that $\delta\gamma \cdot \id_X = (\delta \cdot \id_X)(\gamma \cdot \id_X)$, so that we have a group homomorphism $\gamma |-> \gamma \cdot \id_X$ from $\#F_\omega$ into the group of Borel bisections of $G$, known as the \defn{(Borel) full group} of $G$.
Transitivity of the fiberwise $\#F_\omega$-actions means that $G = \bigcup_{\gamma \in \#F_\omega} (\gamma \cdot \id_X)$.
Thus, the ``Feldman--Moore theorem for groupoids'' says that \emph{every locally countable Borel groupoid can be covered by a countable subgroup of its full group}; in this form, the result was stated (without proof) in \cite[4.1]{TW}.
\end{example}

\begin{example}
\label{ex:str-gpd-freeact}
For a fixed locally countable Borel groupoid $G$, a structuring by its Scott theory $\@T_G$ of another groupoid $H$ is by \cref{thm:gpd-str-fib} a fibration $H -> G$, which is equivalently by \cref{ex:gpd-action} a representation of $H$ as the action groupoid of a Borel action of $G$.
If $H$ is a CBER, then the action is free.
Thus by \cref{thm:einfT-LNsep}, the theory $\@T_G \sqcup \@T_\LN \sqcup \@T_\sep$ is the Scott theory of the \defn{invariantly universal CBER $\#E_{\infty G}$ generated by a free Borel action of $G$}.

For example, if $G$ is a countable group, it is well-known that such a universal free Borel $G$-action may be constructed as the free part of the shift action on $(2^\#N)^G$, giving an explicit realization of $\#E_{\infty  G}$.
A ``fiberwise'' version of this works for arbitrary $G$, by taking the fibration $\#E_{\infty G} -> G$ whose fiber over each $x \in X$ is $(2^\#N)^{\dom^{-1}(x)}$, equipped with a natural Borel structure and action.%
\footnote{This is the \emph{coinduced action} or \emph{right Kan extension} of the universal bundle $X \times 2^\#N -> X$ along the inclusion $(X,=) -> (X,G)$.}
\end{example}

\begin{remark}
It is not true that more generally, for a locally countable Borel groupoid $G$, $\@T_G \sqcup \@T_\LN$ is the Scott theory of the invariantly universal locally countable Borel groupoid equipped with a Borel fibration to $G$ (i.e., the invariantly universal action groupoid of a $G$-action).

More generally, it is not true that (analogously to \cref{thm:einfT-LNsep}) for a theory $\@T$, $\@T \sqcup \@T_\LN$ is the Scott theory of the invariantly universal $\@T$-structurable locally countable Borel groupoid.
This would follow as in \cref{thm:einfT-LNsep} if $\@T_\LN$ were the Scott theory of the invariantly universal locally countable Borel groupoid $\#G_\infty$.
But this is false, for the rather trivial reason that the groupoid whose Scott theory is $\@T_\LN$ only has a single trivial connected component (since there is only a single model of $\@T_\LN$ on a singleton set), whereas $\#G_\infty$ must have $2^{\aleph_0}$ such components.

(Note that $\#G_\infty$ exists, by \cref{thm:einfT-LNsep}, since we may encode a locally countable Borel groupoid $(X,G)$ as a structuring by a suitable theory of the connectedness CBER on the underlying space $G$.
It is possible to explicitly construct a Scott theory for it, by combining $\@T_\LN$ with a $\@T_\sep$-like theory applied to pairs.
Such a theory would be mutually interpretable with $\@T_\LN$, and is arguably a more canonical representative for a ``theory characterizing locally countable Borel groupoids'' in results such as \cref{thm:gpd-scott-LN}; however, it seems subjectively less clean than $\@T_\LN$.)
\end{remark}

\def\MR#1{}
\bibliographystyle{amsalpha}
\bibliography{refs}

\medskip
\noindent
Department of Mathematics, Statistics, and Computer Science\\
University of Illinois Chicago\\
Chicago, IL, USA\\
\nolinkurl{rbane8@uic.edu}

\medskip
\noindent
Department of Mathematics\\
University of Michigan\\
Ann Arbor, MI, USA\\
\nolinkurl{ruiyuan@umich.edu}

\end{document}